\documentclass[10pt,english]{smfart}

\usepackage[T1]{fontenc}
\IfFileExists{french.ldf}{\usepackage[french,english]{babel}}{%
  \usepackage[english]{babel}\babelprovide[import]{french}}
\usepackage{latexsym,amscd,color}
\usepackage{amsmath,amsfonts,amssymb,mathrsfs}
\usepackage{enumerate,euscript}
\usepackage{amssymb,url,xspace,smfthm}
\IfFileExists{smfhyperref.sty}{\usepackage{hyperref}}{}
\input xy
\xyoption{all}

\renewcommand{\a}{\mathfrak a}

\DeclareMathOperator{\gal}{Gal}

\newcommand{\BibTeX}{{\scshape Bib}\kern-.08em\TeX}
\DeclareMathOperator{\Gr}{Gr}
\newcommand{\T}{\S\kern .15em\relax }
\newcommand{\AMS}{$\mathcal{A}$\kern-.1667em\lower.5ex\hbox
        {$\mathcal{M}$}\kern-.125em$\mathcal{S}$}
\newcommand{\resp}{\textit{resp}.\xspace}

\DeclareMathOperator{\im}{Im}
\DeclareMathOperator{\proj}{Proj}

\DeclareMathOperator{\spm}{Spm}

\DeclareMathOperator{\rg}{rk}

\DeclareMathOperator{\spec}{Spec}

\renewcommand{\P}{\mathbb{P}}
\newcommand{\wmu}{\widehat{\mu}}
\newcommand{\C}{\mathbb{C}}
\newcommand{\Q}{\mathbb{Q}}

\newcommand{\adeg}{\widehat{\deg}}

\newcommand{\p}{\mathfrak{p}}
\DeclareMathOperator{\sym}{Sym}

\newcommand{\E}{\overline{E}}
\newcommand{\F}{\overline{F}}
\newcommand{\sE}{\mathcal{E}}
\newcommand{\G}{\overline{G}}

\renewcommand{\O}{\mathcal{O}}
\renewcommand{\L}{\overline{L}}

\newcommand{\f}{\mathbb{F}}
\newcommand{\ndot}{\raisebox{.4ex}{.}}
\DeclareMathOperator{\Hilb}{Hilb}

\title{Conics and rational points of bounded height on cubic surfaces}
\alttitle{Coniques et points rationnels de hauteur born\'ee sur les surfaces cubiques}

\date{}
\author{Chunhui Liu}
\address{Institute for Advanced Study in Mathematics\\
Harbin Institute of Technology\\150001 Harbin\\P. R. China}
\email{chunhui.liu@hit.edu.cn}

\begin{document}
\def\smfbyname{}
\begin{abstract}
We consider geometrically integral cubic surfaces over a fixed number field that are not ruled by lines. We bound the number of rational and integral points of bounded height on conic components of auxiliary intersections. The projective estimate combines a global determinant method retaining dependence on the height of the Cayley form with uniform bounds for conics of bounded height. The geometry of residual pencils and the Deligne pairing yield cubic image curves in fixed coefficient spaces. If the section at infinity is geometrically integral, we give explicit uniform estimates for integral points on any finite collection of smooth conics. The number of conics tangent to infinity is uniformly bounded. Each remaining conic contributes a subpolynomial number of integral points, uniformly in the conic and its containing plane. Applied to auxiliary intersections, these estimates give any height exponent greater than the reciprocal of the square root of three for the affine conic contribution.
\end{abstract}
\begin{altabstract}
On considère des surfaces cubiques géométriquement intègres non réglées par des droites sur un corps de nombres fixé. On majore le nombre de points rationnels et entiers de hauteur bornée sur les composantes coniques d'intersections auxiliaires. L'estimation projective combine une méthode globale des déterminants tenant compte des hauteurs de la forme de Cayley avec le dénombrement uniforme des coniques de hauteur bornée. Les pinceaux résiduels et l'accouplement de Deligne donnent des courbes images cubiques dans des espaces de coefficients fixes. Si la section à l'infini est géométriquement intègre, on donne des bornes uniformes explicites pour les points entiers sur toute collection finie de coniques lisses. Les coniques tangentes à l'infini sont en nombre uniformément borné. Les autres ont un nombre sous-polynomial de points entiers, uniformément en la conique et son plan. Pour la contribution conique affine à l'intersection auxiliaire, on obtient tout exposant supérieur à l'inverse de la racine carrée de trois.
\end{altabstract}

\maketitle

\tableofcontents
\section{Introduction}
In this article, we study the contribution of conics to uniform counts of rational and integral points on cubic surfaces over a number field. We combine estimates for individual curves with bounds for the number of conics of bounded height. In the affine case, we also prove a stronger estimate for arbitrary finite collections of smooth conics, under a geometric hypothesis at infinity.

We first fix the height convention. Let $K$ be a number field and $X$ a closed subscheme of $\mathbb P^n_K$. For a rational point $\xi=[x_0:\cdots:x_n]\in X(K)$ via an embedding into $\mathbb P^n_K$, its \textit{height} is
\[H_K(\xi)=\prod_{v\in M_K}\max\limits_{0\leqslant i\leqslant n}\{|x_i|_v^{[K_v:\Q_v]}\}.\]
Let
\[S(X;B)=\{\xi\in X(K)\mid H_K(\xi)\leqslant B\},\]
where the embedding is understood, and the absolute values are normalized as in \S \ref{chap.2}. The Northcott property implies that $S(X;B)$ is finite for every fixed $B$. Uniform point counting seeks upper bounds that remain valid as $X$ varies, with its ambient space, dimension and degree fixed.
\subsection{History of determinant methods}
The \textit{determinant method} constructs auxiliary hypersurfaces through points of bounded height without containing the variety under study. Over $\mathbb Q$, divisibility properties of evaluation determinants at finite primes are compared with upper bounds coming from the heights of the points. The resulting vanishing gives algebraic relations and reduces the counting problem to subvarieties of smaller dimension.
\subsubsection{}
Bombieri and Pila \cite{Bombieri_Pila} introduced a determinant argument for integral points on plane affine curves. Pila \cite{Pila95} developed bounds for integral and rational points on varieties. Heath-Brown \cite{Heath-Brown} extended the method to rational points on curves and surfaces.

\subsubsection{}
A central motivation is the \textit{dimension growth conjecture} formulated by Heath-Brown \cite{Heath-Brown}. For a geometrically integral projective variety $X$ over $\mathbb Q$ of dimension $d\geqslant2$ and degree $\delta\geqslant2$, it predicts
\[\#S(X;B)\ll_{d,\delta,\epsilon}B^{d+\epsilon}\]
for every $\epsilon>0$. Serre \cite{Serre1997} considered a non-uniform version. The conjecture has motivated many refinements of the determinant method, whose extension to number fields was developed by Broberg \cite{Broberg04}.

Further refinements include the work of Browning, Ellenberg, Heath-Brown, Venkatesh and Salberger in \cite{Browning_Heath05,Browning_Heath06I,Browning_Heath06II,Ellenberg-Venkatesh2005,Bro_HeathB_Salb,Salberger07}. In particular, Salberger's global determinant method \cite{Salberger_preprint2013} proves the dimension growth conjecture in degree at least $4$, while the case of degree $2$ was proved by Heath-Brown \cite{Heath-Brown} by quadratic forms. 

Walsh \cite{Walsh_2015} and Castryck--Cluckers--Dittmann--Nguyen \cite{CCDN2020} obtained further improvements. Paredes--Sasyk \cite{ParedesSasyk2022} proved uniform bounds over global fields.
\subsubsection{}
Salberger \cite{Salberger2008} studied the contribution of conics using heights on their Hilbert schemes, obtaining uniform estimates for points on surfaces outside their lines. This approach was developed further in \cite{Salberger2015,Salberger_preprint2013}.

Embedding the relevant Hilbert schemes into Grassmannians relates the heights of conics to the Pl\"ucker heights of their parameter points. Bounds for the number of conics of bounded height can then be combined with estimates for points on individual conics. The geometry and arithmetic of the family thus enter the sum of the individual point estimates.
\subsubsection{}
H.~Chen \cite{Chen1,Chen2} reformulated the work of Salberger \cite{Salberger07} using Arakelov geometry. In this framework, evaluation maps between Hermitian vector bundles and the slope method of Bost \cite{BostBour96} encode the determinant argument. This formulation gives quantitative estimates over number fields.

Following the strategy of \cite{Chen1,Chen2}, the author reformulated the global determinant method of \cite{Salberger_preprint2013,CCDN2020} using Arakelov geometry in \cite{Liu2022d}.
\subsection{Uniform bounds for conic contributions}
The affine approach to the uniform dimension growth conjecture in degree three requires bounds for integral points on cubic surfaces that are uniform in the defining equation; see \cite{CDHNV2023}. The determinant method produces auxiliary intersections whose number of conic components may grow with the height bound. We estimate their contribution using the geometry of conic pencils and uniform bounds for individual conics.

\subsubsection{}
For the projective estimate, Theorem \ref{global determinant of hypersurface in a linear subvariety} retains the dependence on the height of the Cayley form for hypersurfaces in varying linear subvarieties. Its constants are independent of the containing linear space. Its proof follows the argument of \cite[Theorem 5.4]{Liu2022d}, using the arithmetic Hilbert--Samuel lower bound of \cite{Liu2024b}.
\subsubsection{}
The geometry of the conic pencils and the height-sensitive estimates yield the following result, proved in Theorem \ref{estimate of rational points in conics of cubic}.

Here "non-ruled" means not swept out by lines in the fixed embedding. In the two theorems below, a conic contribution refers to the points lying on the specified geometrically integral conic components of an auxiliary intersection.
\begin{theo}[Projective conic contribution]
Let $X\subset\mathbb P^3_K$ be a geometrically integral non-ruled cubic surface. For $B\geqslant2$ and every $\epsilon>0$, there is an auxiliary hypersurface of degree $\ll_{K,\epsilon}B^{\sqrt3/2+\epsilon/(6\sqrt3)}$ (cf.\ Corollary \ref{global determinant of hypersurface in a linear subvariety without height term}) containing $S(X;B)$ and not containing $X$, such that the points lying on its smooth $K$-conic intersection components number
\[
 \ll_{K,\epsilon}B^{1+5\sqrt3/16+\epsilon}.
\]
An explicit coefficient is given in Theorem \ref{estimate of rational points in conics of cubic}.
\end{theo}
\subsubsection{}
Let $\mathbb P^3_K=\proj\left(K[T_0,T_1,T_2,T_3]\right)$. For the affine application let $H_\infty=V(T_0)$ and $X'=X\setminus H_\infty$.

Write $s(Y;B)$ for integral points whose Euclidean norm at every embedding into $\mathbb C$ is at most $B^{1/[K:\mathbb Q]}$, as in \S \ref{affine open subset of P(E)}.

Let $X$ be a geometrically integral non-ruled cubic surface in $\mathbb P^3_K$ such that $X\cap H_\infty$ is geometrically integral. For every finite collection $\mathcal C$ of distinct smooth $K$-conics on $X$, Theorem \ref{integral points in finite conic families} gives
\[
 \sum_{C\in\mathcal C}\#s(C\smallsetminus H_\infty;B)
 \ll_{K,\epsilon}\#\mathcal C\,B^\epsilon+B^{1/2+\epsilon}
\]
for every $B\geqslant2$ and $\epsilon>0$. Explicit coefficients depending only on $K$ and $\epsilon$ are given in that theorem. The collection $\mathcal C$ need not arise from an auxiliary intersection.

At most $162$ elements of $\mathcal C$ are tangent to $H_\infty$, and their total contribution is $O_{K,\epsilon}(B^{1/2+\epsilon})$. Each remaining conic contributes $O_{K,\epsilon}(B^\epsilon)$, uniformly in the conic and its containing plane, extending \cite[Lemma 13]{Bro_HeathB_Salb} to $K$.

The exponent $1/2$ in the second term is necessary even when both $X$ and its section at infinity are smooth; see Remark \ref{necessity of square-root term} for such an example.

Combining this estimate with the affine auxiliary covering gives the following result, proved in Corollary \ref{sharper affine conic contribution}.
\begin{theo}[Affine conic contribution]\label{intro:affine conic contribution}
Let $X\subset\mathbb P^3_K$ be a geometrically integral non-ruled cubic surface such that $X\cap H_\infty$ is geometrically integral, and put $X'=X\setminus H_\infty$. For $B\geqslant2$ and every $\epsilon>0$, there is an affine auxiliary hypersurface of degree $\ll_{K,\epsilon}B^{1/\sqrt3+\epsilon/(6\sqrt3)}$ (cf.\ Corollary \ref{integral points in affine surface}) containing $s(X';B)$ and not containing $X'$, such that the points lying on its geometrically integral conic intersection components number
\[
 \ll_{K,\epsilon}B^{1/\sqrt3+\epsilon}.
\]
An explicit coefficient for this bound is given in Corollary \ref{sharper affine conic contribution}.
\end{theo}
\subsubsection{}
The auxiliary coverings are supplied by Sections \ref{projective coverings} and \ref{chap. 4}. Over $\mathbb Q$, the projective exponent agrees with the conic contribution in \cite[Lemma 5.3, the proof of Theorem 6.3]{Salberger_preprint2013}.

Over $\mathbb Q$, \cite[Proposition 1]{Bro_HeathB_Salb} bounds the integral points on the union of any finite collection $\mathcal C$ of geometrically integral conics on $X$ by
\[
 \ll_\epsilon B^\epsilon\max\{B^{2/3},B^{1/2},\#\mathcal C\}
\]
under the same hypothesis at infinity.

For the auxiliary intersection above, this gives $O_\epsilon(B^{2/3+\epsilon})$. Theorem \ref{intro:affine conic contribution} gives the smaller exponent $1/\sqrt3$ over every fixed number field.

\subsubsection{}
For the non-ruled surfaces considered here, the geometric lines form a uniformly bounded set, and their integral points contribute \(O_K(B)\). A degree $2$ component that is irreducible over \(K\) but geometrically reducible has at most one $K$-rational point.

Thus, under the hypotheses of Theorem \ref{intro:affine conic contribution}, an $O_{K,\epsilon}(B^{1+\epsilon})$ bound for all integral points reduces to controlling the components of degree at least three in the same auxiliary covering. Since their degrees may grow with $B$, estimates uniform in these degrees remain necessary.
\subsection{The role of Chow forms and Cayley forms}
Arithmetic intersection theory \cite{Gillet_Soule-IHES90} expresses the height of a projective variety as an arithmetic intersection number. Bost, Gillet and Soul\'e \cite{BGS94} compare this height with heights defined by Chow forms. Proposition \ref{compare arakelov height and chow form height} gives the comparison used here for fixed Cayley representatives; see also \cite[Proposition 3.7]{Liu-reduced}. This comparison connects the arithmetic estimates with the coefficient heights used in the projective and affine determinant bounds.

\subsubsection{}
Using the functoriality of the Deligne pairing \cite[\S 4]{YuanZhang2023}, we identify the Chow line bundle on each residual conic pencil with \(\mathcal O_{\mathbb P^1}(3)\). The Cayley coefficients define a geometrically integral cubic image curve in a fixed coefficient space. Uniform point estimates on these curves bound the number of conics of bounded height, with constants independent of the surface and pencil.

Counting directly on these image curves avoids the comparison constants depending on the surface and pencil that arise from standard parameter heights. This gives a formulation of Salberger's height argument \cite[Lemma 4.1, \S 5.1 -- \S 5.3]{Salberger_preprint2013} using a fixed Chow height in place of heights associated with Hilbert embeddings of increasing order.

The individual conic estimate retains an inverse fourth-root dependence on the Cayley height. This allows us to sum the point estimates using the height distribution of the conic family.

\subsubsection{}
An analogous argument with the leading part of the Cayley form gives Theorem \ref{estimate of integral points in conics of cubic}. This part records the intersection cycle at infinity. Geometric integrality of \(X\cap H_\infty\) ensures that it is nonzero throughout each pencil; its coefficients again define a cubic image curve.

\subsubsection{}
The height-distribution estimates concern all smooth $K$-conics on the surface, independently of an auxiliary intersection. Corollary \ref{number of conics with height smaller than B} bounds the number of such conics of Cayley height at most $B$ by $O_{K,\epsilon}(B^{2/3+\epsilon})$. Under the hypothesis at infinity, Corollary \ref{number of conics 2-height smaller than B} gives the same exponent for the height of the leading part.

Corollary \ref{integral points in affine curve} supplies a complementary estimate for integral points on a plane curve, uniform in its containing plane and retaining the leading height of the Cayley form.
\subsection{Organization of the article}
Section \ref{chap.2} constructs the evaluation determinants, and Section \ref{Chap. Cayley form} studies the geometric and arithmetic properties of Chow and Cayley forms. Section \ref{projective coverings} establishes the projective covering estimates. Section \ref{chap. 4} proves their affine counterparts. Section \ref{chap.5} recalls the required geometry of cubic surfaces and Deligne pairings. Section \ref{chap.6} proves the height-distribution estimates and the resulting bounds for rational and integral points on conics. Section \ref{sharper affine conic contribution section} proves the stronger affine bound through estimates for arbitrary finite collections of smooth conics.
\subsection*{Acknowledgement}
I am deeply grateful to Prof. Per Salberger for introducing me to his brilliant work \cite{Salberger_preprint2013} and explaining key aspects of it, especially for highlighting the important role of height functions on Hilbert schemes in \cite{Salberger2008,Salberger_preprint2013}. I wish to thank Prof. Huayi Chen and Prof. Xiaowen Hu for insightful discussions on key technical points and for their guidance on the literature, Prof. Junyi Xie for suggesting several useful ideas to overcome some technical obstructions, and Prof. Xinyi Yuan for reminding me of the key ingredients in his brilliant collaborative work \cite{YuanZhang2023} with Prof. Shou-Wu Zhang. Parts of this work were completed during visits to the Beijing International Center for Mathematical Research at Peking University and the Institute for Theoretical Sciences at Westlake University. I thank both institutions for their hospitality and support.

\section{Construction of determinants}\label{chap.2}
In this section, we build a particular version of the determinant method with the formulation of Arakelov geometry. This is an immediate  generalization of the construction in \cite{Liu2022d}, where we will apply the uniform lower bound of the arithmetic Hilbert--Samuel function in \cite{Liu2024b}. 
\subsection{Height of rational points}
We begin by recalling some foundations of algebraic number theory and defining the height function. We refer the readers to \cite[Part B]{Hindry} for a systematic introduction to this subject. 
\subsubsection{}
Let $K$ be a number field, and $\O_K$ be the ring of integers of $K$. We denote by $M_{K,\infty}$ the set of infinite places of $K$, by $M_{K,f}$ the set of finite places of $K$, and by $M_K=M_{K,f}\sqcup M_{K,\infty}$ the set of places of $K$.

For each $v\in M_K$ and $x\in K$, we define the absolute value \[|x|_v=\left|N_{K_v/\mathbb Q_v}(x)\right|_v^{1/[K_v:\mathbb Q_v]},\] which extends the usual absolute values on $\mathbb Q_p$ or $\mathbb R$. Here, $\mathbb Q_v$ denotes the usual $p$-adic field $\mathbb Q_p$ if $v\in M_{K,f}$ lies above the prime number $p$ under the extension $K/\mathbb Q$, and $\mathbb Q_v=\mathbb R$ if $v\in M_{K,\infty}$.

In this setting, the product formula (cf. \cite[Chap. III, (1.3) Proposition]{Neukirch})
\[\prod_{v\in M_K}|x|_v^{[K_v:\mathbb Q_v]}=1\]
holds for all $x\in K^{\times}$.
\subsubsection{}
Let $\xi=[\xi_0:\cdots:\xi_n]\in\mathbb P^n_K(K)$. We define the \textit{height} of $\xi$ in $\mathbb P^n_K$ as
\begin{equation}\label{definition of naive heigth of rational point}
H_K(\xi)=\prod_{v\in M_K}\max_{i\in\{0,\ldots,n\}}\left\{|\xi_i|_v\right\}^{[K_v:\mathbb Q_v]}.
\end{equation}
We also define the \textit{logarithmic height} of $\xi$ as
\begin{equation}\label{log naive height}
h(\xi)=\frac{\log H_K(\xi)}{[K:\mathbb Q]},\end{equation}
which is invariant under the extensions of $K$ (cf. \cite[Lemma B.2.1]{Hindry}).
\subsubsection{}
Let $X$ be a subscheme of $\mathbb P^n_K$, and $\phi:X\hookrightarrow\mathbb P^n_K$ be the embedding morphism. For a $\xi\in X(K)$, we define $H_K(\xi)=H_K(\phi(\xi))$, and we will omit the embedding $\phi$ if there is no confusion. We denote
\[S(X;B)=\left\{\xi\in X(K)\mid H_K(\xi)\leqslant B\right\}.\]
By the Northcott property (cf. \cite[Theorem B.2.3]{Hindry}), the cardinality $\#S(X;B)$ is finite for a fixed $B\in\mathbb R$.
\subsection{Foundations of Arakelov geometry}
In this part, we introduce some useful notions for formulating the subject within Arakelov geometry. 
\subsubsection{}
A \textit{normed vector bundle} on $\spec\O_K$ is a pair $\overline E=\left(E,(\|\ndot\|_v)_{v\in M_{K,\infty}}\right)$, where:
\begin{itemize}
\item $E$ is a projective $\O_K$-module of finite rank;
\item $(\|\ndot\|_v)_{v\in M_{K,\infty}}$ is a family of norms on $E_v=E\otimes_{\O_K,v}\mathbb C$ for each $v\in M_{K,\infty}$, which is invariant under the action of $\gal(\mathbb C/K_v)$.
\end{itemize}
Each complex place in $M_{K,\infty}$ is counted once, with weight $[K_v:\mathbb Q_v]=2$. One may equivalently sum without weights over all embeddings into $\mathbb C$.

By the \textit{rank} of a normed vector bundle $\E$, we mean the rank of $E$, denoted by $\rg_{\O_K}(E)$ or simply $\rg(E)$ if there is no ambiguity.

If for every $v\in M_{K,\infty}$, $\|\ndot\|_v$ is Hermitian, which means it is induced by an inner product, we say that $\overline E$ is a \textit{Hermitian vector bundle} on $\spec\O_K$. If $\rg_{\O_K}(E)=1$, we say that $\E$ is a \textit{Hermitian line bundle}.
\subsubsection{}
Let $\E=(E,(\|\ndot\|_{E,v})_{v\in M_{K,\infty}})$ and $\F=(F,(\|\ndot\|_{F,v})_{v\in M_{K,\infty}})$ be two normed vector bundles on $\spec\O_K$. We say that $\F$ is a normed subbundle of $\E$ if $F$ is a saturated sub-$\O_K$-module of $E$ and $\|\ndot\|_{F,v}$ is the restriction of $\|\ndot\|_{E,v}$ to $F$ for each $v\in M_{K,\infty}$.

We say that $\G=(G,(\|\ndot\|_{G,v})_{v\in M_{K,\infty}})$ is a normed quotient bundle of $\E$ on $\spec\O_K$, if $G$ is a projective quotient $\O_K$-module of $E$, and $\|\ndot\|_{G,v}$ is a quotient norm of $\|\ndot\|_{E,v}$ for all $v\in M_{K,\infty}$.

We write $E_K=E\otimes_{\O_K}K$ for simplicity.
\subsubsection{}
Let $\E$ be a Hermitian vector bundle on $\spec\O_K$, and $s_1,\ldots,s_r$ be a $K$-basis of $E_K$ contained in $E$. The \textit{Arakelov degree} of $\E$ is defined as
\begin{eqnarray*}
\adeg(\E)&=&-\sum_{v\in M_K}[K_v:\Q_v]\log\|s_1\wedge\cdots\wedge s_r\|_v\\
&=&\log\left(\#E/(\O_Ks_1+\cdots+\O_Ks_r)\right)\\
& &-\frac{1}{2}\sum_{v\in M_{K,\infty}}[K_v:\Q_v]\log\det(\langle s_i,s_j\rangle_{v,1\leqslant i,j\leqslant r}),
\end{eqnarray*}
where $\|s_1\wedge\cdots\wedge s_r\|_v$ follows the definition \cite[\S 2.1.3]{Chen10b} for $v\in M_{K,\infty}$, and $\langle s_i,s_j\rangle_{v,1\leqslant i,j\leqslant r}$ is the Gram matrix of $\{s_1,\ldots,s_r\}$ with respect to $v\in M_{K,\infty}$.

For $v\in M_{K,f}$, we define $\|\ndot\|_v$ as the norm given by the model. More precisely, let $\overline E$ be a normed vector bundle on $\spec\O_K$, and $x\in E\otimes_{\O_K}K_v$. We define the norm given by the model with respect to $v\in M_{K,f}$ as
\[\|x\|_v=\inf\{|a|_v\mid a\in K_v^\times,\;a^{-1}x\in E\otimes_{\O_K}\widehat{\O}_{K,v}\}.\]
If $x\in E_K$, we have 
\[\|x\|_v=\inf\{|a|_v\mid a\in K^\times,\;a^{-1}x\in E\}.\]

We refer the readers to \cite[\S 2.1.3]{BGS94} for the proof of the above equalities. The Arakelov degree is independent of the choice of the basis $\{s_1,\ldots,s_r\}$ by the product formula.

We define \[\adeg_{\mathrm n}(\E)=\frac{\adeg(\E)}{[K:\mathbb Q]}\] as the \textit{normalized Arakelov degree} of $\E$, which is independent of the choice of the base field $K$.
\subsubsection{}
Let $\E$ be a non-zero Hermitian vector bundle on $\spec\O_K$. The \textit{slope} of $\E$ is defined as
\[\wmu(\E)=\frac{\adeg_{\mathrm n}(\E)}{\rg(E)}.\]
In addition, we denote by $\wmu_{\max}(\E)$ the maximal value of slopes of all non-zero Hermitian subbundles, and by $\wmu_{\min}(\E)$ the minimal value of slopes of all non-zero Hermitian quotient bundles.

\subsection{Height functions with the formulation of Arakelov geometry}
In this part we will formulate the height function of a closed subscheme in a projective space by Arakelov geometry.
\subsubsection{}
Let $\overline{\sE}$ be a Hermitian vector bundle on $\spec\O_K$ of rank $n+1$, and $\mathbb P(\sE)$ be the projective bundle which represents the functor from the category of commutative $\O_K$-algebras to the category of sets sending each $\O_K$-algebra $A$ to the set of isomorphism classes of projective quotient $A$-modules of $\sE\otimes_{\O_K}A$ of rank $1$.

For a $D\in\mathbb N$, let $\O_{\mathbb P(\sE)}(1)$ (or $\O(1)$ if there is no ambiguity) be the universal bundle on $\mathbb P(\sE)$, and $\O_{\mathbb P(\sE)}(D)$ or $\O(D)$ be the tensor product $\O_{\mathbb P(\sE)}(1)^{\otimes D}$. For each $v\in M_{K,\infty}$, the Hermitian norm on $\overline{\sE}$ induces the corresponding Fubini--Study metric on $\O(1)$, which defines the Hermitian line bundle $\overline{\O_{\mathbb P(\sE)}(1)}$ (or $\overline{\O(1)}$) on $\mathbb P(\sE)$.
\subsubsection{}
Let $P\in\mathbb P(\sE_K)(K)$, and $\mathcal P\in \mathbb P(\sE)(\O_K)$ be the Zariski closure of $P$ in $\mathbb P(\sE)$. For a Hermitian line bundle $\overline{\mathcal L}$ on $\mathbb P(\sE)$, we define the \textit{Arakelov height} of $P$ with respect to $\overline{\mathcal L}$ as
\begin{equation}\label{general arakelov height of point}
h_{\overline{\mathcal L}}(P)=\adeg_{\mathrm n}\left(\mathcal P^*\overline{\mathcal L}\right).
\end{equation}
\subsubsection{}\label{O{n+1} with l2-norm}
Here we consider a special case. Let
\[\overline{\sE}=\left(\O_K^{\oplus(n+1)},(\|\ndot\|_v)_{v\in M_{K,\infty}}\right),\]
where
\[\begin{array}{rrcl}
\|\ndot\|_v:&K^{\oplus(n+1)}&\longrightarrow&\mathbb R\\
&(z_0,\ldots,z_n)&\mapsto&\sqrt{|v(z_0)|^2+\cdots+|v(z_n)|^2}
\end{array}\]
for every $v\in M_{K,\infty}$. In this case, if $P=[x_0:\cdots:x_n]\in \mathbb P(\sE_K)$, then we have (cf. \cite[Proposition 9.10]{Moriwaki-book})
\begin{eqnarray}\label{arakelov height of point}
h_{\overline{\O(1)}}(P)&=&\sum_{v\in M_{K,f}}\frac{[K_v:\Q_v]}{[K:\Q]}\log\left(\max_{0\leqslant i\leqslant n}\{|x_i|_v\}\right)\\
& &+\frac{1}{2}\sum_{v\in M_{K,\infty}}\frac{[K_v:\Q_v]}{[K:\Q]}\log\left(|v(x_0)|^2+\cdots+|v(x_n)|^2\right).\nonumber
\end{eqnarray}

Let $h(\ndot)$ be the height defined at \eqref{log naive height} on $\mathbb P^n_K$, and $\overline{\sE}$ on $\spec\O_K$ be the same as above. Then for an arbitrary $P\in \mathbb P(\sE_K)(K)$, we have
\[\left|h(P)-h_{\overline{\O(1)}}(P)\right|\leqslant\frac{1}{2}\log(n+1)\]
by an elementary calculation. This estimate ensures that it is reasonable to involve the Arakelov geometry in the study of counting rational points of bounded height.
\subsubsection{}\label{def of Arakelov height}
Let $\overline{\mathcal L}$ be a Hermitian line bundle on $\mathbb P(\sE)$ over $\spec\O_K$. Let $X$ be a closed subscheme of $\mathbb P(\sE_K)$ of pure dimension $d$, and $\mathscr X$ be its Zariski closure in $\mathbb P(\sE)$. The \textit{Arakelov height} of $X$ is defined as the arithmetic intersection number
\[h_{\overline{\mathcal L}}(X)=h_{\overline{\mathcal L}}(\mathscr X)=\adeg_{\mathrm n}\left(\widehat{c}_1(\overline{\mathcal L})^{d+1}\cdot[\mathscr X]\right),\]
where $\widehat{c}_1(\overline{\mathcal L})$ is the first arithmetic Chern class of $\overline{\mathcal L}$. We refer the readers to \cite[Chap. III.4, Proposition 1]{Soule92} for its precise definition. 

When $X$ is a $K$-rational point, $h_{\overline{\mathcal L}}(X)$ is just the height defined at \eqref{general arakelov height of point}.
\subsubsection{}
We consider a more special case for the height of varieties. Let $X$ be a hypersurface in $\mathbb P(\sE_K)$ of degree $\delta$. By \cite[Chap. 1, Proposition 7.6 (d)]{GTM52}, $X$ is defined by a homogeneous polynomial of degree $\delta$. We suppose that the above $X$ is defined by the polynomial
\[f(T_0,\ldots,T_n)=\sum_{i_0+\cdots+i_n=\delta}a_{i_0\ldots i_n}T_0^{i_0}\cdots T_n^{i_n}, \]
and then we define the \textit{naive height} of $X$ as
\begin{equation}\label{naive height of a polynomial or hypersurface}
h(X)=h(f)=\sum_{v\in M_K}\frac{[K_v:\Q_v]}{[K:\mathbb Q]}\log\max_{i_0+\cdots+i_n=\delta}\{|a_{i_0\ldots i_n}|_v\}.\end{equation}
We write $H_K(f)=\exp([K:\mathbb Q]h(f))$ for the multiplicative height of the coefficient vector of $f$.

In addition, if $\overline{\sE}$ is defined as above, then we have
\begin{equation}\label{comparison of naive height and Arakelov height}
\begin{aligned}
-\frac{1}{2}\left(\log((n+1)(\delta+1))+\delta\sum_{j=1}^{n}\mathcal H_j\right)
&\leqslant h(X)-h_{\overline{\O(1)}}(X)\\
&\leqslant(n+1)\delta\log2+4\delta\log(n+1)\\
&\qquad-\frac{1}{2}\delta\sum_{j=1}^{n}\mathcal H_j,
\end{aligned}
\end{equation}
where $\mathcal H_n=1+\cdots+\frac{1}{n}$. These inequalities follow from the hypersurface calculation in the proof of Proposition \ref{compare arakelov height and chow form height} below.
\subsection{Construction of determinants}
In this part, we provide the foundations of Arakelov geometry to build the determinant method, whose strategy originates in \cite{Chen1,Chen2} and was applied in \cite{Liu2022d}. 
\subsubsection{}\label{definition of John norm on E_D}
Let $\overline{\sE}$ be the same Hermitian vector bundle on $\spec\O_K$ of rank $n+1$. For every $D\in\mathbb N$, we write $E_D=H^0(\mathbb P(\sE),\O(D))$. For each $v\in M_{K,\infty}$, we define the norm $\|\ndot\|_{v,\sup}$ on $E_D\otimes_{\O_K,v}\mathbb C$ as
\begin{equation}\forall\; s\in E_{D,v},\ \|s\|_{v,\sup}=\sup_{x\in\mathbb P(\sE_{K_v})(\C)}\|s(x)\|_{v,\mathrm{FS}},\end{equation}
where $\|\ndot\|_{v,\mathrm{FS}}$ is the Fubini--Study metric on $\O(D)$ with respect to the place $v\in M_{K,\infty}$. 

Now we recall the following fact (cf. \cite[Theorem 3.3.6]{Thompson96}): for every symmetric convex body $C$, there exists a unique ellipsoid of maximal volume among the ellipsoids contained in $C$. This ellipsoid is called the \textit{ellipsoid of John} of $C$.

For the $\O_K$-module $E_D$ and a place $v\in M_{K,\infty}$, we take the ellipsoid of John of the unit closed ball with respect to $\|\ndot\|_{v,\sup}$. This ellipsoid of John induces a Hermitian norm denoted by $\|\ndot\|_{v,J}$. In fact, we have
\[\|s\|_{v,\sup}\leqslant\|s\|_{v,J}\leqslant\sqrt{\rg_{\O_K}(E_D)}\|s\|_{v,\sup}\]
for every $s\in E_D$.
\subsubsection{}
Let $X$ be a closed integral subscheme of $\mathbb P(\sE_K)$, and $\mathscr X$ be its Zariski closure in $\mathbb P(\sE)$. We define
\begin{equation}\label{evaluation map of eta_D}
\eta_{D,K}:\;E_{D,K}\longrightarrow H^0\left(X,\O_{\mathbb P(\sE_K)}(1)|_X^{\otimes D}\right)
\end{equation}
as the \textit{evaluation map} induced by the closed immersion of $X$ in $\mathbb P(\sE_K)$.

We denote by $F_{D}$ the largest saturated sub-$\O_K$-module of $H^0\left(\mathscr X,\O_{\mathbb P(\sE)}(1)|_{\mathscr X}^{\otimes D}\right)$ such that $F_{D,K}=\im(\eta_{D,K})$.
\begin{rema}
We keep all the notations above. When $D$ is large enough, we have $F_D=H^0\left(\mathscr X,\O_{\mathbb P(\sE)}(1)|_{\mathscr X}^{\otimes D}\right)$.
\end{rema}
We denote by $\F_{D,J}$ the Hermitian vector bundle on $\spec\O_K$, which is equipped with the quotient norm of $\|\ndot\|_{v,J}$ on $E_D$ defined in \S \ref{definition of John norm on E_D} for each $v\in M_{K,\infty}$.
\subsubsection{}
We consider the function
\[\begin{array}{rcl}
\mathbb N^+&\longrightarrow&\mathbb R\\
D&\mapsto&\adeg_{\mathrm n}(\F_{D,J})\text{ or }\wmu(\F_{D,J}),
\end{array}\]
and we define it as the \textit{arithmetic Hilbert--Samuel function} of $X$ with respect to the line bundle $\overline{\O_{\mathbb P(\sE)}(1)}|_{\mathscr X}$.

In \cite[Theorem 5.7]{Liu2024b}, we studied the uniform lower bound of $\wmu(\F_{D,J})$ for the closed subschemes of $\mathbb P(\sE_K)$ satisfying certain conditions. More precisely, let $\overline{\sE}$ be a Hermitian vector bundle of rank $n+1$ on $\spec\O_K$ (not necessarily the same as in \S \ref{O{n+1} with l2-norm}), and $X$ be a closed subscheme of $\mathbb P(\sE_K)$ of pure dimension $d\geqslant1$ and degree $\delta$, which is a hypersurface in a linear subvariety of $\mathbb P(\sE_K)$. In this case, we have
\[
\begin{aligned}
\frac{\wmu(\F_{D,J})}{D}\geqslant{}&\frac{r(d+1,D-\delta)}{Dr_1(d+1,D)}h_{\overline{\O(1)}}(X)\\
&+\left(1-\frac{(d+1)\delta r(d+1,D-\delta)}{Dr_1(d+1,D)}\right)\wmu_{\min}(\overline{\sE})+B_1(d),
\end{aligned}
\]
when $D\geqslant\delta$, where $r(n,D)=\rg(E_D)={n+D\choose n}$, $r_1(d+1,D)=\rg(F_D)={d+1+D\choose d+1}-{d+1+D-\delta\choose d+1}$, the constant
\[
B_1(d)=B_0(d+1)-\frac{1}{2}-\frac{1}{2(d+1)}-\frac{\log(d+2)}{2(d+1)}-\frac{\mathcal H_{d+1}}{2(d+1)}
\]
depends only on $d$, and
\[
\begin{aligned}
B_0(n)
={}&
-\frac{\mathcal H_n+n}{2}
-\frac12\log(n+1)
\\
&\quad
-2^{n-1}(n-1)!
\Bigg[
\frac{(n+1)\mathcal H_n-n}{n!}
+\frac{1}{2(n-1)!}\left(\frac{n}{e}+1\right)
\\
&\qquad\qquad\qquad\qquad
+\frac{n}{2n!}
\left|
(2-n^2-2n)\mathcal H_n
+n^2+3n-5
+(2-n)\log(2\pi)
\right|
\\
&\qquad\qquad\qquad\qquad
+2\left(
\frac{25}{12}
+\frac{7}{12e}
+\frac{7n(n+1)(2n+1)}{12}
-\frac{7}{2}
\right)
\Bigg].
\end{aligned}
\]

We choose $\overline{\sE}$ as that defined in \S \ref{O{n+1} with l2-norm}. In this case, we have
\begin{equation}\label{lower bound of arithmetic Hilbert-Samuel}
\frac{\wmu(\F_{D,J})}{D}\geqslant\frac{r(d+1,D-\delta)}{Dr_1(d+1,D)}h_{\overline{\O(1)}}(X)+B_1(d)
\end{equation}
with the same notations as above when $D\geqslant\delta$. See \cite[(25)]{Liu2024b} for more details of this calculation.
\section{Geometry and arithmetics of Chow forms and Cayley forms}\label{Chap. Cayley form}
For a projective variety of pure dimension $d$ embedded in $\mathbb P^n$, its Chow divisor parameterizes the $(n-d-1)$-dimensional linear spaces that intersect it non-trivially. By definition, the points representing these linear spaces form a subset of the corresponding Grassmannian.

For a pure-dimensional projective variety, this divisor has class equal to the degree of the variety times the Pl\"ucker hyperplane class.

Following \cite{Chen1}, a hypersurface of this type defined using Stiefel coordinates is called a Chow form, while one defined using Pl\"ucker coordinates is called a Cayley form. In this section, we will follow the notations in \cite{Chen1,Chen2} originating in \cite[\S 4]{BGS94}, and see also \cite[\S 2]{Liu-reduced}. 

In this section, we provide a geometric construction of these divisors using Pl\"ucker coordinates, working over an arbitrary field. We then study their arithmetic properties.
\subsection{Geometric construction}
First, we give the geometric construction of Chow forms and Cayley forms. This part is inspired by \cite[\S 3]{Chen1} and has been applied in \cite{Liu-reduced}.

\subsubsection{}\label{construction of Cayley variety}
Let $k$ be a field, $V$ be a vector space of finite dimension over $k$, and $V^\vee$ be the dual space of $V$. Let $0\leqslant d<\dim_k V-1$. We denote by $\check G=\Gr(d+1,V^\vee)$ the Grassmannian which classifies the quotients of dimension $d+1$ of $V^\vee$. By the Pl\"ucker embedding
\[\Gr(d+1,V^\vee)\hookrightarrow\mathbb P\left(\bigwedge\nolimits^{d+1}V^\vee\right),\]
the coordinate algebra $B(\check G)=\bigoplus\limits_{D\geqslant0}B_D(\check G)$ of $\check G$ is a homogeneous quotient algebra of $\bigoplus\limits_{D\geqslant0}\sym^D\left(\bigwedge^{d+1}V^\vee\right)$.

Let $\widetilde\Gamma\subset\mathbb P(V)\times\check G$ be the incidence variety of points and linear subspaces containing them, and let $p$ and $q$ be its projections to $\mathbb P(V)$ and $\check G$ respectively. The incidence construction is made on the Grassmannian of decomposable Pl\"ucker coordinates.
\subsubsection{}
We have the following property on the structure of the Cayley form (see \cite[Chap.~3, \S 2.B]{Gelfandal94} and \cite[Proposition 2.2, Remarque 2.4]{Liu-reduced}). 
\begin{prop}\label{structure of cayley form}
Let $X\subset\mathbb P(V)$ be of pure dimension $d$, with fundamental cycle $[X]=\sum_i m_i[X_i]$ and degree $\delta$. Then $q_*p^*[X]$ is the Chow divisor on $\check G$. Its associated line bundle is $\O_{\check G}(\delta)$, and it is the sum of the distinct Chow divisors of $X_i$ with multiplicities $m_i$.
\end{prop}
Thus the Cayley form $\psi_X$ is an element of the degree $\delta$ part $B_\delta(\check G)$ of the Pl\"ucker coordinate ring, determined up to scalar, and $\psi_X=\prod_i\psi_{X_i}^{m_i}$ in this ring.
\begin{defi}\label{definition of Cayley variety and Cayley form}
The divisor $q_*p^*[X]$ on $\check G$ is the Cayley divisor, and a section defining it is the Cayley form. A polynomial in the ambient Pl\"ucker coordinates representing this section is determined only modulo the Pl\"ucker ideal.
\end{defi}
We interpret Cayley forms in the homogeneous coordinate ring of the Grassmannian. For coefficient heights, fix a standard monomial basis over $\mathbb Z$ in each degree and the corresponding polynomial representative. The basis is homogeneous for the weight given by the number of Pl\"ucker factors containing the index $0$. Thus the decompositions and diagonal substitutions below respect these representatives. All comparisons with arithmetic heights use these fixed choices. In Section 7 we specify the resulting convention for conics in $\mathbb P^3$.

\subsection{A description by resultant}
Let $k$ be a field, and $V$ be a vector space of dimension $n+1$ over $k$. When $X$ is a complete intersection closed subscheme of $\mathbb P(V)$, its Cayley form can also be described by a particular resultant. We will only introduce the necessary foundations of this notion, rather than giving a self-contained treatment. We refer the readers to \cite{Gelfandal94} for a systematic introduction on this subject.
\subsubsection{}\label{definition of resultant}
Let $T_0,\ldots,T_n$ be a basis of $H^0\left(\mathbb P(V),\O_{\mathbb P(V)}(1)\right)$, $\delta_i\in\mathbb N_{\geqslant1}$, and $g_i\in H^0\left(\mathbb P(V),\O_{\mathbb P(V)}(\delta_i)\right)$ be a family of non-zero sections, where $i=0,\ldots,n$. Then the \textit{resultant} of $g_0,\ldots,g_n$ denoted by 
\[\operatorname{Res}(g_0,\ldots,g_n),\]
is a polynomial whose variables are coefficients of $g_0,\ldots,g_n$, where $g_0,\ldots,g_n$ are considered as polynomials of variables in $T_0,\ldots,T_n$. To ensure the uniqueness of the resultant, we require the normalization
\[\operatorname{Res}(T_0^{\delta_0},\ldots,T_n^{\delta_n})=1.\]
We refer the readers to \cite[Chap. 13]{Gelfandal94} for a detailed introduction to this notion. 

By \cite[Chap. 13, Proposition 1.1]{Gelfandal94}, $\operatorname{Res}(g_0,\ldots,g_n)$ is a homogeneous polynomial in the coefficients of each $g_i$ of degree $\left(\prod\limits_{j=0}^n\delta_j\right)/ \delta_i$, where $i=0,\ldots,n$. 

\subsubsection{}\label{Cayley form of complete intersection}
Let $X$ be a complete intersection closed subscheme of $\mathbb P(V)$, which is defined by $f_i\in H^0\left(\mathbb P(V),\O_{\mathbb P(V)}(\delta_i)\right)$ with $i=1,\ldots,m$, where $1\leqslant m\leqslant n$. In this case, $X$ is of dimension $n-m$, and of degree $\delta_1\cdots\delta_m$. 

Let $h_1,\ldots,h_{n-m+1}\in H^0\left(\mathbb P(V),\O_{\mathbb P(V)}(1)\right)$. By \S \ref{definition of resultant}, $f_1,\ldots,f_m,h_1,\ldots,h_{n-m+1}$ have common zeros in $\mathbb P(V)$ if and only if 
\[\operatorname{Res}(f_1,\ldots,f_m,h_1,\ldots,h_{n-m+1})=0.\]
In other words, if we suppose that $h_1,\ldots,h_{n-m+1}$ define a complete intersection linear subvariety in $\mathbb P(V)$ of dimension $m-1$, then the above equality is valid if and only if this linear subvariety has nonempty intersection with $X$. 

This fact gives an explicit method to define the Cayley form of $X$. Let $s_0,\ldots,s_n$ be the dual basis of $T_0,\ldots,T_n$, and then $(s_{i_0}\wedge\cdots\wedge s_{i_{n-m}})_{0\leqslant i_0<\cdots<i_{n-m}\leqslant n}$ are the Pl\"ucker coordinates of the linear varieties of dimension $m-1$ in $\mathbb P(V)$. By \S \ref{definition of resultant}, $\operatorname{Res}(f_1,\ldots,f_m,h_1,\ldots,h_{n-m+1})$ is a homogeneous polynomial of degree $\delta_1\cdots\delta_m$ of the variables $(s_{i_0}\wedge\cdots\wedge s_{i_{n-m}})_{0\leqslant i_0<\cdots<i_{n-m}\leqslant n}$, which is an example of Proposition \ref{structure of cayley form}. 
\subsubsection{}
We consider a special case. Let $X$ be a hypersurface of degree $\delta$ in $\mathbb P(V)$. By \cite{GTM52}, $X$ is defined by an element $f\in\sym^\delta\left(V\right)$. With respect to the basis $T_0,\ldots,T_n$ of $H^0\left(\mathbb P(V),\O_{\mathbb P(V)}(1)\right)$, $X$ is defined by the homogeneous polynomial $f(T_0,\ldots,T_n)$ of degree $\delta$. 

Let $s_0,\ldots,s_n$ be the dual basis of $T_0,\ldots,T_n$. By the argument in \S \ref{Cayley form of complete intersection}, the Cayley form $\psi_X\in\sym^\delta\left(\bigwedge^nV^\vee\right)$ by Proposition \ref{structure of cayley form} is determined by the homogeneous equation
\[f\left(s_1\wedge\cdots\wedge s_n,-s_0\wedge s_2\wedge\cdots\wedge s_n,\ldots,(-1)^ns_0\wedge\cdots\wedge s_{n-1}\right),\]
where the successive wedge products omit $s_0,\ldots,s_n$, respectively. In a sense, it coincides with the original homogeneous polynomial $f(T_0,\ldots,T_n)$.
\subsection{Change of coordinates}
In this part, we consider the change of the Cayley form under a particular endomorphism of projective spaces.  
\subsubsection{}
Let $V$ be a vector space of dimension $n+1$ over a field $k$. The following result shows that the construction of the Cayley divisor in Definition \ref{definition of Cayley variety and Cayley form} commutes with automorphisms of $\mathbb P(V)$, which is inspired by the proof of \cite[Proposition 2.7]{Liu-reduced} with the idea originating in \cite[Lemma 4.3.1]{BGS94}.
\begin{prop}\label{Cayley form commmutes with automorphism}
Let $X$ be a closed subscheme of $\mathbb P(V)$ of pure dimension $d$, and $F:\mathbb P(V)\rightarrow\mathbb P(V)$ be an automorphism. Then the construction of the Cayley divisor of $X$ in \S \ref{construction of Cayley variety} is commutative with the base change induced by $F$.
\end{prop}
\begin{proof}
With all the notations in \S \ref{construction of Cayley variety}, we have the following commutative diagram
\[\xymatrix{\relax \mathbb P(V)\ar[d]^F&\widetilde{\Gamma}\ar[l]^p\ar[r]_q\ar[d]^{r_1}&\check G\ar[d]^{r_2}\\\mathbb P(V)&\widetilde{\Gamma}\ar[l]^p\ar[r]_q&\check G,}\]
where $r_1$ is obtained from the base change, and $r_2$ is induced from $r_1$. 

Since $F$ is an automorphism, it is flat; consequently, $r_1$ and $r_2$ are also flat. By definition, we have $p\circ r_1=F\circ p$, so we have
\[r_1^*(p^*[X])=p^*(F^*[X])\]
by \cite[Lemma 1.7.1]{Fulton} combined with \cite[\S 20.1]{Fulton}, where $[X]$ denotes the fundamental cycle of $X$.

For simplicity, set $[Y]=p^*[X]$. 
Since the morphism $q$ is proper, it follows from \cite[Proposition 1.7]{Fulton} and \cite[\S 20.1]{Fulton} that \[r_2^*(q_*[Y])=q_*(r_1^*[Y]).\]
We thus obtain
\[q_*(p^*(F^*[X]))=r_2^*(q_*(p^*[X])),\]
completing the proof.
\end{proof}
\subsubsection{}\label{change of coordinate: multiple H}
Fix homogeneous coordinates $T_0,\ldots,T_n$ and let $s_0,\ldots,s_n$ be the dual basis. For $H\in k^\times$, define
\[
\begin{aligned}
 F_H([x_0:\cdots:x_n])&=[x_0:Hx_1:\cdots:Hx_n],\\
 F_H^{-1}([x_0:\cdots:x_n])&=[Hx_0:x_1:\cdots:x_n].
\end{aligned}
\]
In particular, $[1:x_1:\cdots:x_n]\in X$ if and only if
$[H:x_1:\cdots:x_n]\in F_H^{-1}(X)$.
\subsubsection{}\label{change of coordinate: multiple by H}
Let $X$ have pure dimension $d$ and degree $\delta$. Write
\[
 \psi_X=\sum_{i=0}^\delta\psi_{X,i},
\]
where $\psi_{X,i}$ has total degree $i$ in the Pl\"ucker variables containing the index $0$. By Proposition \ref{Cayley form commmutes with automorphism}, a Cayley form of $F_H^{-1}(X)$ is
\[
 \psi'_X=\sum_{i=0}^\delta H^i\psi_{X,i}.
\]
All equalities are in the coordinate ring of the Grassmannian, with the fixed weight-homogeneous representatives used for coefficient heights.
\subsubsection{}\label{change of coordinate: translation by a}
Let $X$ be a closed subscheme of $\mathbb P(V)$ of pure dimension $d$ and degree $\delta$. We denote by $T_0,\ldots,T_n$ and $s_0,\ldots,s_n$ the same bases as those in \S \ref{change of coordinate: multiple by H}. Let $\mathbf a=(a_1,\ldots,a_n)\in k^n$, and we consider the morphism
\[\begin{array}{rrcl}
T_{\mathbf a}:&\mathbb P(V)&\longrightarrow&\mathbb P(V)\\
&[T_0:T_1:\cdots:T_n]&\mapsto&[T_0:T_1+a_1T_0:\cdots:T_n+a_nT_0],
\end{array}\]
and we denote $X'=T_{\mathbf a}(X)$. 

Let $\psi_X\in\sym^\delta\left(\bigwedge^{d+1}V^\vee\right)$ be the fixed polynomial representative of the Cayley form of $X$ on $\check G$. By Proposition \ref{Cayley form commmutes with automorphism}, we may choose $\psi_{X'}$ by substituting $s_0+a_1s_1+\cdots+a_ns_n$ for $s_0$ in $\psi_X$. Every Pl\"ucker variable containing the index $0$ is then replaced by itself plus a linear combination of variables not containing that index, while all remaining variables are unchanged. Consequently, every monomial, including a product of variables of both kinds, is replaced by itself plus terms of strictly smaller weight. The Pl\"ucker relations are homogeneous for this weight, so this conclusion also holds in the coordinate ring of the Grassmannian and for the fixed weight-homogeneous representatives.

Let $\psi_{X,\delta}$ (\resp $\psi_{X',\delta}$) be the sum of monomials in $\psi_X$ (\resp $\psi_{X'}$) whose total degree of the variables with the form $s_0\wedge s_{i_1}\wedge\cdots\wedge s_{i_{d}}$ is exactly $\delta$, where $1\leqslant i_1<\cdots<i_{d}\leqslant n$. Then by the above calculation, we have
\[\psi_{X,\delta}=\psi_{X',\delta}.\]
\subsection{Height of Cayley forms}
By \cite[Theorem 4.3.2]{BGS94}, the arithmetic intersection height of a pure-dimensional projective variety over a number field can be compared with the coefficient height of its Chow or Cayley form. For fixed ambient dimension, dimension and degree, their difference is uniformly and explicitly bounded. 

We state the following result on the comparison of the height of $X$ and that of $\psi_X$, which will be useful later. 
\begin{prop}[cf. Proposition 3.7, \cite{Liu-reduced}]\label{compare arakelov height and chow form height}
Let $K$ be a number field, and $\overline{\sE}$ be the Hermitian vector bundle of rank $n+1$ over $\spec\O_K$ defined in \S \ref{O{n+1} with l2-norm}. 
Let $X$ be a closed subscheme of $\mathbb P(\sE_K)$ of pure dimension $0\leqslant d<n$ and degree $\delta$, and $\psi_X\in\sym^\delta\left(\bigwedge^{d+1}\sE_K^\vee\right)$ be the fixed standard-monomial representative of its Cayley form. Then we have
\[
\begin{aligned}
-\frac12\log{N(n,d)+\delta\choose\delta}
-\frac{(d+1)\delta}{2}\mathcal H_n
&\leqslant h(\psi_X)-h_{\overline{\O(1)}}(X)\\
&\leqslant4^{d+2}(d+1)^4n^2\log(2d+2)\,\delta,
\end{aligned}
\]
where $h_{\overline{\O(1)}}(X)$ is defined in \S \ref{def of Arakelov height} with $\overline{\O(1)}$ equipped with the corresponding Fubini--Study metrics, and $\displaystyle N(n,d)={n+1\choose d+1}-1$, $\mathcal H_N=1+\cdots+\frac{1}{N}$. If $n=d+1$, or if $d=1$ and $\delta=2$, the upper bound can be replaced by
\[
\delta\left((N(n,d)+1)\log2+4\log(N(n,d)+1)
                  -\frac12\mathcal H_{N(n,d)}\right).
\]
\end{prop}
\begin{proof}
For a hypersurface defined by $f$, the arithmetic divisor formula in \cite[\S 3.3]{BGS94} expresses its arithmetic height as the finite coefficient terms and the archimedean unit-sphere averages of $\log|f|$, plus $\frac{\delta}{2}\sum\limits_{j=1}^{n}\mathcal H_j$. The sphere average is at most the logarithm of the largest coefficient: use Jensen's inequality, orthogonality of distinct monomials, and $\sum\limits_{i_0+\cdots+i_n=\delta}|z_0^{i_0}\cdots z_n^{i_n}|^2\leqslant1$ on the unit sphere. In the other direction, rotate a point of maximum modulus to a coordinate point and apply Jensen's formula on circles in the remaining coordinates. Since the sphere average of $\log|z_0|$ is $-\mathcal H_n/2$, Cauchy's coefficient formula on the torus $|z_i|=(n+1)^{-1/2}$ bounds the logarithm of the largest coefficient minus the sphere average by $\frac{\delta}{2}(\mathcal H_n+\log(n+1))$. These two estimates prove the hypersurface comparison stated in \eqref{comparison of naive height and Arakelov height}.

Substitute the maximal minors of a generic $(d+1)\times(n+1)$ matrix into $\psi_X$. The resulting Chow form has degree $\delta$ in each row. At every finite place its coefficient maximum equals that of $\psi_X$: substitution gives one inequality, and for the reverse inequality we divide $\psi_X$ by a coefficient of maximal absolute value at that place. The normalized form has a nonzero coefficient in the residue field, where the standard monomials in the maximal minors are linearly independent. By the Chow-height formula in \cite[\S 4.3]{BGS94}, $h_{\overline{\O(1)}}(X)$ is the normalized sum of these finite logarithmic coefficient norms and the archimedean averages of the logarithm of the absolute value of the Chow form on a product of unit spheres, plus $\frac{(d+1)\delta}{2}\mathcal H_n$.

For unit row vectors, Cauchy--Binet and Hadamard's inequality give Euclidean norm at most one for the vector of maximal minors. The sum of the squared absolute values of all its monomials of degree $\delta$ is therefore at most one. Cauchy--Schwarz bounds the absolute value of the Chow form by ${N(n,d)+\delta\choose\delta}^{1/2}$ times the largest absolute value of a coefficient of $\psi_X$. Integration proves the lower bound. The sphere and coefficient estimates used above for a hypersurface, applied to each row, also show that the coefficient height of the Chow form is at most $h_{\overline{\O(1)}}(X)+(d+1)\delta\log(n+1)$.

For the upper bound we compare the two coefficient norms explicitly. Relabel the coordinate order defining the standard monomials as $0<\cdots<n$; this only permutes coordinates and changes signs. The classical quadratic straightening law expresses a product of incomparable Pl\"ucker variables as an integral sum of standard products. In each resulting product the smaller index is strictly smaller than both original indices in the componentwise partial order, and the multiset of entries is preserved; see \cite[Theorem 2, classical case]{SottileSturmfels2001}. For a fixed multiset there are at most $4^{d+1}$ standard quadratic products. Their expansions in matrix entries are linearly independent and have coefficients bounded by $((d+1)!)^2$. Choose a nonsingular square integer submatrix of their coefficient matrix and apply Cramer's rule. Its determinant has absolute value at least one, so the sum of the absolute values of the coefficients in any quadratic straightening is at most
\[
4^{d+1}(4^{d+1})!\bigl((d+1)!\bigr)^{2\cdot4^{d+1}}
\leqslant\exp\bigl(4^{d+2}(d+1)\log(2d+2)\bigr).
\]
Give the Pl\"ucker variable indexed by $(i_0,\ldots,i_d)$ the nonnegative integer weight $(d+1)^2n^2-(i_0+\cdots+i_d)^2$. Each straightening term lowers the total weight by at least two: the sum of the two index sums is preserved, while the smaller index sum strictly decreases below both original ones. Hence each branch of straightening a monomial of degree $D$ has length at most $D(d+1)^2n^2/2$. The sum of absolute coefficients is consequently multiplied by at most
\[
\exp\bigl(2D(d+1)^3n^2\,4^{d+1}\log(2d+2)\bigr).
\]

Now put the identity matrix in the first $d+1$ columns and leave the remaining entries free. No surviving coefficients of the Chow form are added together: the exponents of the free entries determine the exponent of the single entry set to one in each row, since that row has degree $\delta$. The resulting polynomial has at most $(n-d+1)^{(d+1)\delta}$ monomials. Express each free entry as a signed one-replacement minor divided by the first maximal minor, and clear denominators to homogeneous degree $(d+1)\delta$. In the Grassmannian coordinate ring this is $\psi_X$ multiplied by the $d\delta$-th power of the first Pl\"ucker variable. That variable is the least one, so this multiplication preserves standard monomials. The coefficients of the final standard form are therefore precisely those of $\psi_X$. Taking $D=(d+1)\delta$ in the preceding estimate and summing over all places gives
\[
\begin{aligned}
h(\psi_X)-h_{\overline{\O(1)}}(X)
\leqslant{}&\delta\bigl((d+1)\log(n+1)+(d+1)\log(n-d+1)\\
&\hspace{17mm}+2\cdot4^{d+1}(d+1)^4n^2\log(2d+2)\bigr)\\
\leqslant{}&4^{d+2}(d+1)^4n^2\log(2d+2)\,\delta.
\end{aligned}
\]

If $n=d+1$, the fixed Cayley representative is the defining equation up to signs and permutation of coefficients; the hypersurface calculation at the start of this proof gives the asserted sharper bound. If $d=1$ and $\delta=2$, the diagonal term of each standard product of two minors occurs in no other term of any standard quadratic product. Indeed, a swapped determinant term violates either the smallest-index or the largest-index inequality required of a standard diagonal term. Thus each Cayley coefficient occurs, up to sign, as an individual Chow coefficient. The preceding Chow-height bound gives $h(\psi_X)-h_{\overline{\O(1)}}(X)\leqslant2\delta\log(n+1)$, which is bounded by the asserted special-case expression. The convention used in Section 7 is the standard-monomial convention for the order $0<2<3<1$, so this argument applies to that representative as well.
\end{proof}
\begin{rema}
In \cite[Proposition 3.7]{Liu-reduced}, the computation of the constants is incorrect, and Proposition \ref{compare arakelov height and chow form height} gives the corrected comparison for the fixed standard-monomial representative used here. The estimate
\[|h(\psi_X)-h_{\overline{\O(1)}}(X)|\ll_{n,d}\delta\]
remains valid and is the estimate used in the applications. 
\end{rema}

\section{Projective coverings}\label{projective coverings}
In the global determinant method developed by Salberger \cite{Salberger_preprint2013}, the generalized versions \cite{CCDN2020,ParedesSasyk2022}, and the Arakelov formulation \cite{Liu2022d}, the auxiliary-covering estimates are stated for geometrically integral hypersurfaces in a projective space. For our further application, we will consider the projective variety which is a hypersurface in a linear subvariety of the base projective space. In other words, let $X$ be a hypersurface in $\mathbb P^m$, and $\mathbb P^m\hookrightarrow\mathbb P^n$ be an embedding such that the image of $\mathbb P^m$ is a linear subvariety of dimension $m$ in $\mathbb P^n$. We build the global determinant method for this kind of $X$ in $\mathbb P^n$.
\subsection{The global determinant method of Arakelov formulation: a modification d'apr\`es \cite{Liu2022d}}
First, we will give a suitable generalization of \cite[Theorem 3.1]{Liu2022d} for the target. 
\subsubsection{}\label{def of Q(r)}
Let $k$ be a field, and $X$ be a closed subvariety in $\mathbb P^n_k$ of dimension $d\geqslant1$ which is a hypersurface in a linear subvariety of $\mathbb P^n_k$. Let $\xi\in X(k)$, and we consider the local ring $(\O_{X,\xi},\mathfrak m_{X,\xi},\kappa(\xi))$, where $\mathfrak m_{X,\xi}$ is the maximal ideal of $\O_{X,\xi}$ and $\kappa(\xi)$ is the residue field of $\O_{X,\xi}$.

We consider the function
\[H_\xi(s)=\dim_{\kappa(\xi)}\left(\mathfrak m_{X,\xi}^s/\mathfrak m_{X,\xi}^{s+1}\right)\]
of the variable $s\in\mathbb N$, which is called the \textit{local Hilbert--Samuel function} of $\xi$ in $X$.

In fact, we have
\[H_\xi(s)=\frac{\mu_\xi(X)}{(d-1)!}s^{d-1}+O(s^{d-2}),\]
where $\mu_\xi(X)\in\mathbb N_{\geqslant1}$. We call the integer $\mu_\xi(X)$ the \textit{multiplicity} of $\xi$ in $X$.

If $X$ is a hypersurface in a linear subvariety of $\mathbb P^n_k$, then for every $\xi\in X(k)$, its local Hilbert--Samuel function is the same as when $X$ is considered as a hypersurface in $\mathbb P^{d+1}_k$. Then in this case, we have
\[H_\xi(s)={d+s\choose s}-{d+s-\mu_\xi(X)\choose s-\mu_\xi(X)}\]
by \cite[Example 2.70 (2)]{Kollar2007}. Here the second binomial coefficient is understood to be zero when $s<\mu_\xi(X)$.

We define the sequence $\{q_\xi(m)\}_{m>0}$ as the non-decreasing sequence of non-negative integers such that every integer $s\in\mathbb N$ appears exactly $H_\xi(s)$ times in this sequence, and $\{Q_\xi(m)\}_{m\geqslant 0}$ as its sequence of partial sums, with $Q_\xi(0)=0$. 

By \cite[Proposition A.1]{Liu2022d}, we have the explicit lower bound
\begin{equation}Q_\xi(m)>\left(\frac{d!}{\mu_\xi(X)}\right)^{\frac{1}{d}}\left(\frac{d}{d+1}\right)m^{\frac{d+1}{d}}-\frac{d^3+5d^2+8d}{2(d+1)(d+2)}m\end{equation}
for the case described above and every $m\geqslant1$.
\subsubsection{}
We have the following immediate generalization of \cite[Theorem 3.1]{Liu2022d}, where we consider the hypersurfaces in linear spaces of a projective space. 
\begin{theo}\label{existence of the determinant}
Let $K$ be a number field, and $\O_K$ be its ring of integers. Let $\overline{\sE}$ be a Hermitian vector bundle on $\spec\O_K$, $X$ be an integral closed subscheme in $\mathbb P(\sE_K)$ which is a hypersurface in a linear subvariety of $\mathbb P(\sE_K)$, and $\mathscr X$ be the Zariski closure of $X$ in $\mathbb P(\sE)$. Let $\{\p_j\}_{j\in J}$ be a finite set of maximal ideals of $\O_K$, and $\{P_i\}_{i\in I}$ be a nonempty finite family of rational points in $X$. For a fixed prime ideal $\p$, let $\mu_\xi(\mathscr X_\p)$ be the multiplicity of the point $\xi$ in $\mathscr X_\p$, and we denote $n(\mathscr X_\p)=\sum\limits_{\xi\in\mathscr X(\f_\p)}\mu_\xi(\mathscr X_\p)$. We denote $d=\dim(X)\geqslant1$ and $N(\p)=\#(\O_K/\p)$. For an integer $D\geqslant1$, if the inequality
\begin{eqnarray*}
\sup_{i\in I}h_{\overline{\O(1)}}(P_i)&<&\frac{\wmu(\F_{D,J})}{D}-\frac{\log\rg(F_D)}{2D}\\
& &+\frac{1}{[K:\Q]}\sum_{j\in J}\frac{(d!)^{1/d}d\rg(F_D)^{1/d}}{(d+1)Dn(\mathscr X_{\p_j})^{1/d}}\log N(\p_j)\\
& &-\frac{d^3+5d^2+8d}{2D(d+1)(d+2)[K:\Q]}\sum_{j\in J}\log N(\p_j)
\end{eqnarray*}
is valid, then there exists a section $s\in E_{D,K}$ whose zero locus contains the set $\{P_i\}_{i\in I}$ but does not contain the generic point of $X$. In other words, $\{P_i\}_{i\in I}$ can be covered by a hypersurface of degree $D$ in $\mathbb P(\sE_K)$ which does not contain the generic point of $X$.
\end{theo}
\begin{rema}
The closure of the containing linear subvariety is a projective subbundle of $\mathbb P(\sE)$, and each fiber of $\mathscr X$ is a hypersurface in the corresponding fiber of this subbundle. The proof of Theorem \ref{existence of the determinant} therefore follows \cite[Theorem 3.1]{Liu2022d}, using the local Hilbert--Samuel bound in \S \ref{def of Q(r)} with the dimension $d$. We refer to \cite[\S 3]{Liu2022d} for the details. 
\end{rema}
\subsection{Some quantitative estimates}
In order to apply Theorem \ref{existence of the determinant} to control the auxiliary hypersurface in that case, we need the estimates of some arithmetic and geometric quantities. The same estimates are used in \cite[\S 4]{Liu2022d}, following \cite{Salberger_preprint2013,CCDN2020}; see also \cite{ParedesSasyk2022}.
\subsubsection{}
Let $k$ be a field, and $X$ be a closed subvariety in $\mathbb P^n_k$ which is a hypersurface in a linear subvariety of $\mathbb P^n_k$. We suppose $\dim(X)=d$ and $\deg(X)=\delta$. Then we have the embedding
\[X\hookrightarrow\mathbb P^{d+1}_k\hookrightarrow\mathbb P^n_k,\]
where the image of $\mathbb P^{d+1}_k$ in $\mathbb P^n_k$ is a linear subvariety. In this case, we have $\O_{\mathbb P^n_k}(1)|_X=\O_{\mathbb P^{d+1}_k}(1)|_X$. Then for every $D\in\mathbb N$, we have $H^0(X,\O_{\mathbb P^n_k}(1)|_X^{\otimes D})=H^0(X,\O_{\mathbb P^{d+1}_k}(1)|_X^{\otimes D})$. For $D\geqslant\delta$, we have
\[\rg(F_D)={d+1+D\choose d+1}-{d+1-\delta+D\choose d+1},\]
which is from the result for usual projective hypersurfaces directly. 

We have the following explicit estimates of $\rg(F_D)$ below, which are obtained from the same calculation as \cite[Lemma 4.6]{Liu2022d} directly. We omit the detailed calculation here. 
\begin{lemm}\label{explicit bounds of geometric Hilbert-Samuel function}
Let $k$ be a field, and $X$ be a closed subscheme of $\mathbb P^n_k$ of dimension $d\geqslant 1$ and degree $\delta$, which is a hypersurface in a linear subvariety of $\mathbb P^n_k$. Let $\rg(F_D)$ be the geometric Hilbert--Samuel function of $X$ with the variable $D\in\mathbb N$. When $D\geqslant\delta$, we have
\[\rg(F_D)^{\frac{1}{d}}\geqslant\sqrt[d]{\frac{\delta}{d!}}D-(\delta-2)\sqrt[d]{\frac{\delta}{d!}},\]
and
\[\rg(F_D)^{\frac{1}{d}}\leqslant\sqrt[d]{\frac{\delta}{d!}}D+\frac{d+1}{2}\sqrt[d]{\frac{\delta}{d!}}.\]
\end{lemm}
\subsubsection{}
Let $\f_q$ be the finite field of $q$ elements, $X$ be a geometrically integral closed subscheme of $\mathbb P^n_{\f_q}$ of dimension $d\geqslant1$ and degree $\delta$, which is a hypersurface in a linear subvariety of $\mathbb P^n_{\f_q}$. In order to estimate $n(X)=\sum\limits_{\xi\in X(\f_q)}\mu_\xi(X)$, we note that the multiplicity $\mu_\xi(X)$ does not depend on the closed embedding. Then the method to estimate $n(X)$ from \cite{cafure2006improved,Liu-multiplicity} in \cite[\S 4.2.1]{Liu2022d} is still valid. In fact, if $X(\f_q)\ne\varnothing$ and either $q\leqslant\delta^2$ or $q\geqslant27\delta^4$, we have
\begin{equation}\label{estimate of sqrt[n-1]{n(X)}}
\frac{1}{n(X)^{\frac{1}{d}}}\geqslant\frac{1}{q}-\frac{(d+1)^2\delta^2}{\max\{q,\delta-1\}^{\frac{3}{2}}}
\end{equation}
by the same calculation. 
\subsubsection{}
Let $\a$ be a nonzero proper ideal of $\O_K$, $\p\in\spm\O_K$, and $N(\a)=\#(\O_K/\a)$. Then we have
\[\frac{1}{[K:\Q]}\sum_{\p\supseteq\a}\frac{\log N(\p)}{N(\p)}\leqslant\log\log(N(\a))+2\]
in \cite[Lemma 4.3]{Liu2022d}, which is a generalization of \cite[Lemma 1.10]{Salberger_preprint2013}, and see also \cite[Lemma 2.6]{ParedesSasyk2022}.
\subsubsection{}
Let $x\geqslant2$, $\p\in\spm\O_K$, and $N(\p)=\#(\O_K/\p)$. For the estimate of
\[\theta_K(x)=\sum_{N(\p)\leqslant x}\log N(\p),\;\psi_K(x)=\sum_{N(\p)\leqslant x}\frac{\log N(\p)}{N(\p)},\; \phi_K(x)=\sum_{N(\p)\leqslant x}\frac{\log N(\p)}{N(\p)^{\frac{3}{2}}},\]
we use the estimates below, as in \cite[\S 4.3.2]{Liu2022d}, following \cite{Rosen1999}. We specify unconditional explicit error bounds using \cite{GarciaLee2022}.

First, we have
\begin{equation}\label{implicit estimate of theta_K}
\left|\theta_K(x)-x\right|\leqslant\epsilon_1(K,x),
\end{equation}
where we take $\epsilon_1(K,x)=3[K:\mathbb Q]x$. Indeed, $\theta_K(x)\leqslant[K:\mathbb Q]\theta_{\mathbb Q}(x)$, and \cite[(6)]{GarciaLee2022} gives $\theta_{\mathbb Q}(x)<1.1x$ for $x>0$. Since $\theta_K(x)\geqslant0$, this choice bounds $|\theta_K(x)-x|$ for every $x\geqslant2$.

Next, we have
\begin{equation}\label{estimate of logx/x}
\left|\psi_K(x)-\log x\right|\leqslant\epsilon_2(K),\end{equation}
where we take
\[
\begin{aligned}
 \epsilon_2(K)={}&
 \exp\bigl(30[K:\mathbb Q]+10\bigr)
 \bigl([K:\mathbb Q]+1\bigr)^{3[K:\mathbb Q]+5}\\
 &\quad\cdot
 |\Delta_K|
 \bigl(1+\log|\Delta_K|\bigr)^{[K:\mathbb Q]},
\end{aligned}
\]
and $\Delta_K$ is the discriminant of $K$. For $[K:\mathbb Q]\geqslant2$, bounding the constant in \cite[Theorem 1]{GarciaLee2022} by means of \cite[Remark 4]{GarciaLee2022} gives
\[
\begin{aligned}
 |\psi_K(x)-\log x|
 &\leqslant
 127\exp\bigl(28.2[K:\mathbb Q]+5\bigr)
 \bigl([K:\mathbb Q]+1\bigr)^{(5[K:\mathbb Q]+9)/2}\\
 &\quad\cdot
 |\Delta_K|^{1/2+1/([K:\mathbb Q]+1)}
 \bigl(\log|\Delta_K|\bigr)^{[K:\mathbb Q]}
 \leqslant\epsilon_2(K).
\end{aligned}
\]

When $K=\mathbb Q$, \cite[Lemma 7]{GarciaLee2022} gives $\psi_{\mathbb Q}(x)<\log x$. The factorial identity
\[
 \log\bigl(\lfloor x\rfloor!\bigr)
 =
 \sum_{p^j\leqslant x}
 \left\lfloor\frac{x}{p^j}\right\rfloor\log p,
\]
together with
\[
 \sum_p\sum_{j\geqslant2}\frac{\log p}{p^j}
 \leqslant
 \sum_{m\geqslant2}\frac{\log m}{m(m-1)}
 \leqslant4,
\]
gives $\psi_{\mathbb Q}(x)\geqslant\log x-6$ for $x\geqslant2$. Here $p$ runs over rational primes and $j$ over positive integers in the factorial identity. Since $\epsilon_2(\mathbb Q)=256e^{40}>6$, the displayed choice also gives \eqref{estimate of logx/x} in this case.

We also have
\begin{equation}\label{implicit estimate of phi_K}
\left|\sum_{N(\p)>x}\frac{\log N(\p)}{N(\p)^{3/2}}
      -\frac2{\sqrt{x}}\right|\leqslant\epsilon_3(K,x),
\end{equation}
where partial summation gives
\[
 \epsilon_3(K,x)=\frac{\epsilon_1(K,x)}{x^{3/2}}
 +\frac32\int_x^{+\infty}\frac{\epsilon_1(K,t)}{t^{5/2}}\,\mathrm dt
\]
for $x\geqslant2$. With the above choice of $\epsilon_1(K,x)$, this gives
\[
 \epsilon_3(K,x)=\frac{12[K:\mathbb Q]}{\sqrt x}.
\]
\begin{rema}
In \cite[\S 7]{Liu2022d}, the prime-distribution constants were made explicit under the Generalized Riemann Hypothesis by applying \cite{grenie2016explicit}. The unconditional estimates in \cite{GarciaLee2022} allow us to make the constants used here explicit without this assumption.
\end{rema}
\subsubsection{}
In this part, we consider certain non-geometrically integral reductions of a scheme over $\spec\O_K$. Let $\overline{\sE}$ be the Hermitian vector bundle defined at \S \ref{O{n+1} with l2-norm} on $\spec\O_K$, $X$ be a geometrically integral closed subscheme of $\mathbb P(\sE_K)$ of dimension $d\geqslant1$ and degree $\delta\geqslant2$, which is a hypersurface in a linear subvariety, and $\mathscr X$ be its Zariski closure in $\mathbb P(\sE)$.

We consider the base change
\[\xymatrix{\mathscr X_{\f_\p}\ar[d]\ar[r]\ar@{}[dr]|{\square}&\spec\f_\p\ar[d]\\\mathscr X\ar[r]&\spec\O_K}\]
for an arbitrary $\p\in\spm\O_K$. Let
\[\mathcal Q(\mathscr X)=\left\{\p\in\spm\O_K\mid \mathscr X_{\f_\p}\rightarrow\spec\f_\p\text{ is not geometrically integral}\right\},\]
which is a finite set by \cite[Th\'eor\`eme 9.7.7]{EGAIV_3}. Suppose the Hermitian bundle $\overline{\O(1)}=\overline{\O_{\mathbb P(\sE)}(1)}$ on $\mathbb P(\sE)$ is equipped with the corresponding Fubini--Study metrics for each $v\in M_{K,\infty}$. In this case, by the argument of \cite[Proposition 3.2, Proposition 4.1]{Liu-non-geometricallyintegral}, originating in \cite{Ruppert1986} and applied as justified below, we have
\begin{eqnarray}\label{non geometrically integral reduction for general case}
& &\frac{1}{[K:\Q]}\sum_{\p\in\mathcal Q(\mathscr X)}\log N(\p)\\
&\leqslant&(\delta^2-1)h_{\overline{\O(1)}}(X)+(\delta^2-1)\Big(3\log\delta +\log{N(n,d)+\delta\choose \delta}\nonumber \\
& &+\left((N(n,d)+1)\log2+4\log(N(n,d)+1)+\log3-\frac{1}{2}\mathcal H_{N(n,d)}\right)\delta\Big),\nonumber
\end{eqnarray}\label{upper bound of non geometrically integral reductions}
where $h_{\overline{\O(1)}}(X)$ is the height of $X$ defined in \S \ref{def of Arakelov height} by the arithmetic intersection theory, $N(\p)=\#(\O_K/\p)$, and the constant $N(n,d)={n+1\choose d+1}-1$, $\mathcal H_m=1+\cdots+\frac{1}{m}$.

Choose $d+2$ coordinates whose projection is an isomorphism on the linear span of $X$. Restrict the Stiefel Chow form to these columns, put an identity block in the first $d+1$ columns, and homogenize to degree $\delta$. By multihomogeneity, the resulting hypersurface equation has coefficients among those of the Chow form, up to signs. At every finite place, retain the primitive normalization of the whole Chow form, without removing the content of this equation. Since $\mathscr X$ is a hypersurface in the projective-bundle closure of its linear span, Chow specialization and multiplicativity \cite[\S 4.3.2(i)]{BGS94} make this degree-$\delta$ equation zero or reducible over an algebraic closure at every $\p\in\mathcal Q(\mathscr X)$, with multiplicities retained. The cited Noether-form argument therefore applies with these local coefficient norms.

The proof of Proposition \ref{compare arakelov height and chow form height} bounds the Chow coefficient height by $h_{\overline{\O(1)}}(X)+(d+1)\delta\log(n+1)$. For $n>d+1$, the inequalities $N(n,d)+1\geqslant n(n+1)/2$ and $d+1\leqslant n-1$ give
\[
(d+1)\log(n+1)\leqslant
(N(n,d)+1)\log2+4\log(N(n,d)+1)-\frac{1}{2}\mathcal H_{N(n,d)}.
\]
Together with ${d+1+\delta\choose\delta}\leqslant{N(n,d)+\delta\choose\delta}$, this gives \eqref{non geometrically integral reduction for general case}. For $n=d+1$, apply the Noether forms directly to the defining equation and use the hypersurface height comparison.

Let
\begin{eqnarray*}\mathcal Q'(\mathscr X)&=&\Big\{\p\in\spm\O_K\mid N(\p)>27\delta^4\text{ and }\mathscr X_{\f_\p}\rightarrow\spec\f_\p\\
& &\text{ is not geometrically integral}\Big\},\end{eqnarray*}
and
\begin{equation}\label{def of b'(X)}
b'(\mathscr X)=\prod_{\p\in\mathcal Q'(\mathscr X)}\exp\left(\frac{\log N(\p)}{N(\p)}\right).
\end{equation}
We have the following upper bound of $b'(\mathscr X)$ by a similar calculation to that of \cite[Proposition 4.4]{Liu2022d} from \eqref{non geometrically integral reduction for general case}, and see also \cite[Corollary 3.2.3]{CCDN2020} and \cite[Lemma 3.23]{ParedesSasyk2022}.
\begin{prop}\label{estiamte of b'(x)}
With the above notations, put $N=N(n,d)$. Then
\[
 b'(\mathscr X)\leqslant
 7(N+1)e^{2\epsilon_2(K)+[K:\mathbb Q]}
 \max\{\delta^{-2}h_{\overline{\O(1)}}(X),1\}.
\]
More precisely, if $c'(\mathscr X)$ is the right-hand side of
\eqref{non geometrically integral reduction for general case}, then
\[
 b'(\mathscr X)\leqslant
 e^{2\epsilon_2(K)+[K:\mathbb Q]}
 \max\left\{1,\frac{c'(\mathscr X)}{27\delta^4}\right\}.
\]
\end{prop}
\begin{proof}
By \eqref{non geometrically integral reduction for general case},
\[
 \sum_{\p\in\mathcal Q(\mathscr X)}\log N(\p)
 \leqslant[K:\mathbb Q]c'(\mathscr X).
\]
Put $Y=\max\{27\delta^4,c'(\mathscr X)\}$. Splitting the sum defining $\log b'(\mathscr X)$ at $Y$ gives
\[
\begin{aligned}
 \log b'(\mathscr X)
 &\leqslant\sum_{27\delta^4<N(\p)\leqslant Y}\frac{\log N(\p)}{N(\p)}
 +\frac1Y\sum_{\p\in\mathcal Q(\mathscr X)}\log N(\p)\\
 &\leqslant\psi_K(Y)-\psi_K(27\delta^4)+[K:\mathbb Q]\\
 &\leqslant\log\frac{Y}{27\delta^4}+2\epsilon_2(K)+[K:\mathbb Q].
\end{aligned}
\]
Exponentiating proves the second assertion. For the first, use
\[
 \log\delta\leqslant\delta,\qquad
 \log{N+\delta\choose\delta}\leqslant N\log(\delta+1)\leqslant N\delta
\]
and
\[
 3+N+(N+1)\log2+4\log(N+1)+\log3\leqslant7(N+1).
\]
Since $\delta\geqslant2$ and the term involving $-\mathcal H_N/2$ can be omitted in an upper bound, we obtain
\[
 \frac{c'(\mathscr X)}{27\delta^4}
 \leqslant\frac1{27}\left(\delta^{-2}h_{\overline{\O(1)}}(X)
                         +\frac{7(N+1)}\delta\right)
 \leqslant7(N+1)\max\{\delta^{-2}h_{\overline{\O(1)}}(X),1\}.
\]
The same last bound is at least $1$, which completes the proof.
\end{proof}
\begin{rema}
With all the notations in Proposition \ref{estiamte of b'(x)}, we have
\[b'(\mathscr X)\ll_{n,K}\max\left\{\delta^{-2}h_{\overline{\O(1)}}(X),1\right\}.\]
\end{rema}
\subsection{Projective coverings}
Let $\overline{\sE}$ be the Hermitian vector bundle defined at \S \ref{O{n+1} with l2-norm} on $\spec\O_K$, and $X$ be a closed subscheme in $\mathbb P(\sE_K)$, which is a hypersurface in a linear subvariety of $\mathbb P(\sE_K)$. Then the geometry of such $X$ is the same as that of a projective hypersurface, and we have a uniform lower bound of its arithmetic Hilbert--Samuel function with optimal dominant term in \cite{Liu2024b}. In the determinant method with the formulation of Arakelov geometry, we will construct an auxiliary hypersurface in $\mathbb P(\sE_K)$ which contains $S(X;B)$ but does not contain the generic point of $X$, where $B\in\mathbb R^+$.
\subsubsection{}
If $B$ is small enough compared with the height of $X$, then $S(X;B)$ is contained in a hypersurface of bounded degree in terms of $n,\delta,\epsilon$ in $\mathbb P(\sE_K)$ which does not contain the generic point of $X$. In order to give an accurate description, first we recall the following result, which is obtained by applying the slope inequalities of Arakelov geometry to \eqref{evaluation map of eta_D}.
\begin{prop}[\cite{Chen1}, Proposition 2.12]\label{evaluation map and slope inequalities}
Let $X$ be a closed integral subscheme of $\mathbb P(\sE_K)$, $Z=(P_i)_{i\in I}$ be a nonempty finite family of rational points of $X$, and
\[\begin{array}{rrcl}
\phi_{Z,D}:&F_{D,K}&\longrightarrow&\bigoplus\limits_{i\in I}P_i^*\O_{\mathbb P(\sE_K)}(D)\\
&s&\mapsto&(s(P_i))_{i\in I}  
\end{array}\]
be the evaluation map with respect to $D\in\mathbb N_{\geqslant1}$. If the inequality
\[\sup_{i\in I}h_{\overline{\O(1)}}(P_i)<\frac{\wmu_{\max}(\overline F_{D,J})}{D}-\frac{\log \rg(F_D)}{2D}\]
is valid, where $h_{\overline{\O(1)}}(\ndot)$ follows the definition in \eqref{general arakelov height of point}, then the homomorphism $\phi_{Z,D}$ is not able to be injective.
\end{prop}

We apply the lower bound of the arithmetic Hilbert--Samuel function in \eqref{lower bound of arithmetic Hilbert-Samuel} of certain arithmetic varieties, and then we obtain the following result. It is an immediate generalization of \cite[Theorem 5.11]{Liu2024b}.
\begin{prop}\label{naive siegal lemma}
Let $X$ be an integral closed subscheme of $\mathbb P(\sE_K)$ of dimension $d\geqslant1$ and degree $\delta$, which is a hypersurface of a linear subvariety of $\mathbb P(\sE_K)$. For every $\epsilon>0$, if 
\begin{eqnarray*}
  \frac{\log B}{[K:\mathbb Q]}&<&\frac{h_{\overline{\O(1)}}(X)}{(d+1)\delta+\epsilon}-\frac{1}{2(d+1)}\left(7\log(N(n,d)+1)+\frac{1}{2}\mathcal H_{N(n,d)}\right)\\
  & &+B_1(d)-\log(n+1)
\end{eqnarray*}
is valid, where the constant $B_1(d)$ depending only on $d$ is defined in \eqref{lower bound of arithmetic Hilbert-Samuel}, then there exists a hypersurface of degree at most
\[\left[ \frac{(d+2)(\delta-1)((d+1)\delta+\epsilon)}{2\epsilon} \right]+1\]
which covers $S(X;B)$ but does not contain the generic point of $X$, where $\displaystyle N(n,d)={n+1\choose d+1}-1$, $\mathcal H_N=1+\cdots+\frac{1}{N}$.
\end{prop}
To prove Proposition \ref{naive siegal lemma}, we use the following numerical lemma. 
\begin{lemm}[Lemma 5.12, \cite{Liu2024b}]\label{control of R(D)}
Let $r(n,D)={n+D\choose n}$, $r_1(n,D)=r(n,D)-r(n,D-\delta)$, and 
\[R(D)=\frac{r(n,D-\delta)}{Dr_1(n,D)}.\]
When $\displaystyle D\geqslant\delta$, we have 
\[R(D)\geqslant\frac{1}{n\delta}\left(1-\frac{(n+1)(\delta-1)}{2D}\right).\]
\end{lemm}
\begin{proof}[Proof of Proposition \ref{naive siegal lemma}]
If $S(X;B)=\varnothing$, any hyperplane not containing $X$ gives the assertion. We therefore assume that $S(X;B)\ne\varnothing$.

Apply Lemma \ref{control of R(D)} with $n=d+1$. For a given $\epsilon>0$, we first find $D_\epsilon$ such that
\[
 \frac{r(d+1,D-\delta)}{Dr_1(d+1,D)}
 \geqslant\frac1{(d+1)\delta+\epsilon}
 \qquad(D\geqslant D_\epsilon).
\]
By that lemma it suffices to have
\[
 \frac1{(d+1)\delta}
 \left(1-\frac{(d+2)(\delta-1)}{2D}\right)
 \geqslant\frac1{(d+1)\delta+\epsilon},
\]
or equivalently
\[
 \frac{(d+2)(\delta-1)}{2D}
 \leqslant\frac{\epsilon}{(d+1)\delta+\epsilon}.
\]
Therefore we take
\[
 D_\epsilon=
 \left[\frac{(d+2)(\delta-1)((d+1)\delta+\epsilon)}{2\epsilon}\right]+1,
\]
which gives the required inequality for every $D\geqslant D_\epsilon$.
The standard arithmetic height $h_{\overline{\O(1)}}(X)$ is nonnegative. Hence \eqref{lower bound of arithmetic Hilbert-Samuel} gives the full chain
\[
\begin{aligned}
 \frac{\wmu(\F_{D,J})}{D}
 &\geqslant\frac{r(d+1,D-\delta)}{Dr_1(d+1,D)}
 h_{\overline{\O(1)}}(X)+B_1(d)\\
 &\geqslant\frac{h_{\overline{\O(1)}}(X)}{(d+1)\delta+\epsilon}+B_1(d)\\
 &\geqslant\frac{h_{\overline{\O(1)}}(X)}{(d+1)\delta+\epsilon}
 -\frac{7\log(N(n,d)+1)+\mathcal H_{N(n,d)}/2}{2(d+1)}+B_1(d).
\end{aligned}
\]
Also,
\[
 r_1(d+1,D)\leqslant\rg(E_D)={n+D\choose n}\leqslant(n+1)^D,
 \qquad \frac{\log\rg(F_D)}{2D}\leqslant\frac12\log(n+1).
\]
For every $P\in S(X;B)$, the comparison of the point heights gives
\[
 h_{\overline{\O(1)}}(P)
 \leqslant\frac{\log B}{[K:\mathbb Q]}+\frac12\log(n+1).
\]
Consequently, the strict inequality in the statement implies
\[
\begin{aligned}
 \sup_{P\in S(X;B)}h_{\overline{\O(1)}}(P)
 &<\frac{\wmu(\F_{D,J})}{D}-\frac{\log\rg(F_D)}{2D}\\
 &\leqslant\frac{\wmu_{\max}(\F_{D,J})}{D}-\frac{\log\rg(F_D)}{2D}.
\end{aligned}
\]
By Proposition \ref{evaluation map and slope inequalities}, the evaluation map is not injective. Taking $D=D_\epsilon$, and lifting a nonzero kernel element to $E_{D,K}$, gives a homogeneous equation of degree $D_\epsilon$ vanishing on $S(X;B)$ but not identically on $X$.
\end{proof}

\subsubsection{}
For a general $B\in\mathbb R^+$ in $S(X;B)$, we will provide a result to control $S(X;B)$ from the so-called global determinant method in this part. This method was first introduced by Salberger in \cite{Salberger_preprint2013}, then improved by \cite{CCDN2020}, and was generalized over global fields by \cite{ParedesSasyk2022}. See \cite[\S 5]{Liu2022d} for the formulation of Arakelov geometry.

We have the following result by applying the uniform lower bound of the Hilbert--Samuel function in \eqref{lower bound of arithmetic Hilbert-Samuel} to Theorem \ref{existence of the determinant}. The proof follows the argument of \cite[Theorem 5.4]{Liu2022d}, with the arithmetic Hilbert--Samuel lower bound appropriate to this setting. We give the calculation below.

Let
\begin{equation}\label{kappa_1(K)}
\kappa_1(K)=\sup_{x\geqslant2}\frac{\epsilon_1(K,x)}{x}=3[K:\mathbb Q]
\end{equation}
be a constant depending only on $K$, where $\epsilon_1(K,x)$ is defined in \eqref{implicit estimate of theta_K}. Let
\begin{equation}\label{kappa_2(K,d)}
\kappa_2(K,d)=3\log3+\frac{2(d+1)^2}{3\sqrt3}
 (1+4\kappa_1(K))+2\epsilon_2(K),
\end{equation}
where $\epsilon_2(K)$ is defined in \eqref{estimate of logx/x}. Indeed,
\[
 \epsilon_3(K,x)\leqslant\frac{4\kappa_1(K)}{\sqrt x},
\]
so this value bounds
$3\log3+2(d+1)^2/(3\sqrt3)+2(d+1)^2\delta^2\epsilon_3(K,27\delta^4)+2\epsilon_2(K)$
for every $\delta\geqslant1$.

We have the following result. 
\begin{theo}\label{global determinant of hypersurface in a linear subvariety}
Let $X\subset\mathbb P(\sE_K)$ be geometrically integral of dimension $d\geqslant1$ and degree $\delta\geqslant2$, and suppose that $X$ is a hypersurface in a linear subvariety. For $0<\epsilon<1$, put
\[
 D_\epsilon=\left[\frac{(d+2)(\delta-1)((d+1)\delta+\epsilon)}{2\epsilon}\right]+1.
\]
For $B\geqslant2$, there is a hypersurface of degree $\omega$ containing $S(X;B)$ and not containing $X$, with
\[
 \omega\leqslant D_\epsilon+1+
 \frac{e^{C_1(n,d,K,\epsilon)}}\epsilon\delta^{4-1/d}
 B^{\frac{d+1}{d\delta^{1/d}}+\frac{\epsilon}{d\delta^{1+1/d}}}
 \frac{b'(\mathscr X)}{H_K(X)^{1/(d\delta^{1+1/d})}}.
\]
One may take the following explicit constant, independent of $X,B,\delta$:
\[
\begin{aligned}
 C_1(n,d,K,\epsilon)={}&
 \frac{[K:\mathbb Q]}{d\,2^{1/d}}
 \left(d+1+\frac\epsilon2\right)\\
 &\quad\cdot\left[\frac{7\log(N(n,d)+1)+\mathcal H_{N(n,d)}/2}{2(d+1)}
       -B_1(d)+\log(n+1)\right]\\
 &+\kappa_2(K,d)+4+\frac{\log(d!)}d\\
 &+\frac{d^3+5d^2+8d}{2d(d+2)(d!)^{1/d}}
       \left(1+\frac{d+1}4\right)(1+\kappa_1(K)).
\end{aligned}
\]
Here $H_K(X)=\exp([K:\mathbb Q]h_{\overline{\O(1)}}(X))$,
$N(n,d)={n+1\choose d+1}-1$, and $\kappa_1,\kappa_2,b'$ are defined in
\eqref{kappa_1(K)}, \eqref{kappa_2(K,d)}, and \eqref{def of b'(X)}.
\end{theo}
\begin{proof}
We follow the argument of \cite[Theorem 5.4]{Liu2022d}, applying the arithmetic Hilbert--Samuel estimate \eqref{lower bound of arithmetic Hilbert-Samuel} to the present setting. Put
\[
 C_0=\frac1{2(d+1)}\left(7\log(N(n,d)+1)+\frac12\mathcal H_{N(n,d)}\right)
       -B_1(d)+\log(n+1).
\]
This is the same constant $C_0$ used below; it depends only on $n,d$ and is positive.

\textbf{I. The case of large height varieties. --} Suppose that
\[
 \frac{\log B}{[K:\mathbb Q]}
 <\frac{h_{\overline{\O(1)}}(X)}{(d+1)\delta+\epsilon}-C_0.
\]
By Proposition \ref{naive siegal lemma}, there exists a hypersurface of degree at most $D_\epsilon$ containing $S(X;B)$ and not containing $X$. This gives the assertion in this case.

\textbf{II. The case of small height varieties. --} We now suppose that
\begin{equation}\label{small height case of projective covering}
 h_{\overline{\O(1)}}(X)\leqslant
 ((d+1)\delta+\epsilon)
 \left(\frac{\log B}{[K:\mathbb Q]}+C_0\right).
\end{equation}
Let $D\geqslant D_\epsilon$. In particular, $D\geqslant3\delta\log\delta+d$ for $\delta\geqslant2$ and $0<\epsilon<1$. We use the same notation $r_1(d+1,D)=\rg(F_D)$ as in Lemma \ref{control of R(D)}, and take
\[
 \mathcal R(\mathscr X)=
 \left\{\p\in\spm\O_K\mid\begin{array}{l}
 27\delta^4<N(\p)\leqslant r_1(d+1,D)^{1/d},\\
 \mathscr X_{\f_\p}\text{ is geometrically integral}
 \end{array}\right\}.
\]
Suppose that the evaluation map on $S(X;B)$ is injective. Theorem \ref{existence of the determinant}, the slope estimate used in Proposition \ref{naive siegal lemma}, and the comparison between the naive and Fubini--Study heights then give
\begin{equation}\label{lower bound of the determinant1}
\begin{aligned}
 &\frac{\log B}{[K:\mathbb Q]}
 -\frac{h_{\overline{\O(1)}}(X)}{(d+1)\delta+\epsilon}+C_0\\
 &\quad\geqslant
 \frac{(d!)^{1/d}d\,r_1(d+1,D)^{1/d}}{(d+1)D[K:\mathbb Q]}
 \sum_{\p\in\mathcal R(\mathscr X)}
 \frac{\log N(\p)}{n(\mathscr X_{\f_\p})^{1/d}}\\
 &\qquad-
 \frac{d^3+5d^2+8d}{2D(d+1)(d+2)[K:\mathbb Q]}
 \sum_{\p\in\mathcal R(\mathscr X)}\log N(\p).
\end{aligned}
\end{equation}
We estimate the two sums separately.

\textbf{II-1. --} By \eqref{estimate of sqrt[n-1]{n(X)}},
\[
\begin{aligned}
 \sum_{\p\in\mathcal R(\mathscr X)}
 \frac{\log N(\p)}{n(\mathscr X_{\f_\p})^{1/d}}
 \geqslant{}&
 \sum_{\p\in\mathcal R(\mathscr X)}\frac{\log N(\p)}{N(\p)}\\
 &-(d+1)^2\delta^2
 \sum_{\p\in\mathcal R(\mathscr X)}
 \frac{\log N(\p)}{N(\p)^{3/2}}.
\end{aligned}
\]
If $r_1(d+1,D)^{1/d}>27\delta^4$, the first sum is bounded below by
\[
\begin{aligned}
 &\psi_K\big(r_1(d+1,D)^{1/d}\big)
 -\psi_K(27\delta^4)-\log b'(\mathscr X)\\
 &\qquad\geqslant
 \frac1d\log r_1(d+1,D)-4\log\delta-3\log3
 -\log b'(\mathscr X)-2\epsilon_2(K).
\end{aligned}
\]
Here the excluded primes with non-geometrically integral reduction contribute at most $\log b'(\mathscr X)$, by \eqref{def of b'(X)}. For the second sum, partial summation gives
\[
 \sum_{N(\p)>27\delta^4}\frac{\log N(\p)}{N(\p)^{3/2}}
 \leqslant\frac{2}{3\sqrt3\,\delta^2}
       +\epsilon_3(K,27\delta^4).
\]
Consequently, by \eqref{kappa_2(K,d)},
\[
 \sum_{\p\in\mathcal R(\mathscr X)}
 \frac{\log N(\p)}{n(\mathscr X_{\f_\p})^{1/d}}
 \geqslant\frac1d\log r_1(d+1,D)-4\log\delta
       -\log b'(\mathscr X)-\kappa_2(K,d).
\]
This inequality also holds if $r_1(d+1,D)^{1/d}\leqslant27\delta^4$: its left-hand side is zero, whereas its right-hand side is nonpositive.

\textbf{II-2. --} By \eqref{implicit estimate of theta_K} and \eqref{kappa_1(K)},
\[
 \frac1D\sum_{\p\in\mathcal R(\mathscr X)}\log N(\p)
 \leqslant\frac1D\theta_K\big(r_1(d+1,D)^{1/d}\big)
 \leqslant\frac{r_1(d+1,D)^{1/d}}D(1+\kappa_1(K)).
\]
The upper bound in Lemma \ref{explicit bounds of geometric Hilbert-Samuel function} therefore shows that this is at most
\[
 \left(\frac\delta{d!}\right)^{1/d}
 \left(1+\frac{d+1}{2D}\right)(1+\kappa_1(K)).
\]

\textbf{II-3. --} For this calculation only, write
\[
 \lambda=\frac{d^3+5d^2+8d}{2d(d+2)(d!)^{1/d}}(1+\kappa_1(K)).
\]
Combining II-1 and II-2 with \eqref{lower bound of the determinant1} gives
\[
\begin{aligned}
 &\frac{\log B}{[K:\mathbb Q]}-
 \frac{h_{\overline{\O(1)}}(X)}{(d+1)\delta+\epsilon}+C_0\\
 &\quad\geqslant
 \frac{(d!)^{1/d}d}{(d+1)[K:\mathbb Q]}
 \frac{r_1(d+1,D)^{1/d}}D\\
 &\qquad\cdot\left(\frac1d\log r_1(d+1,D)-4\log\delta
 -\log b'(\mathscr X)-\kappa_2(K,d)-\lambda\right).
\end{aligned}
\]
The rank estimates give
\[
\begin{aligned}
 \frac{r_1(d+1,D)^{1/d}}D
 &\geqslant\left(\frac\delta{d!}\right)^{1/d}
       \left(1-\frac{\delta-2}D\right),\\
 \frac1d\log r_1(d+1,D)
 &\geqslant\log D+\frac{\log\delta-\log(d!)}d
       +\log\left(1-\frac{\delta-2}D\right).
\end{aligned}
\]
Since $D\geqslant3\delta\log\delta+d\geqslant2\delta$, we have
\[
 -\log\left(1-\frac{\delta-2}D\right)\leqslant\log2,
 \qquad \frac{\delta-2}D\log D\leqslant3.
\]
The upper rank estimate also gives
\[
\begin{aligned}
 &\frac1d\log r_1(d+1,D)-4\log\delta-\log b'(\mathscr X)-\kappa_2(K,d)-\lambda\\
 &\quad\leqslant\log D-\left(4-\frac1d\right)\log\delta
 -\frac{\log(d!)}d+\log\left(1+\frac{d+1}{2D}\right)\\
 &\hspace{25mm}-\log b'(\mathscr X)-\kappa_2(K,d)-\lambda
 \leqslant\log D.
\end{aligned}
\]
Here $D\geqslant d+1$, $\kappa_2(K,d)\geqslant3\log3$, and $b'(\mathscr X)\geqslant1$ justify the last inequality.
If the parenthesized expression in the preceding lower bound is nonpositive, the logarithmic rank estimate and the nonnegativity of its left-hand side give the inequality below immediately. If it is positive, apply the lower bound for $r_1(d+1,D)^{1/d}/D$. The loss from the factor $1-(\delta-2)/D$ is at most $3$, because $b'(\mathscr X)\geqslant1$ and all the other subtracted constants are nonnegative. The loss from the logarithm is at most $\log2<1$. Thus in both cases
\begin{equation}\label{lower bound of the determinant2}
\begin{aligned}
 \log D\leqslant{}&\frac{d+1}{d\delta^{1/d}}\log B
 -\frac{(d+1)[K:\mathbb Q]h_{\overline{\O(1)}}(X)}
 {d\delta^{1/d}((d+1)\delta+\epsilon)}\\
 &+\left(4-\frac1d\right)\log\delta+\log b'(\mathscr X)
 +\frac{(d+1)[K:\mathbb Q]}{d\delta^{1/d}}C_0\\
 &+\kappa_2(K,d)+\lambda+\frac{\log(d!)}d+4.
\end{aligned}
\end{equation}
By the small-height condition \eqref{small height case of projective covering},
\[
\begin{aligned}
 &-\frac{(d+1)[K:\mathbb Q]h_{\overline{\O(1)}}(X)}
 {d\delta^{1/d}((d+1)\delta+\epsilon)}\\
 &\quad=-\frac{\log H_K(X)}{d\delta^{1+1/d}}
 +\frac{\epsilon[K:\mathbb Q]h_{\overline{\O(1)}}(X)}
 {d\delta^{1+1/d}((d+1)\delta+\epsilon)}\\
 &\quad\leqslant-\frac{\log H_K(X)}{d\delta^{1+1/d}}
 +\frac{\epsilon\log B}{d\delta^{1+1/d}}
 +\frac{\epsilon[K:\mathbb Q]C_0}{d\delta^{1+1/d}}.
\end{aligned}
\]
The remaining coefficient of $C_0$ is at most
\[
 \frac{[K:\mathbb Q]}{d\,2^{1/d}}\left(d+1+\frac\epsilon2\right).
\]
Substitution into \eqref{lower bound of the determinant2} yields
\[
\begin{aligned}
 \log D\leqslant{}&
 \left(\frac{d+1}{d\delta^{1/d}}+
       \frac{\epsilon}{d\delta^{1+1/d}}\right)\log B
 -\frac{\log H_K(X)}{d\delta^{1+1/d}}\\
 &+\left(4-\frac1d\right)\log\delta
 +\log b'(\mathscr X)+C_1(n,d,K,\epsilon).
\end{aligned}
\]
Choose an integer $D\geqslant D_\epsilon$ strictly greater than the exponential of the right-hand side. The evaluation map is not injective, so a nonzero kernel element lifts to the required auxiliary equation. Rounding up, and using $1/\epsilon\geqslant1$, gives the stated bound.
\end{proof}
\begin{rema}
The restriction to hypersurfaces in linear subvarieties is used in the uniform arithmetic Hilbert--Samuel lower bound. The local determinant and prime-sum arguments are those of \cite{Liu2022d}.
\end{rema}
\subsubsection{}
We give a direct consequence of Theorem \ref{global determinant of hypersurface in a linear subvariety} on the control of auxiliary hypersurfaces, which is independent of the height of the variety. The proof is quite similar to that of \cite[Corollary 5.5]{Liu2022d}.
\begin{coro}\label{global determinant of hypersurface in a linear subvariety without height term}
Under the hypotheses of Theorem \ref{global determinant of hypersurface in a linear subvariety}, there is a hypersurface containing $S(X;B)$ and not containing $X$, of degree
\[
 \omega\leqslant C'_1(n,d,K,\epsilon)\delta^{4-1/d}
 B^{(d+1)/(d\delta^{1/d})+\epsilon/(d\delta^{1+1/d})},
\]
where one may take
\[
\begin{aligned}
 C'_1(n,d,K,\epsilon)={}&\frac{(d+2)^2}\epsilon\\
 &+\frac{7(N(n,d)+1)}\epsilon
 \exp\big(C_1(n,d,K,\epsilon)+2\epsilon_2(K)+[K:\mathbb Q]\big)\\
 &\quad\cdot\left(1+\frac{2d}{e[K:\mathbb Q]}\right).
\end{aligned}
\]
In particular, the same bound holds with the additional factor
$B^{\epsilon/(d\delta^{1+1/d})}$.
\end{coro}
\begin{proof}
We use \eqref{lower bound of the determinant2} before splitting its height term. Since $0<\epsilon<1$ and $\delta\geqslant2$,
\[
 \frac{d+1}{d\delta^{1/d}((d+1)\delta+\epsilon)}
 \geqslant\frac1{2d\delta^{1+1/d}}.
\]
Proposition \ref{estiamte of b'(x)} and $\sup_{u\geqslant0}ue^{-u}=1/e$ imply
\[
\begin{aligned}
 &b'(\mathscr X)
 \exp\left(-\frac{[K:\mathbb Q]h_{\overline{\O(1)}}(X)}{2d\delta^{1+1/d}}\right)\\
 &\quad\leqslant7(N(n,d)+1)e^{2\epsilon_2(K)+[K:\mathbb Q]}
 \max\left\{1,\frac{2d\delta^{-1+1/d}}{e[K:\mathbb Q]}\right\}\\
 &\quad\leqslant7(N(n,d)+1)e^{2\epsilon_2(K)+[K:\mathbb Q]}
 \left(1+\frac{2d}{e[K:\mathbb Q]}\right).
\end{aligned}
\]
The other constant terms in \eqref{lower bound of the determinant2} are bounded by $C_1(n,d,K,\epsilon)$. In the small-height case, the same contradiction argument therefore gives the asserted nonconstant degree term. In the large-height case the degree is at most $D_\epsilon$. In either case the integer cutoff is absorbed using
\[
 D_\epsilon+1\leqslant\frac{(d+2)^2}\epsilon\delta^2
 \leqslant\frac{(d+2)^2}\epsilon\delta^{4-1/d}
 B^{(d+1)/(d\delta^{1/d})}.
\]
These estimates give the displayed value of $C'_1(n,d,K,\epsilon)$.
\end{proof}
Similar to Corollary \ref{global determinant of hypersurface in a linear subvariety without height term}, we have the following control of the auxiliary hypersurface, which retains a dependence on the height.
\begin{coro}\label{global determinant of hypersurface in a linear subvariety with height term}
Under the same hypotheses there is an auxiliary hypersurface with the same covering property, of degree
\[
 \omega\leqslant D_\epsilon+1+
 C''_1(n,d,K,\epsilon)\delta^{3-1/d}
 \frac{B^{\frac{d+1}{d\delta^{1/d}}+\frac{\epsilon}{d\delta^{1+1/d}}}
 \max\{\log B,\delta\}}{H_K(X)^{1/(d\delta^{1+1/d})}},
\]
where
\[
\begin{aligned}
 C''_1(n,d,K,\epsilon)={}&
 \frac{7(N(n,d)+1)}\epsilon
 e^{C_1(n,d,K,\epsilon)+2\epsilon_2(K)+[K:\mathbb Q]}\\
 &\cdot\Bigg[1+(d+2)\Bigg(\frac1{[K:\mathbb Q]}
 +\frac{7\log(N(n,d)+1)+\mathcal H_{N(n,d)}/2}{2(d+1)}\\
 &\hspace{40mm}-B_1(d)+\log(n+1)\Bigg)\Bigg].
\end{aligned}
\]
\end{coro}
\begin{proof}
If the large-height inequality of Proposition \ref{naive siegal lemma} holds, it supplies an auxiliary hypersurface of degree at most $D_\epsilon$, which is included in the bound of the statement.

Otherwise, reversing that inequality gives
\[
\begin{aligned}
 h_{\overline{\O(1)}}(X)\leqslant{}&
 ((d+1)\delta+\epsilon)\frac{\log B}{[K:\mathbb Q]}\\
 &+\frac{(d+1)\delta+\epsilon}{2(d+1)}
 \left(7\log(N(n,d)+1)+\frac12\mathcal H_{N(n,d)}\right)\\
 &-((d+1)\delta+\epsilon)B_1(d)
 +((d+1)\delta+\epsilon)\log(n+1).
\end{aligned}
\]
Writing $C_0$ as in the proof of Theorem \ref{global determinant of hypersurface in a linear subvariety}, we consequently have
\[
 h_{\overline{\O(1)}}(X)
 \leqslant(d+2)\delta\left(\frac{\log B}{[K:\mathbb Q]}+C_0\right).
\]
Proposition \ref{estiamte of b'(x)} gives
\[
\begin{aligned}
 b'(\mathscr X)\leqslant{}&
 7(N(n,d)+1)e^{2\epsilon_2(K)+[K:\mathbb Q]}\\
 &\cdot\left[1+(d+2)\left(\frac1{[K:\mathbb Q]}+C_0\right)\right]
 \frac{\max\{\log B,\delta\}}\delta.
\end{aligned}
\]
Substitute this into Theorem \ref{global determinant of hypersurface in a linear subvariety}. This gives exactly the stated $C''_1(n,d,K,\epsilon)$.
\end{proof}
\begin{coro}\label{uniform upper bound of points of bounded height in plane curves}
Let $X\subset\mathbb P^n_K$ be a geometrically integral plane curve of degree $\delta\geqslant2$. For $B\geqslant2$ and $0<\epsilon<1$,
\[
\begin{aligned}
 \#S(X;B)\leqslant{}&\delta(D_\epsilon+1)\\
 &+C''_1(n,1,K,\epsilon)\delta^3
 \frac{B^{2/\delta+\epsilon/\delta^2}\max\{\log B,\delta\}}
 {H_K(X)^{1/\delta^2}}.
\end{aligned}
\]
For the bound with $H_K(\psi_X)$ in the denominator, multiply the nonconstant term by
\[
 \exp\left(\frac{[K:\mathbb Q]}\delta
 1024n^2\log4\right).
\]
If $n=2$ or $\delta=2$, one may instead use the multiplier
\[
 \exp\left(\frac{[K:\mathbb Q]}\delta
 \left((N(n,1)+1)\log2+4\log(N(n,1)+1)-\frac12\mathcal H_{N(n,1)}\right)\right).
\]
\end{coro}
\begin{proof}
Apply Corollary \ref{global determinant of hypersurface in a linear subvariety with height term} with $d=1$. B\'ezout's theorem bounds the number of points by $\delta$ times the degree of its auxiliary hypersurface, giving the first inequality. Proposition \ref{compare arakelov height and chow form height} gives
\[
\begin{aligned}
 h_{\overline{\O(1)}}(X)\geqslant h(\psi_X)
 -1024n^2\log4\,\delta.
\end{aligned}
\]
Exponentiating and raising to the power $-1/\delta^2$ gives the first conversion factor. If $n=2$ or $\delta=2$, the corresponding sharper upper bound in Proposition \ref{compare arakelov height and chow form height} gives the second factor. 
\end{proof}
\begin{rema}
Compared with the estimates in \cite{BCN2024,BCK2025}, Corollary \ref{uniform upper bound of points of bounded height in plane curves} provides a worse dependence on the degree of the curves and an additional $\epsilon$-factor. We will use the dependence on the height of the curve in Corollary \ref{uniform upper bound of points of bounded height in plane curves} later, so we present it.
\end{rema}

\section{Affine coverings}\label{chap. 4}
In \cite[Proposition 4.2.1]{CCDN2020} and \cite[Theorem 5.14]{ParedesSasyk2022}, the authors bound the degree of an auxiliary hypersurface covering the integral points of bounded height on a hypersurface in $\mathbb A^n$. This idea originates in the work of Ellenberg--Venkatesh \cite{Ellenberg-Venkatesh2005}. In this section, we will consider the same issue for an affine subvariety which is a hypersurface in a linear subvariety of $\mathbb A^n$.

In \cite{CCDN2020,ParedesSasyk2022}, the obtained upper bounds depend on the height of the leading homogeneous part of the defining polynomial. Its role is replaced by that of the leading part of the Cayley form studied in \S \ref{Chap. Cayley form} in this generalization. Compared with \cite[Theorem 5.14]{ParedesSasyk2022}, we consider a less general set of integral points with bounded height here. 

Throughout the affine estimates, the degree is fixed and constants are uniform in the coefficients and in the ambient linear subvariety.
\subsection{A naive upper bound of the number of integral points}
We begin by setting up the foundational framework for this subject; some ideas are inspired by \cite[\S 5B]{ParedesSasyk2022}.
\subsubsection{}
Let $K$ be a number field, and $x\in\O_K$. We define
\[[x]=\max_{\sigma:K\hookrightarrow\mathbb C}|\sigma(x)|^{[K:\mathbb Q]}.\] Let $B\in\mathbb R_{\geqslant1}$, and we define the set
\[[B]_{\O_K}=\{x\in\O_K\mid [x]\leqslant B\}.\]

For every point $\xi\in\mathbb P^n(K)$ with $H_K(\xi)\leqslant B$, there is an integral representative $(y_0,\ldots,y_n)$ with
\[
 \max_{i,\sigma}|\sigma(y_i)|\leqslant m_1(K)B^{1/[K:\mathbb Q]}.
\]
Here
\begin{equation}\label{constant from 2nd Minkowski theorem}
 m_1(K)\geqslant1
\end{equation}
is a fixed constant obtained by choosing integral ideal-class representatives and balancing the archimedean norms by units; see \cite[Proposition 2.2]{ParedesSasyk2022}. The coefficient ideal can be chosen from a fixed finite set, independently of the number of coordinates.
\subsubsection{}\label{affine open subset of P(E)}
Let $\overline{\sE}$ be the Hermitian vector bundle defined in \S \ref{O{n+1} with l2-norm} over $\spec\O_K$. We choose $T_0,\ldots,T_n$ as a basis of $H^0\left(\mathbb P(\sE_K), \O_{\mathbb P(\sE_K)}(1)\right)$, where $T_i$ corresponds to the $i$-th coordinate of $\O_K^{\oplus(n+1)}$, $i=0,\ldots,n$. 

Let $\mathbb A_0(\sE_K)$ be the affine open set of $\mathbb P(\sE_K)$, which is defined by $T_0\neq0$ in $\mathbb P(\sE_K)$. In this case, its coordinates are $T_1,\ldots,T_n$, and we let $T_0=1$ when we write the coordinate of a closed point. We will omit the coordinate of $T_0$ if there is no ambiguity. 

Let $X$ be a subscheme of $\mathbb A_0(\sE_K)$, $B\in\mathbb R_{\geqslant1}$, and 
\[[\![B]\!]_{\O_K}^{n}=\left\{x\in\O_K^n\mid\max_{\sigma:K\hookrightarrow\mathbb C}\|\sigma(x)\|_2\leqslant B^{1/[K:\mathbb Q]}\right\}.\]
We define \[s(X;B)=X(K)\cap [\![B]\!]_{\O_K}^{n}\] when the embedding of $X$ into $\mathbb A^{n}$ is clear from context.
\subsubsection{}\label{trivial bound of integral points of affine variety}
First, we introduce a naive upper bound for the number of integral points in an affine scheme, and see \cite[Lemma 4.1.1]{CCDN2020} as a reference. It can be proved directly by reductions, which is similar to the proof of \cite[Theorem 3.1]{Browning-PM277}.

Let $X$ be a closed subscheme of $\mathbb A_0(\sE_{K})$ of pure dimension $d$, which is not necessarily irreducible, and suppose its projective closure in $\mathbb P(\sE_K)$ is of degree $\delta$. Let $B\in\mathbb R$, and we denote by $s(X(\mathbb Z);B)$ the subset of $\mathbb Z^n$, where for all $\xi=(x_1,\ldots,x_n)\in s(X(\mathbb Z);B)$, it satisfies 
\[\max\{|x_1|,\ldots,|x_n|\}\leqslant\sqrt{|x_1|^2+\cdots+|x_n|^2}\leqslant B,\]
and it defines a closed point of $X$. Then we have
\begin{equation}\label{trivial upper bound of integral point in X}
\#s(X(\mathbb Z);B)\leqslant \delta(2B+1)^d\leqslant 3^d\delta B^d
\end{equation}
when $B\geqslant1$, and see \cite[\S 2]{Heath-Brown} or \cite[Lemma 4.1.1]{CCDN2020} for a proof.
\subsubsection{}
We consider the construction in \S \ref{change of coordinate: translation by a}. Let $T_0,\ldots,T_n\in H^0\left(\mathbb P(\sE_K),\O(1)\right)$ be a family of coordinate functions on $\mathbb P(\sE_K)$, and $s_0,\ldots,s_n$ be the dual basis of $T_0,\ldots,T_n$. For an $\mathbf{a}=(a_1,\ldots,a_n)\in\mathbb Z^n$, let 
\[T_{\mathbf{a}}:\mathbb P(\sE_K)\rightarrow\mathbb P(\sE_K)\]
be the morphism defined in \S \ref{change of coordinate: translation by a} with respect to the above coordinate, where we consider $\mathbf a$ as an element in $K^n$. Let $X$ be a closed subscheme of $\mathbb P(\sE_K)$ of pure dimension $d$ and degree $\delta$, and $X'=T_{\mathbf a}(X)$. 

Let $\psi_X\in\sym^\delta_K\left(\bigwedge^{d+1}\sE_K^\vee\right)$ be the Cayley form of $X$. We write 
\[\psi_X=\sum_{i=0}^\delta\psi_{X,i},\]
where the total degree of the variables $\left(s_{0}\wedge s_{i_1}\wedge\cdots\wedge s_{i_{d}}\right)_{1\leqslant i_1<\cdots<i_{d}\leqslant n}$ in $\psi_{X,i}$ is exactly $i$ with $i=0,\ldots,\delta$.

We take the same notations for $X'$, and then we have $\psi_{X'}=\psi_X(s_0+a_1s_1+\cdots+a_ns_n,s_1,\ldots,s_n)$ by \S \ref{change of coordinate: translation by a}. Let 
\[\psi_{X'}=\sum_{i=0}^\delta \psi_{X',i},\]
with the same arrangement as above.

We have the following property about $\psi_X$ and $\psi_{X'}$ introduced above, which is useful for a further application.
\begin{lemm}\label{existence of small height integral points}
We keep all the above notations and constructions, and suppose that $d<n$. If $\psi_{X,\delta}\neq0$ and $\psi_{X,0}\equiv0$, then there exists an element $(a_1,\ldots,a_n)\in\mathbb Z^n$ with $\max\limits_{1\leqslant i\leqslant n}\{|a_i|\}\leqslant\delta$, such that $\psi_{X',0}$ is not zero.
\end{lemm}
\begin{proof}
The coefficients of $\psi_{X',0}$ are polynomials in $a_1,\ldots,a_n$ of degree at most $\delta$. They are not all identically zero: an affine point outside $X$ lies on an $(n-d-1)$-plane disjoint from $X$, so a translation moving this point to the origin gives a nonzero constant-weight part of the Cayley form. A nonzero polynomial of degree at most $\delta$ cannot vanish on the whole grid $\{0,\ldots,\delta\}^n$. Choose a grid point where one of these coefficient polynomials is nonzero.
\end{proof}
\subsection{A control of affine hypersurface coverings}
In this part, we give a control of an affine auxiliary hypersurface which covers the integral points of bounded height, which is deduced from Theorem \ref{global determinant of hypersurface in a linear subvariety}. This is a generalization of \cite[Proposition 4.2.1]{CCDN2020} and \cite[Theorem 5.14]{ParedesSasyk2022}, where they consider the case of hypersurfaces in affine space. 

As in both cited results, a single auxiliary hypersurface covers the integral points under consideration. We retain the dependence on the leading Cayley form. Since the ring of integers need not be a principal ideal domain, some techniques are inspired by \cite[Theorem 5.14]{ParedesSasyk2022}. 
\subsubsection{}\label{constructions of affine hypersurface and its Cayley form}
Let $\overline{\sE}$ be the Hermitian vector bundle over $\spec\O_K$ defined in \S \ref{O{n+1} with l2-norm}, and $\mathbb A_0(\sE_K)$ be the affine open subset of $\mathbb P(\sE_K)$ introduced in \S \ref{affine open subset of P(E)} with respect to the coordinates $T_0,\ldots,T_n$, which satisfies $T_0\neq0$. Let $X'$ be a geometrically integral closed subscheme of $\mathbb A_0(\sE_K)$, whose projective closure $X$ in $\mathbb P(\sE_K)$ is a hypersurface in a linear subvariety of $\mathbb P(\sE_K)$, and $\dim(X)=d$, $\deg(X)=\delta$. We denote by $s_0,\ldots,s_n$ the dual basis of $T_0,\ldots,T_n$. 

For the above variety $X$, let $\psi_X\in\sym^\delta_K\left(\bigwedge^{d+1}\sE_K^\vee\right)$ be its Cayley form, and we write 
\[\psi_X=\sum_{i=0}^\delta\psi_{X,i},\]
where the total degree of the variables $\left(s_0\wedge s_{i_1}\wedge\cdots\wedge s_{i_d}\right)_{1\leqslant i_1<\cdots<i_d\leqslant n}$ in $\psi_{X,i}$ is exactly $i$ with $i=0,\ldots,\delta$.

By the above formulation, our main target is the control of the set $s(X';B)$ defined in \S \ref{affine open subset of P(E)}. 
\subsubsection{}
First we introduce some lemmas below on elementary calculations. The first one is an explicit version of \cite[Lemma 4.2.3]{CCDN2020} and \cite[Lemma 5.13]{ParedesSasyk2022} over a number field.
\begin{lemm}\label{increasing and decreasing of logx/x}
Suppose $d\geqslant1$ and $\delta\geqslant2$. Put $a=1/(d\delta^{1+1/d})$ and $N=N(n,d)$. For $1\leqslant y\leqslant H_K(\psi_X)$,
\[
 \frac{\delta^2b'(\mathscr X)}{H_K(\psi_X)^a}
 \leqslant C_2(n,d,K)\frac{\log y+\delta^2}{y^a},
\]
where one may take
\[
\begin{aligned}
 C_2(n,d,K)={}&7(N+1)e^{2\epsilon_2(K)+[K:\mathbb Q]}\\
 &\cdot\left(1+\frac{\log(N+1)+1+(d+1)\mathcal H_n}{4}
             +\frac{d}{e[K:\mathbb Q]}\right).
\end{aligned}
\]
\end{lemm}
\begin{proof}
For a fixed $a>0$, an elementary calculation gives
\[
 \frac{\mathrm d}{\mathrm dx}\left(\frac{\log x}{x^a}\right)
 =\frac{1-a\log x}{x^{a+1}}.
\]
Thus $\log x/x^a$ increases on $[1,e^{1/a}]$, decreases on $[e^{1/a},+\infty[$, and has maximum $1/(ae)$.

Proposition \ref{estiamte of b'(x)} and Proposition \ref{compare arakelov height and chow form height} give
\[
\begin{aligned}
 \delta^2 b'(\mathscr X)
 &\leqslant7(N+1)e^{2\epsilon_2(K)+[K:\mathbb Q]}
       \big(h_{\overline{\O(1)}}(X)+\delta^2\big)\\
 &\leqslant7(N+1)e^{2\epsilon_2(K)+[K:\mathbb Q]}\\
 &\qquad\cdot\left(\frac{\log H_K(\psi_X)}{[K:\mathbb Q]}+
       \frac{\log\binom{N+\delta}{\delta}+(d+1)\delta\mathcal H_n}{2}
       +\delta^2\right).
\end{aligned}
\]
For $x\geqslant y\geqslant1$, write $x=yt$ with $t\geqslant1$. Then
\[
 \frac{\log x}{x^a}
 =\frac{\log y}{y^at^a}+\frac{\log t}{y^at^a}
 \leqslant\frac{\log y+1/(ae)}{y^a}.
\]
Since $\binom{N+\delta}{\delta}\leqslant(N+1)^\delta$ and $\delta\geqslant2$,
\[
\begin{aligned}
 \frac{\log\binom{N+\delta}{\delta}+(d+1)\delta\mathcal H_n}{2}
 &\leqslant\frac{\delta^2}{4}\big(\log(N+1)+1+(d+1)\mathcal H_n\big),\\
 \frac1{[K:\mathbb Q]ae}&=\frac{d\delta^{1+1/d}}{[K:\mathbb Q]e}\leqslant\frac{d\delta^2}{[K:\mathbb Q]e}.
\end{aligned}
\]
Substitute $x=H_K(\psi_X)$ into these inequalities. After division by $x^a$, the preceding upper bound is at most
\[
\begin{aligned}
 \frac{7(N+1)e^{2\epsilon_2(K)+[K:\mathbb Q]}}{y^a}
 \left[\log y+\delta^2\left(1+
 \frac{\log(N+1)+1+(d+1)\mathcal H_n}{4}+\frac{d}{[K:\mathbb Q]e}\right)\right],
\end{aligned}
\]
which proves the assertion with the displayed $C_2$.
\end{proof}

In order to treat the analogue of height in an affine open subset, we introduce a height function in the affine case below. 
\begin{defi}\label{definition of affine height}
Let $x=(x_1,\ldots,x_n)\in K^n\setminus\{0\}$. We define the \textit{affine height} of $x$ as
\[H_{K,\mathrm{aff}}(x)=\prod_{v\in M_{K,\infty}}\max_{1\leqslant i\leqslant n}\{|x_i|_v^{[K_v:\Q_v]}\}\text{, and }h_{\mathrm{aff}}(x)=\frac{\log H_{K,\mathrm{aff}}(x)}{[K:\Q]}.\]
Moreover, let 
\[f(T_0,\ldots,T_n)=\sum_{i_0+\cdots+i_n=\delta}a_{i_0\ldots i_n}T_0^{i_0}\cdots T_n^{i_n}, \]
be a nonzero polynomial. We define
\[H_{K,\mathrm{aff}}(f)=\prod_{v\in M_{K,\infty}}\max_{i_0+\cdots+i_n=\delta}\left\{|a_{i_0\ldots i_n}|_v^{[K_v:\Q_v]}\right\}\text{, and }h_{\mathrm{aff}}(f)=\frac{\log H_{K,\mathrm{aff}}(f)}{[K:\Q]}.\]
\end{defi}

To normalize the coefficients of the Cayley form, we use the following result, inspired by the proof of \cite[Theorem 5.14]{ParedesSasyk2022}.
\begin{lemm}\label{existence of integral Cayley form with bounded height}
The Cayley form $\psi_X$ has an integral coefficient representative $\Psi_X$ such that
\[
 H_K(\psi_X)\leqslant H_{K,\mathrm{aff}}(\Psi_X)
 \leqslant c_K H_K(\psi_X).
\]
Here
\[
 c_K=\max\left\{1,\left(\frac4\pi\right)^{r_2}
                 \frac{[K:\mathbb Q]!}{[K:\mathbb Q]^{[K:\mathbb Q]}}
                 \sqrt{|\Delta_K|}\right\},
\]
where $r_2$ is the number of complex places and $\Delta_K$ is the discriminant of $K$. Its coefficient ideal belongs to a fixed finite set of integral ideals of norm at most $c_K$.
\end{lemm}
\begin{proof}
By the Minkowski ideal-class bound, choose integral ideals of norm at most $c_K$ representing the finitely many ideal classes of $K$. The coefficients of any representative of $\psi_X$ generate a nonzero fractional ideal. Multiplying all coefficients by a suitable element of $K^\times$, we may arrange that this ideal is one of the chosen integral ideals, denoted by $\mathfrak a$. The resulting representative $\Psi_X$ has coefficients in $\O_K$.

At the finite places, the product of the coefficient norms is $N(\mathfrak a)^{-1}$. Hence the definition of projective height gives
\[
 H_K(\psi_X)=\frac{H_{K,\mathrm{aff}}(\Psi_X)}{N(\mathfrak a)}.
\]
The chosen ideals satisfy $1\leqslant N(\mathfrak a)\leqslant c_K$ with the explicit value in the statement. This proves both height inequalities and the assertion about the coefficient ideal. This is the coefficient-vector normalization used in \cite[Proposition 2.2]{ParedesSasyk2022}.
\end{proof}
Let $\Psi_X\in\sym_{\O_K}^{\delta}\left(\bigwedge^{d+1}\sE^\vee\right)$ be one of the forms determined in the statement of Lemma \ref{existence of integral Cayley form with bounded height}. We write 
\begin{equation}\label{decomposition by degree of integral coefficient Cayley form}
\Psi_X=\sum_{i=0}^\delta\Psi_{X,i},
\end{equation}
where $\Psi_{X,i}$ denotes the sum of monomials in $\Psi_X$ whose total degree of the variables $(s_0\wedge s_{i_1}\wedge\cdots\wedge s_{i_d})_{1\leqslant i_1<\cdots< i_d\leqslant n}$ in $\Psi_{X,i}$ is exactly $i$ with $i=0,\ldots,\delta$. 

\subsubsection{}
In order to construct an affine hypersurface to cover the integral points of bounded height in the varieties described in \S \ref{constructions of affine hypersurface and its Cayley form}, we have the following result. 
\begin{theo}\label{control of integral points}
Keep the notation of \S \ref{constructions of affine hypersurface and its Cayley form}, with $d\geqslant1$, $\delta\geqslant2$ and $\psi_{X,\delta}\ne0$.
For $B\geqslant2$ and $0<\epsilon<1$, the set $s(X';B)$ is contained in an affine hypersurface not containing $X'$, of degree
\[
\begin{aligned}
 \omega\leqslant{}&D_\epsilon+1+C_3(n,d,\delta,K,\epsilon)\delta^{2-1/d}
 B^{\frac1{\delta^{1/d}}+\frac{\epsilon}{d\delta^{1+1/d}}}\\
 &\cdot\frac{\min\{h(\psi_{X,\delta})+\delta\log B+\delta^2,\
                         \delta^2b'(\mathscr X)\}}
 {H_K(\psi_{X,\delta})^{1/(d\delta^{1+1/d})}}.
\end{aligned}
\]
Here $D_\epsilon$ is as in Theorem \ref{global determinant of hypersurface in a linear subvariety}. To specify $C_3$, write $N=N(n,d)$ for simplicity below, and retain
\[
 C_0=\frac{7\log(N+1)+\mathcal H_N/2}{2(d+1)}-B_1(d)+\log(n+1).
\]
Define the explicit threshold
\begin{equation}\label{control of B by distribution of prime ideals}
\begin{aligned}
 B(n,d,\delta,K)=&\max\Big\{8\,000\,000,\ 64\big[1+[K:\mathbb Q](d+2)\delta(1+C_0)
 +\log c_K\\
 &\;+[K:\mathbb Q]\delta\big(4^{d+2}(d+1)^4n^2\log(2d+2)\big)
 \big]^2\Big\}^{[K:\mathbb Q]}.
\end{aligned}
\end{equation}
We may take
\[
\begin{aligned}
 &C_3(n,d,\delta,K,\epsilon)=\frac{e^{C_1(n,d,K,\epsilon)}}\epsilon\\
 &\quad\cdot\exp\left(\frac{[K:\mathbb Q]\big(4^{d+2}(d+1)^4n^2\log(2d+2)\big)}
 {d\delta^{1/d}}+\frac{[K:\mathbb Q]}{e}\right)\\
 &\quad\cdot2^{[K:\mathbb Q]/(d\delta^{1/d})}(1+[K:\mathbb Q]C_2(n,d,K))
 B(n,d,\delta,K)^{1/(d\delta^{1/d})}\\
 &\quad\cdot(1+\sqrt n\,\delta)^{[K:\mathbb Q]/\delta^{1/d}+[K:\mathbb Q]\epsilon/(d\delta^{1+1/d})}
 \left(1+\frac{[K:\mathbb Q]\log(1+\sqrt n\,\delta)}\delta\right).
\end{aligned}
\]
The constant is uniform in $X'$ and $B$; the degree $\delta$ is fixed.
\end{theo}
\begin{proof}
We use the rescaling argument of \cite[Proposition 4.2.1]{CCDN2020}, with the Cayley form in place of the defining equation; compare \cite[Theorem 5.14]{ParedesSasyk2022} for integral normalization over a number field. Write $a=1/(d\delta^{1+1/d})$ during this proof.

For a nonzero $L\in\O_K$, consider the morphism
\[
\begin{array}{rrcl}
 F_L:&\mathbb P(\sE_K)&\longrightarrow&\mathbb P(\sE_K)\\
 &[x_0:x_1:\cdots:x_n]&\mapsto&[x_0:Lx_1:\cdots:Lx_n].
\end{array}
\]
By \S \ref{change of coordinate: multiple H}, the Cayley form of $F_L^{-1}(X)$ has the representative
\begin{equation}\label{rescaled Cayley form}
 \Psi_{F_L^{-1}(X)}=\sum_{i=0}^{\delta}L^i\Psi_{X,i}.
\end{equation}
Here $\Psi_X$ is chosen as in Lemma \ref{existence of integral Cayley form with bounded height}, in the fixed basis respecting the weight decomposition, and \eqref{rescaled Cayley form} defines the integral representative on the left.

A point $[1:x_1:\cdots:x_n]$ from $s(X';B)$ has projective height at most $B$. If the inequality of Proposition \ref{naive siegal lemma} holds, it gives an auxiliary hypersurface of degree at most $D_\epsilon$. Restriction to the affine open set proves the assertion in this case. Otherwise,
\[
 h_{\overline{\O(1)}}(X)\leqslant((d+1)\delta+\epsilon)
                         \left(\frac{\log B}{[K:\mathbb Q]}+C_0\right).
\]
Proposition \ref{compare arakelov height and chow form height} consequently gives
\begin{equation}\label{affine coefficient height bound}
\begin{aligned}
 \log H_K(\psi_X)\leqslant{}&((d+1)\delta+\epsilon)(\log B+[K:\mathbb Q]C_0)\\
 &+[K:\mathbb Q]\delta\big(4^{d+2}(d+1)^4n^2\log(2d+2)\big).
\end{aligned}
\end{equation}
We treat this remaining case below.

\textbf{Case I. --} Suppose first that $2\leqslant B\leqslant B(n,d,\delta,K)$. Theorem \ref{global determinant of hypersurface in a linear subvariety} and Proposition \ref{compare arakelov height and chow form height} give a homogeneous polynomial $G$, not vanishing identically on $X$, whose zero locus contains $S(X;B)$ and whose degree is at most
\[
\begin{aligned}
 D_\epsilon+1+{}&\frac{e^{C_1(n,d,K,\epsilon)}}\epsilon
 \exp\left(\frac{[K:\mathbb Q]\big(4^{d+2}(d+1)^4n^2\log(2d+2)\big)}
 {d\delta^{1/d}}\right)\\
 &\cdot\delta^{4-1/d}B^{(d+1)/(d\delta^{1/d})+\epsilon a}
 \frac{b'(\mathscr X)}{H_K(\psi_X)^a}.
\end{aligned}
\]
Since $1\leqslant H_K(\psi_{X,\delta})\leqslant H_K(\psi_X)$, Lemma \ref{increasing and decreasing of logx/x} and this height inequality give, respectively,
\[
\begin{aligned}
 \frac{\delta^2b'(\mathscr X)}{H_K(\psi_X)^a}
 &\leqslant C_2(n,d,K)
 \frac{[K:\mathbb Q] h(\psi_{X,\delta})+\delta^2}{H_K(\psi_{X,\delta})^a},\\
 \frac{\delta^2b'(\mathscr X)}{H_K(\psi_X)^a}
 &\leqslant\frac{\delta^2b'(\mathscr X)}{H_K(\psi_{X,\delta})^a}.
\end{aligned}
\]
The extra factor $B^{1/(d\delta^{1/d})}$ is at most $B(n,d,\delta,K)^{1/(d\delta^{1/d})}$. Taking the smaller of the two bounds, and using $1+[K:\mathbb Q]C_2\geqslant\max\{1,[K:\mathbb Q]C_2\}$, proves the assertion for $G(1,T_1,\ldots,T_n)$ on this range.

\textbf{Case II. --} Suppose $B>B(n,d,\delta,K)$ and choose an integral representative $\Psi_X$ as in Lemma \ref{existence of integral Cayley form with bounded height}. First suppose $\Psi_{X,0}\ne0$, and fix a nonzero coefficient $c\in\O_K$ of $\Psi_{X,0}$. Bounding its absolute value at each archimedean place gives
\[
 1\leqslant |N_{K/\mathbb Q}(c)|
 \leqslant H_{K,\mathrm{aff}}(\Psi_X)
 \leqslant c_K H_K(\psi_X).
\]
\textbf{Case II-1. --} Suppose there is a rational prime $L$ with
\[
 \tfrac12 B^{1/[K:\mathbb Q]}<L\leqslant B^{1/[K:\mathbb Q]},\qquad L\nmid N_{K/\mathbb Q}(c).
\]
Let $\mathfrak a$ be the coefficient ideal of $\Psi_X$. At a prime ideal not above $L$, multiplication of its weight-$i$ coefficients by $L^i$ preserves their valuations. At a prime ideal above $L$, the coefficient $c$ is a unit and is unchanged in \eqref{rescaled Cayley form}. Hence the coefficient ideal of $\Psi_{F_L^{-1}(X)}$ is again $\mathfrak a$.

At every archimedean place the coefficient norm in \eqref{rescaled Cayley form} is at least that of $L^\delta\Psi_{X,\delta}$. If $\mathfrak a_\delta$ is the coefficient ideal of $\Psi_{X,\delta}$, then $\mathfrak a_\delta\subseteq\mathfrak a$, and thus $N(\mathfrak a_\delta)\geqslant N(\mathfrak a)$. It follows that
\begin{equation}\label{height gain by affine rescaling}
\begin{aligned}
 H_K(\psi_{F_L^{-1}(X)})
 &=\frac{H_{K,\mathrm{aff}}(\Psi_{F_L^{-1}(X)})}{N(\mathfrak a)}\\
 &\geqslant L^{[K:\mathbb Q]\delta}
       \frac{H_{K,\mathrm{aff}}(\Psi_{X,\delta})}{N(\mathfrak a)}\\
 &\geqslant L^{[K:\mathbb Q]\delta}H_K(\psi_{X,\delta})\\
 &\geqslant2^{-[K:\mathbb Q]\delta}B^\delta H_K(\psi_{X,\delta}).
\end{aligned}
\end{equation}
Outside the primes above $L$, the morphism $F_L$ extends to an automorphism of the ambient projective space over $\O_K$. The two Zariski closures therefore have the same geometrically integral reductions there. At the remaining primes, the definition of $b'$ gives
\begin{equation}\label{bad primes under affine rescaling}
\begin{aligned}
 b'\big(\overline{F_L^{-1}(X)}\big)
 &\leqslant b'(\mathscr X)
 \exp\left(\sum_{\p\mid L}\frac{\log N(\p)}{N(\p)}\right)\\
 &\leqslant e^{[K:\mathbb Q]/e}b'(\mathscr X),
\end{aligned}
\end{equation}
where the bar denotes Zariski closure in $\mathbb P(\sE)$. There are at most $[K:\mathbb Q]$ prime ideals above $L$, and $\log x/x\leqslant1/e$.

For $x=(x_1,\ldots,x_n)\in s(X';B)$, its image under $F_L^{-1}$ is $[L:x_1:\cdots:x_n]$. All coordinates are integral and their absolute values at every archimedean embedding are at most $B^{1/[K:\mathbb Q]}$. Its projective height is therefore at most $B$.

Apply Theorem \ref{global determinant of hypersurface in a linear subvariety} to $F_L^{-1}(X)$. As in Case I, the height conversion contributes
\[
 \exp\left(\frac{[K:\mathbb Q]\big(4^{d+2}(d+1)^4n^2\log(2d+2)\big)}
 {d\delta^{1/d}}\right).
\]
By \eqref{height gain by affine rescaling} and \eqref{bad primes under affine rescaling}, the nonconstant term in its degree bound is at most
\[
\begin{aligned}
 &\frac{e^{C_1(n,d,K,\epsilon)}}\epsilon
 \exp\left(\frac{[K:\mathbb Q]\big(4^{d+2}(d+1)^4n^2\log(2d+2)\big)}
 {d\delta^{1/d}}+\frac{[K:\mathbb Q]}{e}\right)\\
 &\hspace{12mm}\cdot 2^{[K:\mathbb Q]/(d\delta^{1/d})}\delta^{2-1/d}
 B^{1/\delta^{1/d}+\epsilon a}
 \frac{\delta^2b'(\mathscr X)}{H_K(\psi_{X,\delta})^a}.
\end{aligned}
\]
Indeed, the gain in the denominator is $B^{1/(d\delta^{1/d})}$. Alternatively, apply Lemma \ref{increasing and decreasing of logx/x} to the rescaled variety with
\[
 y=L^{[K:\mathbb Q]\delta}H_K(\psi_{X,\delta})
 \leqslant H_K(\psi_{F_L^{-1}(X)}).
\]
We have $y\geqslant1$ and
\[
 \log y\leqslant\delta\log B+[K:\mathbb Q] h(\psi_{X,\delta}),\qquad
 y^{-a}\leqslant 2^{[K:\mathbb Q]/(d\delta^{1/d})}
 B^{-1/(d\delta^{1/d})}H_K(\psi_{X,\delta})^{-a}.
\]
The same degree term is therefore bounded by the preceding displayed expression with $\delta^2b'(\mathscr X)$ replaced by
\[
[K:\mathbb Q]C_2(n,d,K)\big(h(\psi_{X,\delta})+\delta\log B+\delta^2\big).
\]
(The additional factor $e^{[K:\mathbb Q]/e}\geqslant1$ is harmless in this second bound.) Taking the smaller bound and using $1+[K:\mathbb Q]C_2\geqslant\max\{1,[K:\mathbb Q]C_2\}$ gives the asserted coefficient. If $G$ is the resulting homogeneous equation, then $G(L,T_1,\ldots,T_n)$ vanishes on $s(X';B)$ and has no larger degree. It does not vanish identically on $X'$, since $F_L^{-1}(X')$ is dense in $F_L^{-1}(X)$.

\textbf{Case II-2. --} Suppose every rational prime in $(B^{1/[K:\mathbb Q]}/2,B^{1/[K:\mathbb Q]}]$ divides $N_{K/\mathbb Q}(c)$. Put $R=B^{1/[K:\mathbb Q]}$. The estimate
\[
 |\theta_{\mathbb Q}(x)-x|\leqslant\frac{0.2x}{\log^2x}
 \quad(x\geqslant3\,594\,641)
\]
of \cite[Page~2]{Dusart2010} implies
$\theta_{\mathbb Q}(R)-\theta_{\mathbb Q}(R/2)\geqslant R/4$ for $R\geqslant8\,000\,000$. Hence \eqref{affine coefficient height bound} gives
\[
\begin{aligned}
 R/4&\leqslant\log|N_{K/\mathbb Q}(c)|\\
 &\leqslant [K:\mathbb Q](d+2)\delta(\log R+C_0)+\log c_K\\
 &\quad+[K:\mathbb Q]\delta\big(4^{d+2}(d+1)^4n^2\log(2d+2)\big)\\
 &\leqslant2\sqrt R\big[1+[K:\mathbb Q](d+2)\delta(1+C_0)+\log c_K\\
 &\hspace{25mm}+[K:\mathbb Q]\delta\big(4^{d+2}(d+1)^4n^2\log(2d+2)\big)\big],
\end{aligned}
\]
where $1+\log R\leqslant2\sqrt R$ for $R\geqslant1$. This implies $B\leqslant B(n,d,\delta,K)$, contrary to the choice of Case II. Thus Case II-1 applies.

Finally, if $\Psi_{X,0}=0$, Lemma \ref{existence of small height integral points} supplies an integral translation $T_{\mathbf a}$ with $\max_i|a_i|\leqslant\delta$ such that $\psi_{T_{\mathbf a}(X),0}\ne0$. By \S \ref{change of coordinate: translation by a},
\[
 \psi_{T_{\mathbf a}(X),\delta}=\psi_{X,\delta},\qquad
 b'\big(\overline{T_{\mathbf a}(X)}\big)=b'(\mathscr X).
\]
The second equality holds because the translation and its inverse are defined over $\O_K$. At every embedding it enlarges the Euclidean radius by at most $\sqrt n\,\delta$, so
\[
 T_{\mathbf a}\big(s(X';B)\big)
 \subseteq s\big(T_{\mathbf a}(X');(1+\sqrt n\,\delta)^{[K:\mathbb Q]} B\big).
\]
Apply the preceding cases, including the large-height case, to this translated variety. The power of $B$ costs the factor
$(1+\sqrt n\,\delta)^{[K:\mathbb Q]/\delta^{1/d}+[K:\mathbb Q]\epsilon a}$.
The logarithmic numerator increases by $[K:\mathbb Q]\delta\log(1+\sqrt n\,\delta)$, whereas its original value is at least $\delta^2$. Thus the minimum in the bound increases by at most the factor
$1+[K:\mathbb Q]\log(1+\sqrt n\,\delta)/\delta$.
These are precisely the last two factors of $C_3$. Pulling the equation back by $T_{\mathbf a}$ preserves its degree and nonvanishing on $X'$, and completes the proof.
\end{proof}
\begin{rema}
In Theorem \ref{global determinant of hypersurface in a linear subvariety} and Theorem \ref{control of integral points}, we require the research objects to be hypersurfaces in a linear subspace, instead of general varieties. In fact, the main obstruction is the study of a uniform lower bound of the arithmetic Hilbert--Samuel function, and we only have the optimal dependence on the height of the variety for this particular case. For general cases, we only have a worse estimate in \cite{David_Philippon99}, and see also \cite[\S 4]{Chen1}.
\end{rema}
\begin{rema}
The constants in this section are allowed to depend on the fixed degree. Only the degrees $2$ and $3$ will be used below. The rescaling step uses the explicit prime estimate over $\mathbb Q$ displayed in the proof. The dependence on $K$ in the projective determinant estimates is explicit through the unconditional choices in \eqref{implicit estimate of theta_K}, \eqref{estimate of logx/x} and \eqref{kappa_1(K)}. Under the Generalized Riemann Hypothesis, sharper estimates are available in \cite{grenie2016explicit}; see also \cite[Theorem 2]{grzeskowiak2017explicit} for unconditional estimates in the cases discussed there. This role of the Generalized Riemann Hypothesis is considered in \cite[\S 7]{Liu2022d}.
\end{rema}
\subsection{Some useful consequences}
There are some direct consequences of Theorem \ref{control of integral points}, which will be useful in the further study. In this part, we keep all the notations in Theorem \ref{control of integral points}.
\subsubsection{}
First, we provide a direct consequence of Theorem \ref{control of integral points} which is easier to apply later.
\begin{coro}\label{improved version of the control of integral points}
Under the hypotheses of Theorem \ref{control of integral points}, there is an affine hypersurface containing $s(X';B)$ and not containing $X'$, of degree
\[
 \omega\leqslant D_\epsilon+1+C_4(n,d,\delta,K,\epsilon)\delta^{3-1/d}
 \frac{B^{1/\delta^{1/d}+\epsilon/(d\delta^{1+1/d})}\max\{\log B,\delta\}}
 {H_K(\psi_{X,\delta})^{1/(d\delta^{1+1/d})}}.
\]
There is also such a hypersurface satisfying the height-free bound
\[
 \omega\leqslant C_4(n,d,\delta,K,\epsilon)\delta^{4-1/d}
 B^{1/\delta^{1/d}+\epsilon/(d\delta^{1+1/d})}.
\]
With $N=N(n,d)$ and $C_0$ as in Theorem \ref{control of integral points}, one may take
\[
\begin{aligned}
 C_4(n,d,\delta,K,\epsilon)={}&\frac{2(d+2)^2}\epsilon\\
 &+C_3(n,d,\delta,K,\epsilon)
 \left[2+\frac{d+2}{[K:\mathbb Q]}+(d+2)C_0\right.\\
 &\hspace{37mm}\left.+4^{d+2}(d+1)^4n^2\log(2d+2)\right]\\
 &+C_3(n,d,\delta,K,\epsilon/2)
 \left(2+\frac{d}{e[K:\mathbb Q]}\right)
 \left(1+\frac{2d\delta^{1+1/d}}{e\epsilon}\right).
\end{aligned}
\]
\end{coro}
\begin{proof}
A point of $s(X';B)$ has projective height at most $B$. If
\[
 \frac{\log B}{[K:\mathbb Q]}<\frac{h_{\overline{\O(1)}}(X)}{(d+1)\delta+\epsilon}-C_0,
\]
Proposition \ref{naive siegal lemma} supplies an auxiliary hypersurface of degree at most $D_\epsilon$. Its restriction to the affine open set gives the constant term in the first assertion.

Otherwise, Proposition \ref{compare arakelov height and chow form height} gives the complete estimate
\[
\begin{aligned}
 h(\psi_{X,\delta})&\leqslant h(\psi_X)\\
 &\leqslant h_{\overline{\O(1)}}(X)
       +\delta\big(4^{d+2}(d+1)^4n^2\log(2d+2)\big)\\
 &\leqslant((d+1)\delta+\epsilon)\left(\frac{\log B}{[K:\mathbb Q]}+C_0\right)\\
 &\quad+\delta\big(4^{d+2}(d+1)^4n^2\log(2d+2)\big).
\end{aligned}
\]
Since $((d+1)\delta+\epsilon)\leqslant(d+2)\delta$, it follows that
\[
\begin{aligned}
 &h(\psi_{X,\delta})+\delta\log B+\delta^2\\
 &\quad\leqslant\delta\max\{\log B,\delta\}
 \left[2+\frac{d+2}{[K:\mathbb Q]}+(d+2)C_0\right.\\
 &\hspace{43mm}\left.+4^{d+2}(d+1)^4n^2\log(2d+2)\right].
\end{aligned}
\]
Substituting this into Theorem \ref{control of integral points} proves the first assertion with the second summand of $C_4$.

For the height-free assertion, apply that theorem with $\epsilon/2$. Put $a=1/(d\delta^{1+1/d})$ and $H=H_K(\psi_{X,\delta})\geqslant1$ for this calculation. The maximum $\log H/H^a\leqslant1/(ae)$ gives
\[
\begin{aligned}
 \frac{h(\psi_{X,\delta})+\delta\log B+\delta^2}{H^a}
 &\leqslant\frac{d\delta^{1+1/d}}{[K:\mathbb Q]e}+\delta\log B+\delta^2\\
 &\leqslant\delta^2\left(2+\frac{d}{[K:\mathbb Q]e}\right)\max\{\log B,1\}.
\end{aligned}
\]
For every $\rho>0$,
\[
 \max\{\log B,1\}\leqslant\left(1+\frac1{e\rho}\right)B^\rho.
\]
Take $\rho=\epsilon/(2d\delta^{1+1/d})$. The resulting coefficient is exactly the third summand of $C_4$, and the two powers involving $\epsilon/2$ add to $\epsilon/(d\delta^{1+1/d})$.
Finally,
\[
 D_{\epsilon/2}+1\leqslant\frac{2(d+2)^2}{\epsilon}\delta^2
 \leqslant\frac{2(d+2)^2}{\epsilon}\delta^{4-1/d}
 B^{1/\delta^{1/d}+\epsilon/(d\delta^{1+1/d})},
\]
which gives the first summand of $C_4$ and completes the proof.
\end{proof}
\subsubsection{}
Next, we consider the case of dimensions $1$ and $2$, which can be applied directly later.
\begin{coro}\label{integral points in affine curve}
Let $X'\subset\mathbb A^n_K$ be a geometrically integral curve whose projective closure $X$ is a plane curve of degree $\delta\geqslant2$, and suppose $\psi_{X,\delta}\ne0$. For $B\geqslant2$ and $0<\epsilon<1$,
\[
 \#s(X';B)\leqslant\delta(D_\epsilon+1)+C_4(n,1,\delta,K,\epsilon)\delta^3
 \frac{B^{1/\delta+\epsilon/\delta^2}\max\{\log B,\delta\}}
 {H_K(\psi_{X,\delta})^{1/\delta^2}}.
\]
Also,
\[
 \#s(X';B)\leqslant C_4(n,1,\delta,K,\epsilon)\delta^4
 B^{1/\delta+\epsilon/\delta^2}.
\]
\end{coro}
\begin{proof}
Apply Corollary \ref{improved version of the control of integral points} with $d=1$. Its auxiliary hypersurface does not contain $X'$. After projective closure, B\'ezout's theorem bounds the number of intersection points by $\delta$ times its degree. Multiplying each of the two degree bounds by $\delta$ gives the stated formulas.
\end{proof}
\begin{coro}\label{integral points in affine surface}
Let $X'\subset\mathbb A^n_K$ be a geometrically integral surface whose projective closure is a hypersurface of degree $\delta\geqslant2$ in a linear subvariety. Suppose $\psi_{X,\delta}\ne0$. For $B\geqslant2$ and $0<\epsilon<1$, the set $s(X';B)$ is contained in an affine hypersurface not containing $X'$, of degree
\[
 \omega\leqslant C_4(n,2,\delta,K,\epsilon)\delta^{7/2}
 B^{1/\sqrt\delta+\epsilon/(2\delta^{3/2})}.
\]
\end{coro}
\begin{proof}
Take $d=2$ in the height-free assertion of Corollary \ref{improved version of the control of integral points}; then $4-1/d=7/2$ and $d\delta^{1+1/d}=2\delta^{3/2}$.
\end{proof}
\begin{rema}
The exponent of the leading-form height in Corollary \ref{integral points in affine curve} is the feature needed for the conic summation in Section 7. The logarithmic term is retained in the height-sensitive estimate, whereas the height-free estimate absorbs it into the displayed power involving $\epsilon$.
\end{rema}

\section{Some preliminaries}\label{chap.5}
In general, it is difficult to study the uniform upper bound of $\#S(X;B)$ for arbitrary projective varieties $X$. However, using specific methods, one can often reduce the general case to special ones, and these special cases constitute an important subject of study. For some special cases, the related geometric and arithmetic properties will be involved in the study. 

This section provides preliminary results needed for our further study of uniform upper bounds for rational points of bounded height.

\subsection{Reduction to the case of affine surfaces with NCC condition}
In this part, we reduce our subject to a case of lower dimension with certain conditions. 
\subsubsection{}
For an arbitrary projective variety $X$ of dimension $d<n$ and degree $\delta$ embedded in $\mathbb P^n$, one can construct a linear projection $\mathbb P^n\dashrightarrow\mathbb P^{d+1}$ such that the restriction to $X$ is finite and birational onto its image. Each fiber has at most $\delta$ geometric points by B\'ezout's theorem. Simultaneously, the change in the height of rational points on $X$ can be controlled uniformly in terms of $n$, $d$, $\delta$, and $K$. If we do not care about the dependence of the uniform upper bound of $\#S(X;B)$ on $\delta$, we can study the case of projective hypersurfaces only. 

This idea originates in \cite{Bro_HeathB_Salb}; see \cite{Salberger07,Salberger_preprint2013,CCDN2020,ParedesSasyk2022,CDHNV2023} for its further developments and applications.
\subsubsection{}
Let $f(T_0,\ldots,T_n)$ be an absolutely irreducible homogeneous polynomial of degree $\delta$ over a number field $K$. It defines a hypersurface $X'$ in $\mathbb A_0(\sE_K')$, and a hypersurface $X$ in $\mathbb P(\sE_K)$, where the Hermitian vector bundle $\overline{\sE'}=(\O_K^{n+2},(\|\ndot\|_v)_{v\in M_{K,\infty}})$ on $\spec\O_K$ is equipped with $\ell^2$-norms for all $v\in M_{K,\infty}$. In this case, we have 
\[\#S(X;B)\leqslant\#s(X';c_{n,K}B)\]
for $B\geqslant1$, with a constant $c_{n,K}$ supplied by integral-coordinate normalization, where $s(X';B)$ is from \S \ref{affine open subset of P(E)}. 

By the above argument, we only need to consider the case of geometrically integral hypersurfaces in $\mathbb A_0(\sE'_K)$ whose projective closure in $\mathbb P(\sE'_K)$ is of degree $\delta$. 
\subsubsection{}
In fact, not all affine hypersurfaces in $\mathbb A^{n+1}_K$ are relevant to this study. In \cite{CDHNV2023}, the authors clarify which kinds of hypersurfaces need to be taken into consideration meaningfully. For a self-contained treatment, we summarize parts of \cite{CDHNV2023}, particularly the arguments in \cite[\S 4]{CDHNV2023} concerning global fields.
\begin{defi}[NCC condition]\label{NCC condition}
Let $f\in K[T_1,\ldots,T_n]$ be a polynomial over a number field $K$, where $n\geqslant3$, and $X'$ be the affine hypersurface in $\mathbb A^n_K$ defined by $f$. The polynomial $f$ is called cylindrical over a curve if there exist a $K$-linear morphism \[\ell:\mathbb A^n_K\rightarrow\mathbb A^2_K\] and a curve $C$ in $\mathbb A^2_K$ such that $X'=\ell^*(C)$. We abbreviate not cylindrical over a curve by \textit{NCC} and say that $X'$ is NCC if $f$ is.
\end{defi}
Note that $f\in K[T_1,\ldots,T_n]$ satisfies the NCC condition over $K$ if there do not exist $K$-linear forms $\ell_1(T_1,\ldots,T_n),\ell_2(T_1,\ldots,T_n)$ and a polynomial $g\in K[T_1,T_2]$ such that
\[f(T_1,\ldots,T_n)=g\left(\ell_1(T_1,\ldots,T_n),\ell_2(T_1,\ldots,T_n)\right).\]
\begin{defi}[$r$-irreducible]\label{r-irreducible}
Let $K$ be a number field, $r\geqslant1$ an integer, and $f\in K[T_1,\ldots,T_n]$. The polynomial $f$ is said to be $r$-irreducible over $K$ (or $r$-irreducible if there is no ambiguity on the base field), if $f$ does not have any nonconstant factors of degree smaller than or equal to $r$ over $K$.
\end{defi}
We recall the following result from \cite[Corollary 4.20]{CDHNV2023}, which allows us to cut the higher dimensional affine hypersurfaces with hyperplanes while preserving NCC and $2$-irreducibility via the effective Hilbert irreducibility results in \cite[\S 4.2]{CDHNV2023}. For a polynomial $f\in\O_K[T_1,\ldots,T_n]$, we denote by $V(f)$ the closed subscheme of $\mathbb A^n_K$ defined by $f$. 
\begin{prop}\label{reducing to lower dimensional case}
Let $f\in\O_K[T_1,\ldots,T_n]$ be a polynomial of degree $\delta\geqslant2$, where $n\geqslant4$. Assume that $f$ is NCC, and its highest homogeneous part $f_\delta$ is $r$-irreducible for some integer $r\geqslant 1$. Then there exist linearly independent linear forms $\ell,\ell'$ of height at most $O_n(\delta^3)$ and $t\in[O_n(\delta^3)]_{\O_K}$, such that $f_{\delta}|_{V(\ell-t\ell')}$ is $r$-irreducible, and such that there are at most $O_n(\delta)$ values $b\in K$ for which $f|_{V(\ell-t\ell'-b)}$ is not NCC. 
\end{prop}
\subsubsection{}
For the dimension-reduction argument, let $K$ be a number field, and $f(t_1,t_2,t_3)$ be a polynomial of degree $\delta$. We suppose that the homogeneous degree $\delta$ part $f_\delta$ of $f$ is $2$-irreducible in Definition \ref{r-irreducible}, and $f$ satisfies the NCC condition in Definition \ref{NCC condition}. Let $X'$ be the hypersurface defined by $f$ in $\mathbb A^3_K$. By Proposition \ref{reducing to lower dimensional case}, for fixed $\theta\geqslant1$, if we prove
\[\#s(X';B)\ll_{K,\delta,\epsilon}B^{\theta+\epsilon}\]
uniformly for all such $f$ and for all $\epsilon>0$, then for any hypersurface $X''$ of degree $\delta$ in $\mathbb A^n_K$, which satisfies the NCC condition and whose highest homogeneous part is $2$-irreducible, we have
\[\#s(X'';B)\ll_{K,n,\delta,\epsilon}B^{\theta+n-3+\epsilon}.\]
At each induction step, the bounded-height change of coordinates in Proposition \ref{reducing to lower dimensional case} yields $O_{K,n,\delta}(B)$ slices. All but $O_{n,\delta}(1)$ of them satisfy NCC and retain $2$-irreducibility, so the induction hypothesis applies. The exceptional slices contribute $O_{K,n,\delta}(B^{n-2})$ by the Schwartz--Zippel bound (see \cite[Proposition 2.4 and \S 4.4]{CDHNV2023}), which is absorbed in the displayed estimate since $\theta\geqslant1$. In particular, if we can prove a uniform bound for the surfaces in $\mathbb A^3_K$ without the $B^\epsilon$ factor, then the same holds for general hypersurfaces.

For $\delta=3$, $1$-irreducibility implies $2$-irreducibility over the same field. The affine conic estimates below are applied under the stronger hypothesis that $f_3$ is geometrically integral, and to non-ruled projective closures. 

\subsection{Review of the classification of cubic surfaces}
In this part, we list some useful geometric properties of cubic surfaces in $\mathbb P_k^3$ over an algebraically closed field $k$ of characteristic $0$. We refer the readers to \cite[Chapter 9]{Dolgachev_book} for a systematic introduction to this subject.
\subsubsection{}\label{summary: geometry of non-ruled cubic surface}
Let $X$ be an integral cubic surface in $\mathbb P^3_k$. If $X$ is not a ruled surface, then $X$ is normal, and it cannot be the cone of a curve. In addition, $X$ contains at most $27$ lines in $\mathbb P^3_k$. We refer the readers to \cite[\S 9.2]{Dolgachev_book} for the proof of this property, and more precise classification of normal cubic surfaces in $\mathbb P^3_k$.

We denote by $\Hilb^{2t+1}(X)$ the Hilbert scheme of conics in $\mathbb P^3_k$ lying in $X$, since it is well known that the Hilbert polynomial of such a conic is $2t+1\in\mathbb Q[t]$. Each conic is obtained from an intersection of $X$ and a plane containing a line in $X$ (cf. \cite[\S 1.b]{Harris1982}). For counting geometrically integral conics, the smooth-conic locus is covered by at most $27$ residual pencils, each with an open parameter space in $\mathbb P^1_{\overline k}$. We only use this description of the smooth locus below.

\subsubsection{}\label{classification of skew cubic}
\leavevmode Let $X$ be an integral ruled cubic surface in $\mathbb P^3_k$. If $X$ is not a cone, it is called a skew ruled cubic surface.

By \cite[\S 9.2]{Dolgachev_book}, if $X$ is not a cone of a plane curve, it must not be normal. In addition, its reduced singular locus is a line in $\mathbb P^3_k=\proj\left(k[T_0,T_1,T_2,T_3]\right)$, which is of dimension $1$.

If we suppose the singular line is $T_0=T_1=0$, along which the surface has multiplicity $2$, then $X$ is projectively equivalent to either
\[T_0^2T_2+T_1^2T_3=0\]
or
\[T_2T_0T_1+T_3T_0^2+T_1^3=0.\]
\subsubsection{}
For the smooth-conic contribution, geometrically integral cubic cones can be omitted. Indeed, over an algebraic closure, a plane through the vertex cuts such a cone in a union of ruling lines, possibly with multiplicities, whereas a plane not through the vertex cuts out a copy of the integral cubic base curve. Neither section contains an integral conic. Thus only non-ruled and skew ruled projective closures need be considered for this contribution.
\subsection{Geometry of conics in a cubic surface}\label{geometry of hilbert scheme of cubic}
We work over a field $k$ of characteristic $0$. Throughout this and the next two sections, "non-ruled" means not swept out by lines over $\overline k$ in the given embedding in $\mathbb P^3$; it does not mean geometrically non-birational to a ruled surface. In the counting statements, a conic is geometrically integral and defined over the ground field. Degenerate plane conics are retained when compactifying a family.
\subsubsection{}
Let $V$ be a $k$-vector space of dimension $4$. Write $\mathcal H=\Hilb^{2t+1}(\mathbb P(V))$ for the projective Hilbert scheme of plane conics, including reducible and non-reduced conics, and $\mathcal H^{\rm sm}$ for its open locus of smooth conics. Let $r:\mathcal H\to\mathbb P(V^\vee)$ record the plane of a conic, and let $S$ be the rank $3$ universal subbundle on the space of planes. Thus $\mathcal H$ is the $\mathbb P^5$-bundle of one-dimensional subspaces of $\sym^2(S^\vee)$. In this description we use the convention that the tautological subbundle is $A^{-1}\subset r^*\sym^2(S^\vee)$, with $A=\O_{\mathcal H}(1)$. Put $Q=\det(S^\vee)$; after fixing the determinant of $V$, this is $\O_{\mathbb P(V^\vee)}(1)$.
\subsubsection{}\label{geometry of Hilbert scheme in X}
Let $\pi:\mathscr E\to\mathcal H$ be the universal family, let $\tau:\mathscr E\to\mathbb P(V)$ be evaluation, and put $M=\tau^*\O_{\mathbb P(V)}(1)$. The relative dimension of $\pi$ is $1$, and its Deligne pairing is the line bundle $N=\langle M,M\rangle_\pi$ on $\mathcal H$.
\begin{prop}\label{universal conic Deligne pairing}
With the above notation, there is an isomorphism
\[
 N\simeq A\otimes r^*Q^{\otimes2}.
\]
For every geometric point $s$ of $\mathcal H$, base change gives
\[
 N|_s\simeq\langle M|_{\mathscr E_s},M|_{\mathscr E_s}\rangle,
 \qquad \deg_{\mathscr E_s}(M|_{\mathscr E_s})=2.
\]
The fibre $N|_s$ is a one-dimensional vector space.
\end{prop}
\begin{proof}
The universal quadratic equation gives
\[
 \pi_*M=r^*S^\vee,\qquad
 0\longrightarrow A^{-1}\longrightarrow r^*\sym^2(S^\vee)
 \longrightarrow\pi_*M^{\otimes2}\longrightarrow0.
\]
Also $\pi_*\O_{\mathscr E}=\O_{\mathcal H}$, and the higher direct images involved here vanish, including for degenerate conics. The determinant-of-cohomology description of the Deligne pairing gives
\[
 N\simeq\det(\pi_*M^{\otimes2})\otimes
          \det(\pi_*M)^{-2}.
\]
Since $\det(\sym^2 S^\vee)=(\det S^\vee)^4$, the asserted formula follows. Compatibility with base change gives the identity on $s$, and the degree of $M|_{\mathscr E_s}$ is $2$ by the Hilbert polynomial. See \cite{MG2000,SLi2024} for these properties of the pairing.
\end{proof}
\subsubsection{}\label{geometry of Hilbert scheme of cubic}
Let $X\subset\mathbb P(V)$ be a geometrically integral non-ruled cubic surface. Over $\overline k$, its smooth conics are parameterized by at most $27$ pencils: the plane of a conic meets $X$ in that conic and a unique residual line. Conversely, fixing a line $L\subset X$, the planes through $L$ give a flat family of residual plane conics. We use only those pencils whose generic member is smooth. Their smooth loci are non-empty open subsets of $\mathbb P^1_{\overline k}$. We use this description of the smooth-conic locus.

A $k$-defined smooth conic has a $k$-defined plane and residual line. Thus, for counting such conics, we need only the pencils belonging to $k$-defined lines. For such a line, let $T=\mathbb P^1_k$ be the full pencil, $j_L:T\to\mathcal H$ its classifying morphism, $p_T:E_T\to T$ the residual family, and $q:E_T\to X$ evaluation. Put $N_T=j_L^*N$.

We compute $j_L^*N$ by Proposition \ref{universal conic Deligne pairing}. Choose homogeneous coordinates $T_0,T_1,T_2,T_3$ such that $L=V(T_0,T_1)$. Since the defining equation $F$ of $X$ vanishes on $L$, we may write
\[
 F=T_0f_0+T_1f_1,
\]
where $f_0,f_1$ are homogeneous quadratic forms. Parameterize the planes through $L$ by
\[
 T=\mathbb P^1_k,\qquad
 H_{[u:v]}=V(vT_0-uT_1).
\]
Put $S_T=(r\circ j_L)^*S$. The universal plane equation gives an exact sequence
\[
 0\longrightarrow S_T
 \longrightarrow V^\vee\otimes_k\mathcal O_T
 \xrightarrow{\,vT_0-uT_1\,}\mathcal O_T(1)
 \longrightarrow0.
\]
Taking determinants, and retaining the fixed trivialization of $\det V$, we obtain
\[
 (r\circ j_L)^*Q=\det(S_T^\vee)
 \simeq\mathcal O_T(1).
\]

We next compute $j_L^*A$. Restriction of $f_0,f_1$ to the universal plane defines a homomorphism
\[
 \mathcal O_T(-1)\longrightarrow\operatorname{Sym}^2(S_T^\vee),
 \qquad
 (u,v)\longmapsto
 \left.(uf_0+vf_1)\right|_{H_{[u:v]}}.
\]
Its fibre at $[u:v]$ is the quadratic equation of the residual conic. Indeed, on this plane we may write $x_0=uz$, $x_1=vz$, where $z=0$ defines $L$, and then
\[
 F|_{H_{[u:v]}}
 =z\left.(uf_0+vf_1)\right|_{H_{[u:v]}}.
\]
This quadratic equation is nonzero at every geometric point of $T$: otherwise the corresponding plane would
be contained in $X$, contrary to geometric integrality.

Consequently, the displayed homomorphism identifies $\mathcal O_T(-1)$ with the line subbundle of residual
conic equations. By the convention $A^{-1}\subset r^*\operatorname{Sym}^2(S^\vee)$, this is
precisely $j_L^*A^{-1}$. Hence
\[
 j_L^*A\simeq\mathcal O_T(1).
\]
Proposition \ref{universal conic Deligne pairing} now gives
\begin{equation}\label{line bundle N=2O(1)}
 N_T=j_L^*N
 \simeq j_L^*A\otimes(r\circ j_L)^*Q^{\otimes2}
 \simeq\mathcal O_T(1)\otimes\mathcal O_T(2)
 \simeq\mathcal O_T(3).
\end{equation}
This calculation holds on the full pencil, including the parameters corresponding to degenerate conics.

\subsection{Deligne pairing of arithmetic intersections}
We compare the height of an embedded conic with that of its corresponding point of the Hilbert scheme. This comparison uses the same universal pairing, independently of the surface containing the conic.
\subsubsection{}
Let $K$ be a number field, and let $\overline{\sE}$ be the rank $4$ Hermitian vector bundle of \S \ref{O{n+1} with l2-norm}. Use the preceding universal construction over $\spec\O_K$, retaining the same notation for the resulting models, with evaluation $\tau:\mathscr E\to\mathbb P(\sE)$. Equip $M$ with the metric pulled back from $\overline{\O(1)}$, and equip $N$ with the induced Deligne metric. Write these metrized line bundles as $\overline M$ and $\overline N$.
\subsubsection{}\label{Deligne pairing of arithmetic intersection}
\leavevmode A point $s\in\mathcal H(K)$ extends uniquely to an $\O_K$-section by properness. The pulled-back universal family is the closure of $\mathscr E_s$ in $\mathbb P(\sE)$. In the following formula, $\overline N|_s$ denotes pullback along this section. Compatibility of the arithmetic Deligne pairing with base change gives
\[
 h_{\overline{\O(1)}}(\mathscr E_s)
 =h_{\overline M}(\mathscr E_s)
 =\adeg_{\mathrm n}(\overline N|_s)
 =h_{\overline N}(s).
\]
We use the unnormalized-in-geometric-degree height of a curve, as in the earlier arithmetic intersection convention; the normalization by $[K:\mathbb Q]$ is retained. See \cite[\S 4.2.1, Theorem 4.2.3 and Lemma 4.4.3(1)]{YuanZhang2023}. For a reducible conic or a double line, the same identity follows by additivity over the components with their multiplicities, as in \cite[Lemma~4.2.2]{YuanZhang2023}.
\subsubsection{}
Fix a $K$-defined line $L$ on a non-ruled cubic surface $X$, a parametrization of its pencil $T=\mathbb P^1_K$, and the standard logarithmic height $h$ on $T$. By \eqref{line bundle N=2O(1)} and the height machine, there is a constant $C(X,L,K)$, also depending on these fixed choices, such that
\begin{equation}\label{compare the height of a conic and its point in hilbert scheme}
 \left|h_{\overline{\O(1)}}(\mathscr E_{j_L(t)})-3h(t)\right|
 \leqslant C(X,L,K)
 \qquad(t\in T(K)).
\end{equation}
The constant may depend on the pencil. For uniform estimates we count rational points on its image in the fixed space of Chow coefficients.

\section{Rational and integral points in conics}\label{chap.6}
In \cite{Salberger2008}, the height of points in the Hilbert scheme of conics in a surface of arbitrary degree was considered, and one counted rational points of bounded height in the complement of lines in the surfaces. Some of these techniques are improved in \cite{Salberger2015,Salberger_preprint2013} to refine the previous work. In this section, we use the techniques of Cayley forms and Arakelov geometry to study these bounds, in particular, for the case of non-ruled cubic surfaces. 
\subsection{Cayley forms of a complete intersection}
In this part, we describe the structure of the Cayley form of conics in a non-ruled cubic surface.
\subsubsection{}
Let $K$ be a number field, $V$ be a $K$-vector space of dimension $4$, and $X$ be a geometrically integral non-ruled cubic surface in $\mathbb P(V)$. In this case, $X$ contains at most $27$ lines over $\overline K$. We consider the case where $X$ contains at least one $K$-rational line; otherwise, by B\'ezout's theorem, $X$ contains no smooth conic over $K$.

In the following argument, we fix a $K$-rational line $L$ in $X$ whose residual pencil has a smooth generic conic. Let $\ell_1,\ell_2\in V$ be two $K$-linearly independent linear forms defining this line. For every nonzero tuple $(t_1,t_2)\in K^2$, the hyperplane defined by $t_1\ell_1+t_2\ell_2$ contains $L$. We denote this hyperplane in $\mathbb P(V)$ by $H_{t_1,t_2}$.

We consider the intersection of $X$ and $H_{t_1,t_2}$ in $\mathbb P(V)$. As $t_1$ and $t_2$ vary, the intersection cycle consists of $L$ together with a residual conic, possibly two lines or a double line. We denote this residual conic by $C_t$, where $t=[t_1:t_2]$. Thus $H_{t_1,t_2}\cdot X=L+C_t$ as cycles. By the above construction, the line $L$ is independent of the choice of $t_1,t_2$.
\subsubsection{}\label{first property of the cayley form of a conic in a cubic}
We consider the Cayley form of the intersection cycle $H_{t_1,t_2}\cdot X$ in $\mathbb P(V)$. Its factors correspond to the fixed line and the residual conic, including the degenerate conics. Let $T_0,T_1,T_2,T_3$ be the projective coordinates, and let $s_0,s_1,s_2,s_3$ be the corresponding dual basis, as in \S \ref{Chap. Cayley form}. The Cayley forms are homogeneous in the Pl\"ucker variables $(s_i\wedge s_j)_{0\leqslant i<j\leqslant3}$.

The Cayley form is intrinsically a section on the Grassmannian of lines. In this section let $p_{ij}=s_i\wedge s_j$, and use its homogeneous coordinate ring
\[
 R=K[p_{01},p_{02},p_{03},p_{12},p_{13},p_{23}]/(F),\qquad
 F=p_{01}p_{23}-p_{02}p_{13}+p_{03}p_{12}.
\]
For coefficient heights, fix once and for all over $\mathbb Q$ the linear section of $R_2$ obtained by eliminating the monomial $p_{01}p_{23}$ using $F$. Thus a quadratic Cayley form has $21$ coefficients with one fixed zero coordinate. This convention is independent of $X$. Any two fixed linear sections give uniformly comparable heights in this fixed degree and ambient dimension. We use $\psi_C$ and $H_K(\psi_C)$ with this convention throughout this section.

We denote by $\psi_{X,\text{line}}$ the Cayley form of $L$, and by $\psi_{X,\text{conic}}^{t_1,t_2}$ that of the residual conic $C_t$. Multiplicativity of Chow forms gives, up to a scalar,
\begin{equation}\label{Cayley form of a line and a conic}
 \psi_{H_{t_1,t_2}\cdot X}=\psi_{X,\text{line}}\psi_{X,\text{conic}}^{t_1,t_2}\quad\hbox{in }R_3.
\end{equation}
The factor $\psi_{X,\text{line}}$ is linear and $\psi_{X,\text{conic}}^{t_1,t_2}$ has degree $2$ in the Pl\"ucker variables. The former is independent of $t_1,t_2$, since $L$ is fixed. From the construction by the resultant, the Cayley form of $H_{t_1,t_2}\cdot X$ has coefficients homogeneous of degree $3$ in $t_1,t_2$. We determine the degree of the residual coefficients below. Equalities of Cayley forms are interpreted in $R$, before choosing the fixed coefficient representatives.
\subsubsection{}
For the affine application, we suppose that the degree $3$ part in $T_1,T_2,T_3$ of a homogeneous equation $f(T_0,T_1,T_2,T_3)$ of $X$ is absolutely irreducible. Put $H_\infty=V(T_0)$ and $X_\infty=X\cap H_\infty$. We then have the following property.
\begin{lemm}\label{r-irreducibility -> no line in T_0=0}
If $X_\infty$ is a geometrically integral plane cubic, no line on $X$ is contained in $H_\infty$.
\end{lemm}
\begin{proof}
If a line on $X$ lies in the plane $T_0=0$, the intersection of $X$ with that plane consists of this line and a possibly degenerate conic. Thus $f(0,T_1,T_2,T_3)$ is reducible over $\overline K$. This is exactly the degree $3$ part in the variables $T_1,T_2,T_3$, contradicting the hypothesis.
\end{proof}
\subsection{Heights of Cayley forms}\label{height and degree of coefficients of Cayley form of conic in cubic}
In \S \ref{first property of the cayley form of a conic in a cubic}, we give a description of the Cayley form of a conic in a non-ruled cubic surface in $\mathbb P^3$. In this part, we will give a more precise description with the help of the height estimate in \eqref{compare the height of a conic and its point in hilbert scheme}. Since we will apply arithmetic properties, we work over the number field $K$.

\subsubsection{}
Let $\overline{\sE}$ be the Hermitian vector bundle defined in \S \ref{O{n+1} with l2-norm}, of rank $4$ over $\spec\O_K$, and let $X$ be a non-ruled cubic surface embedded in $\mathbb P(\sE_K)$. We consider the dependence of $\psi_{X,\text{conic}}^{t_1,t_2}$ on $t_1,t_2$ under this arithmetic setting. We retain the full residual pencil, including its degenerate conics.
\begin{prop}\label{degree of the cayley form of line and conic}
The Cayley forms $\psi_{X,\text{conic}}^{t_1,t_2}$ can be represented by coefficients that are homogeneous binary cubics in $t_1,t_2$, with no common zero in $\mathbb P^1_{\overline K}$.
\end{prop}
\begin{proof}
By \S \ref{first property of the cayley form of a conic in a cubic}, the Cayley form $\psi_{X,\text{conic}}^{t_1,t_2}$ is homogeneous of degree $2$ in the Pl\"ucker variables. To determine the degree of its coefficients in $t_1,t_2$, we use the full flat family over $T=\mathbb P^1_K$.

Let $s=j_L(t)$ be the corresponding point of the Hilbert scheme. The pullback of the hyperplane line bundle in the space of Chow coefficients is the Deligne pairing $N_T=j_L^*N$. Indeed, both line bundles have the incidence divisor of conics meeting a fixed general line as the divisor of a section, by the definition of the Chow form and the divisor formula for the Deligne pairing, respectively. By \eqref{line bundle N=2O(1)},
\[
 N_T\simeq\O_{\mathbb P^1_K}(3).
\]
The Chow morphism is defined at every point of the full pencil, including its degenerate fibres. Its coordinate sections therefore generate $N_T$ everywhere. After fixing this isomorphism, they are homogeneous binary cubics with no common zero in $\mathbb P^1_{\overline K}$. Passing to the fixed coefficient representative in $R_2$ preserves these properties.

This is also compatible with the height comparison used here. By Proposition \ref{compare arakelov height and chow form height},
\[-\tfrac12\log21-2\mathcal H_3 \leqslant h(\psi_{X,\text{conic}}^{t_1,t_2})-h_{\overline N}(s) \leqslant12\log2+8\log6-\mathcal H_5.
\]
Together with \eqref{compare the height of a conic and its point in hilbert scheme}, this gives
$h(\psi_{X,\text{conic}}^{t_1,t_2})=3h([t_1:t_2])+O_{X,L,K}(1)$.
The last constant includes the fixed choice of coordinates on the pencil, and is not used as a uniform counting constant.
\end{proof}
By Proposition \ref{degree of the cayley form of line and conic}, we may write $\psi_{X,\text{conic}}^{t_1,t_2}$ as follows. Let $I=(i_1,i_2)$, $J=(j_1,j_2)$ have increasing entries, and take $I\leqslant J$ in lexicographic order. Then
\begin{equation}\label{def of binary coefficients in cayley form}
 \psi_{X,\text{conic}}^{t_1,t_2}=\sum_{I\leqslant J}b_{IJ}(t_1,t_2)p_Ip_J.
\end{equation}
There are $21$ displayed coordinates, and each $b_{IJ}$ is a binary cubic (possibly zero).
\subsection{Counting rational points}
By \eqref{def of binary coefficients in cayley form}, we can control the height of conics in a cubic surface. This description is useful to control the rational points of bounded height lying on conics.
\subsubsection{}
We have the following property of these $b_{IJ}(t_1,t_2)$.
\begin{prop}\label{all cayley form non zero}
The coefficients $b_{IJ}$ have no common zero in $\mathbb P^1_{\overline K}$.
\end{prop}
\begin{proof}
For every geometric parameter $[t_1:t_2]$, the intersection $H_{t_1,t_2}\cdot X$ has a nonzero residual cycle $C_t$ of degree $2$. There is a line in $\mathbb P^3_{\overline K}$ disjoint from this cycle. Its Cayley form is therefore not identically zero. The coefficient sections supplied by Proposition \ref{degree of the cayley form of line and conic} consequently cannot vanish simultaneously, including at the degenerate fibres.
\end{proof}
\subsubsection{}\label{parameterizing a family of cayley forms}
Put $x_1=t_1^3$, $x_2=t_1^2t_2$, $x_3=t_1t_2^2$, $x_4=t_2^3$. For every pair $I\leqslant J$, we write $b_{IJ}(x_1,x_2,x_3,x_4)=b_{IJ}(t_1,t_2)$, regarding the former as a linear form in $x_1,x_2,x_3,x_4$. In addition, let
\begin{equation}\label{embedding P1 into P2}
 \iota_1:\mathbb P^1_K\longrightarrow\mathbb P^3_K,
 \qquad[t_1:t_2]\longmapsto[t_1^3:t_1^2t_2:t_1t_2^2:t_2^3],
\end{equation}
be the cubic Veronese morphism, and let
\begin{equation}\label{embedding P2 into P20}
 \iota_2:\mathbb P^3_K\dashrightarrow\mathbb P^{20}_K,
 \qquad[x_1:x_2:x_3:x_4]\longmapsto[b_{IJ}(x_1,x_2,x_3,x_4)]_{I\leqslant J}
\end{equation}
be the corresponding linear rational map. By Proposition \ref{all cayley form non zero}, the composition is everywhere defined:
\begin{equation}\label{embedding P1 into P20}
 \iota=\iota_2\circ\iota_1:\mathbb P^1_K\longrightarrow\mathbb P^{20}_K,
 \qquad t\longmapsto[\psi_{X,\text{conic}}^{t_1,t_2}].
\end{equation}
We have the following property of the coefficient matrix.
\begin{lemm}\label{rank of the matrix from all coefficients of cayley form}
The coefficient matrix of the $b_{IJ}$ in the four cubic monomials has rank $3$ or $4$.
\end{lemm}
\begin{proof}
A smooth conic determines its plane, hence its parameter in the fixed pencil. Thus $\iota$ is generically injective and, in characteristic $0$, has degree $1$ onto its image. Since $\iota^*\O(1)=\O(3)$, its image has degree $3$. The rank is at most $4$ and cannot be at most $2$, since a nonconstant image spanning at most a line would have degree $1$.
\end{proof}
By the preceding lemma, we can describe the image of $\iota$ more precisely.
\begin{prop}\label{image of coefficients of cayley form}
The image $\im(\iota)$ is a geometrically integral cubic curve. For rank $3$ it is a plane cubic, and for rank $4$ it is a twisted cubic in its linear span. The morphism $\iota$ is its normalization and has at most $3$ geometric points in any fibre.
\end{prop}
\begin{proof}
By Lemma \ref{rank of the matrix from all coefficients of cayley form}, the map $\iota$ is generically injective, and its image is a geometrically integral curve of degree $3$.

If the coefficient matrix has rank $4$, its four independent linear forms identify the span of the image with $\mathbb P^3_K$. Thus $\iota_2$ embeds the cubic Veronese curve $\im(\iota_1)$ into $\mathbb P^{20}_K$ by a linear map, and $\im(\iota)$ is a twisted cubic in this span.

If the matrix has rank $3$, the image spans a plane in $\mathbb P^{20}_K$ and is a geometrically integral plane cubic. In both cases, the nonconstant proper morphism from $\mathbb P^1_K$ to its image is finite. Since it is birational and $\mathbb P^1_K$ is normal, it is the normalization. A hyperplane through a point of the image, not containing the image, pulls back to a divisor of degree $3$ containing its fibre. Hence every fibre has at most $3$ geometric points.
\end{proof}
\subsubsection{}
By Proposition \ref{image of coefficients of cayley form}, we can estimate the conics in $X$ parameterized by one irreducible component of the smooth-conic locus of the Hilbert scheme, with the height of $\psi_{X,\text{conic}}^{t_1,t_2}$ bounded. More precisely, we study the set
\[
 \left\{[t_1:t_2]\in\mathbb P^1(K)\mid
 H_K\left(\psi_{X,\text{conic}}^{t_1,t_2}\right)\leqslant B\right\},
\]
where the naive height of a polynomial is defined in \eqref{naive height of a polynomial or hypersurface}. We have the following estimate.
\begin{prop}\label{uniform upper bound of cayley form in degree 2 curve}
For every such pencil, $B\geqslant1$ and $0<\epsilon<1$,
\[ \#\left\{t\in\mathbb P^1(K)\mid H_K(\psi_{X,\text{conic}}^{t_1,t_2})\leqslant B\right\}\leqslant243\cdot4^{[K:\mathbb Q](2/3+\epsilon/9)}
 C'_1(20,1,K,\epsilon)B^{2/3+\epsilon/9},
\]
independently of the surface, the line, and the chosen coordinates on the pencil.
\end{prop}
\begin{proof}
In rank $3$, select three independent coefficient coordinates of $\iota$. They define a plane cubic, with height at most the original coefficient height. The resulting map from $\mathbb P^1$ has fibres of size at most $3$.

In rank $4$, select four independent coefficient coordinates. They give a twisted cubic in $\mathbb P^3$, again without increasing height. We give the projection explicitly to control its effect on height. A nonzero quadric vanishes on this twisted cubic. Since a polynomial of degree at most $2$ in each variable cannot vanish on all of $\{0,1,2\}^4$ unless it is zero, choose a point $P=[P_0:P_1:P_2:P_3]$ outside the curve with $P_i\in\{0,1,2\}$. Choose $j$ with $P_j\ne0$ and project by the three forms
\[
 P_jx_i-P_ix_j\qquad(i\ne j).
\]
The common zero of these forms is $P$, so the projection restricts to a morphism on the curve. Its image spans a plane, and the product of its degree with the degree of that image is $3$. Thus the image is a geometrically integral plane cubic. The composite map from $\mathbb P^1$ has fibres of size at most $3$.

Each projection form has integral coefficients and sum of absolute values at most $4$. Hence the projected point has height at most $4^{[K:\mathbb Q]}$ times the original coefficient height. The plane can be embedded as a coordinate plane in $\mathbb P^{20}$ without changing heights. Apply Corollary \ref{global determinant of hypersurface in a linear subvariety without height term} with $d=1$, $\delta=3$ and height bound $4^{[K:\mathbb Q]}B\geqslant2$. The auxiliary hypersurface does not contain the plane cubic, so B\'ezout's theorem gives, in both ranks,
\[
 \#S(\text{plane image};4^{[K:\mathbb Q]}B)
 \leqslant81 C'_1(20,1,K,\epsilon)
 4^{[K:\mathbb Q](2/3+\epsilon/9)}B^{2/3+\epsilon/9}.
\]
Multiplication by the fibre bound $3$ proves the assertion.
\end{proof}
By the above result, we can give an upper bound for the number of conics with bounded height in a non-ruled cubic surface.
\begin{coro}\label{number of conics with height smaller than B}
For every geometrically integral non-ruled cubic surface $X\subset\mathbb P^3_K$, $B\geqslant1$ and $0<\epsilon<1$,
\[
\begin{aligned}
 &\#\{C\subset X:C\text{ is a smooth }K\text{-conic},\ H_K(\psi_C)\leqslant B\}\\
 &\qquad\leqslant6561\cdot4^{[K:\mathbb Q](2/3+\epsilon/9)}
 C'_1(20,1,K,\epsilon)B^{2/3+\epsilon/9}.
\end{aligned}
\]
\end{coro}
\begin{proof}
Every such conic belongs to the pencil of its $K$-defined residual line, and there are at most $27$ lines. Apply Proposition \ref{uniform upper bound of cayley form in degree 2 curve} to each relevant pencil; the coefficient is $27\cdot243=6561$.
\end{proof}
\subsubsection{}
Here we give a uniform upper bound for the rational points of bounded height provided by the conic components of the auxiliary intersection on a non-ruled cubic surface.
\begin{theo}\label{estimate of rational points in conics of cubic}
Let $\overline{\sE}$ be the Hermitian vector bundle of rank $4$ defined in \S \ref{O{n+1} with l2-norm}, and let $X$ be a geometrically integral non-ruled cubic surface in $\mathbb P(\sE_K)$. For $B\geqslant2$ and $0<\epsilon<1$, there is an auxiliary hypersurface of degree at most $3^{7/2}C'_1(3,2,K,\epsilon)B^{\sqrt3/2+\epsilon/(6\sqrt3)}$ (as in Corollary \ref{global determinant of hypersurface in a linear subvariety without height term}) containing $S(X;B)$ and not containing $X$, such that the points of $S(X;B)$ lying on its smooth $K$-conic intersection components number at most
\[
\begin{aligned}
 \Bigg[&\frac{10\cdot3^{9/2}}\epsilon C'_1(3,2,K,\epsilon)\\
 &+\frac{64}{5}
 e^{[K:\mathbb Q](3\log2+2\log6-\mathcal H_5/4)}
 C''_1(3,1,K,\epsilon)\\
 &\quad\cdot\big(6561\cdot4^{[K:\mathbb Q](2/3+\epsilon/9)}C'_1(20,1,K,\epsilon)\big)^{\frac{9}{4(6+\epsilon)}}\\
 &\quad\cdot\left(\frac{3^{9/2}}2 C'_1(3,2,K,\epsilon)\right)^{\frac{15+4\epsilon}{4(6+\epsilon)}}\Bigg]\\
 &\hspace{15mm}\cdot B^{1+5\sqrt3/16+\epsilon/2}\max\{\log B,2\}.
\end{aligned}
\]
In particular, this is $\ll_{K,\epsilon} B^{1+5\sqrt3/16+\epsilon}$; an explicit coefficient for this latter form is the displayed coefficient multiplied by $2+2/(e\epsilon)$.
\end{theo}
\begin{proof}
By Corollary \ref{global determinant of hypersurface in a linear subvariety without height term}, the set $S(X;B)$ can be covered by a hypersurface of degree at most
\[
 3^{7/2}C'_1(3,2,K,\epsilon)B^{\sqrt3/2+\epsilon/(6\sqrt 3)},
\]
which does not contain the generic point of $X$. By B\'ezout's theorem in intersection theory (cf. \cite[Proposition 8.4]{Fulton}), its intersection cycle with $X$ has at most $3/2$ times its degree distinct conic components. Thus their number is at most
\[
 \frac{3^{9/2}}2 C'_1(3,2,K,\epsilon)B^{\sqrt3/2+\epsilon/(6\sqrt 3)}.
\]
From Corollary \ref{uniform upper bound of points of bounded height in plane curves}, with $n=3$, $d=1$ and $\delta=2$, each smooth conic $C$ satisfies
\begin{equation}\label{projective conic input}
\begin{aligned}
 \#S(C;B)\leqslant{}&\frac{20}\epsilon
 +8e^{[K:\mathbb Q](3\log2+2\log6-\mathcal H_5/4)}C''_1(3,1,K,\epsilon)\\
 &\qquad\cdot\frac{B^{1+\epsilon/4}\max\{\log B,2\}}{H_K(\psi_C)^{1/4}}.
\end{aligned}
\end{equation}
Here $N(3,1)=5$, and $2(D_\epsilon+1)\leqslant20/\epsilon$ for this value of $d,\delta$.

Order the distinct smooth $K$-conic components as $C_1,C_2,\ldots$ by nondecreasing height. Corollary \ref{number of conics with height smaller than B} gives
\[
 j\leqslant6561\cdot4^{[K:\mathbb Q](2/3+\epsilon/9)}
 C'_1(20,1,K,\epsilon)H_K(\psi_{C_j})^{2/3+\epsilon/9}.
\]
Consequently,
\[
 H_K(\psi_{C_j})^{-1/4}\leqslant
 \big(6561\cdot4^{[K:\mathbb Q](2/3+\epsilon/9)}
 C'_1(20,1,K,\epsilon)\big)^{\frac{9}{4(6+\epsilon)}}
 j^{-\frac{9}{4(6+\epsilon)}}.
\]
For every integer $q\geqslant1$,
\[
\begin{aligned}
 \sum_{j=1}^q j^{-\frac{9}{4(6+\epsilon)}}
 &\leqslant1+\int_1^q u^{-\frac{9}{4(6+\epsilon)}}\,\mathrm du\\
 &\leqslant\frac{4(6+\epsilon)}{15+4\epsilon}
 q^{\frac{15+4\epsilon}{4(6+\epsilon)}}
 \leqslant\frac85q^{\frac{15+4\epsilon}{4(6+\epsilon)}}.
\end{aligned}
\]
Apply this with the actual number of components. Since $0<\epsilon<1$, we have
\[
\begin{aligned}
 \left(\frac{\sqrt3}{2}+\frac{\epsilon}{6\sqrt3}\right)
 \frac{15+4\epsilon}{4(6+\epsilon)}
 &=\frac{5\sqrt3}{16}+\frac{\epsilon}{6\sqrt3}
   +\frac{3\epsilon}{16\sqrt3(6+\epsilon)}\\
 &\leqslant\frac{5\sqrt3}{16}+\frac{\epsilon}{4}.
\end{aligned}
\]
Combining this inequality with the bound on the number of components gives
\[
\begin{aligned}
 \sum_C H_K(\psi_C)^{-1/4}
 \leqslant{}&\frac85
 \big(6561\cdot4^{[K:\mathbb Q](2/3+\epsilon/9)}
 C'_1(20,1,K,\epsilon)\big)^{\frac{9}{4(6+\epsilon)}}\\
 &\cdot\left(\frac{3^{9/2}}2 C'_1(3,2,K,\epsilon)\right)^{\frac{15+4\epsilon}{4(6+\epsilon)}}
 B^{5\sqrt3/16+\epsilon/4}.
\end{aligned}
\]
If there are no such conics, the assertion is immediate. Otherwise, substituting this bound into \eqref{projective conic input} yields the second summand in the coefficient of the theorem. The constant term $20/\epsilon$ contributes at most
\[
 \frac{10\cdot3^{9/2}}\epsilon C'_1(3,2,K,\epsilon)B^{\sqrt3/2+\epsilon/(6\sqrt3)},
\]
which is bounded by the first summand times the stated power and logarithmic factor. Finally,
\[
 \max\{\log B,2\}\leqslant
 \left(2+\frac2{e\epsilon}\right)B^{\epsilon/2},
\]
so the final asserted form follows with the explicit coefficient given in the statement.
\end{proof}

\subsection{Counting integral points}
For a non-ruled cubic surface, we consider the conics lying on it for which the height of the highest degree part of the Cayley form is bounded. In this part, we will give a uniform upper bound for the number of such conics. We assume that $X_\infty=X\cap V(T_0)$ is a geometrically integral plane cubic, or equivalently that the highest homogeneous part of an affine equation is absolutely irreducible.
\subsubsection{}
We keep all the above notations and constructions. By Proposition \ref{degree of the cayley form of line and conic}, we write $\psi_{X,\text{conic}}^{t_1,t_2}=\psi_{X,2}^{t_1,t_2}+\psi_{X,1}^{t_1,t_2}+\psi_{X,0}^{t_1,t_2}$ according to degree in $p_{01},p_{02},p_{03}$. In particular,
\begin{equation}\label{def of a_1,...,a_6}
\begin{aligned}
 \psi_{X,2}^{t_1,t_2}={}&a_1(t)p_{01}^2+a_2(t)p_{02}^2+a_3(t)p_{03}^2\\
 &+a_4(t)p_{01}p_{02}+a_5(t)p_{01}p_{03}+a_6(t)p_{02}p_{03}.
\end{aligned}
\end{equation}
All $a_i$ are binary cubics. This restriction is independent of the representative modulo $F$, since $F$ vanishes when $p_{12}=p_{13}=p_{23}=0$. It is the Chow form of the length $2$ intersection cycle $Z_t=C_t\cap H_\infty$ in $H_\infty$.
\begin{prop}\label{homogeneous deg 2 part non zero}
The six coefficients $a_i$ have no common zero in $\mathbb P^1_{\overline K}$.
\end{prop}
\begin{proof}
Suppose that there is a geometric parameter $[\lambda_1:\lambda_2]$ where all six coefficients vanish. Then every nonzero monomial of $\psi_{X,\text{conic}}^{\lambda_1,\lambda_2}$ contains at least one of $s_1\wedge s_2$, $s_1\wedge s_3$, $s_2\wedge s_3$.

We consider the Pl\"ucker coordinates of the lines in the plane $T_0=0$. Such a line is the complete intersection of $T_0=0$ and a plane $c_1T_1+c_2T_2+c_3T_3=0$, where $c_1,c_2,c_3$ are not all zero. Its coordinates satisfy
\[
 p_{12}=p_{13}=p_{23}=0,
 \qquad[p_{01}:p_{02}:p_{03}]=[c_1:c_2:c_3].
\]
Our assumption implies that the Cayley form vanishes at every such line. Thus every line in $H_\infty$ meets the residual conic. If the conic had no component in $H_\infty$, its intersection with that plane would be a finite length $2$ scheme, and a general line in the plane would avoid it. Hence a component of this conic must lie in $H_\infty$.

Such a component is a line or a conic contained in $X_\infty$, which is impossible because $X_\infty$ is a geometrically integral plane cubic. The argument includes reducible and non-reduced residual conics, and proves the assertion over $\overline K$.
\end{proof}
\subsubsection{}
Similar to \S \ref{parameterizing a family of cayley forms}, we parameterize the leading part of the Cayley form. Put $x_1=t_1^3$, $x_2=t_1^2t_2$, $x_3=t_1t_2^2$, $x_4=t_2^3$. For every $i=1,\ldots,6$, regard $a_i(t_1,t_2)$ as a linear form $a_i(x_1,x_2,x_3,x_4)$. Let $\iota_1$ be the morphism in \eqref{embedding P1 into P2}, and define
\begin{equation}\label{embedding P2 into P5}
 \iota'_2:\mathbb P^3_K\dashrightarrow\mathbb P^5_K,
 \qquad[x_1:x_2:x_3:x_4]\longmapsto[a_1(x):\cdots:a_6(x)].
\end{equation}
By Proposition \ref{homogeneous deg 2 part non zero}, its composition with $\iota_1$ is a morphism
\begin{equation}\label{embedding P1 into P5}
 \iota':\mathbb P^1_K\longrightarrow\mathbb P^5_K,
 \qquad t\longmapsto[a_1(t):\cdots:a_6(t)].
\end{equation}
Similar to Lemma \ref{rank of the matrix from all coefficients of cayley form}, we have the following result.
\begin{lemm}\label{rank of the matrix from homogeneous deg 2 part of cayley form}
The coefficient matrix of the $a_i$ in the four cubic monomials has rank $3$ or $4$.
\end{lemm}
\begin{proof}
Let $P=L\cap H_\infty$ and $\ell_t=H_t\cap H_\infty$. Then $\ell_t\cdot X_\infty=P+Z_t$ as cycles. A general $Z_t$ consists of two distinct points and determines $\ell_t$. Indeed, if $P$ is smooth, the residual degree $2$ map on the normalization of $X_\infty$ is generically separable; if $P$ is singular, a general line through $P$ gives $Z_t=P+Q_t$ with $Q_t\ne P$. Since $L\not\subset H_\infty$, the line $\ell_t$ determines $H_t$ among the planes through $L$. Hence $\iota'$ is generically injective. Its pullback of $\O(1)$ is $\O(3)$ by Proposition \ref{homogeneous deg 2 part non zero}. Its image therefore has degree $3$, which excludes ranks at most $2$; the rank is at most $4$.
\end{proof}
By Lemma \ref{rank of the matrix from homogeneous deg 2 part of cayley form}, we can determine the geometric structure of $\im(\iota')$ in \eqref{embedding P1 into P5}.
\begin{prop}\label{image of leading coefficient of cayley form of degree 2}
The image $\im(\iota')$ is a geometrically integral cubic curve, and $\iota'$ is its normalization. It is a plane cubic in rank $3$ and a twisted cubic in rank $4$. Every fibre has at most $3$ geometric points.
\end{prop}
\begin{proof}
Apply the argument of Proposition \ref{image of coefficients of cayley form}, using the preceding lemma and $\iota'^*\O(1)=\O(3)$.
\end{proof}
\subsubsection{}
By Proposition \ref{image of leading coefficient of cayley form of degree 2}, we can estimate the conics in $X$ parameterized by one pencil with the height of $\psi_{X,2}^{t_1,t_2}$ bounded. More precisely, we study the set
\[
 \left\{[t_1:t_2]\in\mathbb P^1(K)\mid
 H_K(\psi_{X,2}^{t_1,t_2})\leqslant B\right\},
\]
where the naive height is defined in \eqref{naive height of a polynomial or hypersurface}. Similar to Proposition \ref{uniform upper bound of cayley form in degree 2 curve}, we obtain the following estimate.
\begin{prop}\label{uniform upper bound of points in degree 2 curve}
For every $B\geqslant1$ and $0<\epsilon<1$,
\[\#\{t\in\mathbb P^1(K)\mid H_K(\psi_{X,2}^{t_1,t_2})\leqslant B\}\leqslant243\cdot4^{[K:\mathbb Q](2/3+\epsilon/9)}C'_1(5,1,K,\epsilon)B^{2/3+\epsilon/9},
\]
uniformly over the above pencils and surfaces.
\end{prop}
\begin{proof}
Apply the projection argument of Proposition \ref{uniform upper bound of cayley form in degree 2 curve} to the degree $3$ image $\im(\iota')\subset\mathbb P^5$. Selecting coefficient coordinates does not increase height, and the explicit projection in rank $4$ increases height by at most $4^{[K:\mathbb Q]}$. Embed the resulting plane as a coordinate plane in $\mathbb P^5$. Apply Corollary \ref{global determinant of hypersurface in a linear subvariety without height term} with $d=1$, $\delta=3$ and height bound $4^{[K:\mathbb Q]}B$, followed by B\'ezout's theorem. This gives the point-count coefficient $81\cdot4^{[K:\mathbb Q](2/3+\epsilon/9)}C'_1(5,1,K,\epsilon)$. Multiplication by the fibre bound $3$ proves the assertion.
\end{proof}
The following corollary controls all conics of bounded height in a non-ruled cubic surface in the sense studied in this part. It follows directly from Proposition \ref{uniform upper bound of points in degree 2 curve}.
\begin{coro}\label{number of conics 2-height smaller than B}
For every geometrically integral non-ruled cubic surface $X\subset\mathbb P^3_K$ with geometrically integral $X_\infty$, every $B\geqslant1$ and $0<\epsilon<1$,
\[
\begin{aligned}
 &\#\{C\subset X\mid C\text{ is a smooth }K\text{-conic},\ H_K(\psi_{C,2})\leqslant B\}\\
 &\qquad\leqslant6561\cdot4^{[K:\mathbb Q](2/3+\epsilon/9)}
 C'_1(5,1,K,\epsilon)B^{2/3+\epsilon/9}.
\end{aligned}
\]
\end{coro}
\begin{proof}
Sum the preceding bound over at most $27$ residual pencils; again $27\cdot243=6561$.
\end{proof}
\subsubsection{}
For a non-ruled cubic surface whose section at infinity is geometrically integral, we give a uniform upper bound for the integral points of bounded height provided by the conic components of the auxiliary intersection.
\begin{theo}\label{estimate of integral points in conics of cubic}
Let $\overline{\sE}$ be the Hermitian vector bundle of rank $4$ over $\spec\O_K$ defined in \S \ref{O{n+1} with l2-norm}. Let $X'\subset\mathbb A_0(\sE_K)$ have geometrically integral non-ruled cubic projective closure $X$, and suppose that $X_\infty=X\cap V(T_0)$ is geometrically integral. For $B\geqslant2$ and $0<\epsilon<1$, there is an affine auxiliary hypersurface of degree at most $3^{7/2}C_4(3,2,3,K,\epsilon)B^{1/\sqrt3+\epsilon/(6\sqrt3)}$ (as in Corollary \ref{integral points in affine surface}) containing $s(X';B)$ and not containing $X'$, such that the points of $s(X';B)$ lying on its geometrically integral conic intersection components number at most
\[
\begin{aligned}
 \Bigg[&\frac{10\cdot3^{9/2}}\epsilon C_4(3,2,3,K,\epsilon)
 +\frac{64}{5}C_4(3,1,2,K,\epsilon)\\
 &\quad\cdot\big(6561\cdot4^{[K:\mathbb Q](2/3+\epsilon/9)}C'_1(5,1,K,\epsilon)\big)^{\frac{9}{4(6+\epsilon)}}\\
 &\quad\cdot\left(\frac{3^{9/2}}2 C_4(3,2,3,K,\epsilon)\right)^{\frac{15+4\epsilon}{4(6+\epsilon)}}\Bigg]\\
 &\hspace{12mm}\cdot
 B^{1/2+5/(8\sqrt3)+\epsilon/2}\max\{\log B,2\}.
\end{aligned}
\]
Here $s(X';B)$ is the set defined in \S \ref{affine open subset of P(E)}. In particular, this is $\ll_{K,\epsilon}B^{1/2+5/(8\sqrt3)+\epsilon}$; an explicit coefficient for the latter bound is the displayed coefficient multiplied by $2+2/(e\epsilon)$.
\end{theo}
\begin{proof}
By Corollary \ref{integral points in affine surface}, the set $s(X';B)$ is covered by an affine hypersurface of degree at most
\[
 3^{7/2}C_4(3,2,3,K,\epsilon)B^{1/\sqrt3+\epsilon/(6\sqrt3)},
\]
which does not contain $X'$. Its degree means that of its projective closure in $\mathbb P(\sE_K)$. By B\'ezout's theorem (cf. \cite[Proposition 8.4]{Fulton}), the number of distinct conic components in its intersection with $X$ is at most
\[
 \frac{3^{9/2}}2 C_4(3,2,3,K,\epsilon)B^{1/\sqrt3+\epsilon/(6\sqrt3)}.
\]
For every geometrically integral conic component $C'$ and its projective closure $C$, Proposition \ref{homogeneous deg 2 part non zero} gives $\psi_{C,2}\ne0$. Corollary \ref{integral points in affine curve}, with $n=3$, $\delta=2$, gives
\begin{equation}\label{affine conic input}
 \#s(C';B)\leqslant\frac{20}\epsilon+
 8C_4(3,1,2,K,\epsilon)
 \frac{B^{1/2+\epsilon/4}\max\{\log B,2\}}{H_K(\psi_{C,2})^{1/4}}.
\end{equation}
Order these distinct conics as $C_1,C_2,\ldots$ by nondecreasing leading-part height. Corollary \ref{number of conics 2-height smaller than B} gives
\[
 j\leqslant6561\cdot4^{[K:\mathbb Q](2/3+\epsilon/9)}
 C'_1(5,1,K,\epsilon)H_K(\psi_{C_j,2})^{2/3+\epsilon/9},
\]
and hence
\[
 H_K(\psi_{C_j,2})^{-1/4}\leqslant
 \big(6561\cdot4^{[K:\mathbb Q](2/3+\epsilon/9)}
 C'_1(5,1,K,\epsilon)\big)^{\frac{9}{4(6+\epsilon)}}
 j^{-\frac{9}{4(6+\epsilon)}}.
\]
Apply the summation estimate proved in Theorem \ref{estimate of rational points in conics of cubic} to the actual number of components. Since
\[
 \left(\frac1{\sqrt3}+\frac{\epsilon}{6\sqrt3}\right)
 \frac{15+4\epsilon}{4(6+\epsilon)}
 =\frac5{8\sqrt3}+\frac{\epsilon}{6\sqrt3}
 \leqslant\frac5{8\sqrt3}+\frac{\epsilon}{4},
\]
the bound on the number of components gives
\[
\begin{aligned}
 \sum_C H_K(\psi_{C,2})^{-1/4}
 \leqslant{}&\frac85
 \big(6561\cdot4^{[K:\mathbb Q](2/3+\epsilon/9)}
 C'_1(5,1,K,\epsilon)\big)^{\frac{9}{4(6+\epsilon)}}\\
 &\cdot\left(\frac{3^{9/2}}2 C_4(3,2,3,K,\epsilon)\right)^{\frac{15+4\epsilon}{4(6+\epsilon)}}
 B^{5/(8\sqrt3)+\epsilon/4}.
\end{aligned}
\]
The case of no conic components is immediate. Otherwise, substitute this into \eqref{affine conic input}. The constant term contributes at most
\[
 \frac{10\cdot3^{9/2}}\epsilon C_4(3,2,3,K,\epsilon)
 B^{1/\sqrt3+\epsilon/(6\sqrt3)},
\]
which is smaller than the first summand of the stated bound. The weighted term gives its second summand. Finally,
\[
 \max\{\log B,2\}\leqslant\left(2+\frac2{e\epsilon}\right)B^{\epsilon/2},
\]
which gives the final asserted form with the explicit coefficient in the statement.
\end{proof}

\section{A sharper bound for the affine conic contribution}\label{sharper affine conic contribution section}
Theorem \ref{estimate of integral points in conics of cubic} gives an upper bound for the contribution of conics by combining the height distribution of the family with estimates for integral points on each conic. In this section, we improve this bound under the same hypotheses on the cubic surface and its section at infinity.

For this purpose, we distinguish conics according to their intersection with the hyperplane at infinity. We first prove that the number of smooth conics tangent to this hyperplane is uniformly bounded. For the other conics, we give an upper bound $O_{K,\epsilon}(B^\epsilon)$ with an explicit constant for the number of integral points, uniformly in the conic and its containing plane, for every $\epsilon>0$. Finally, we combine these estimates to bound the contribution of an arbitrary finite collection of smooth conics and apply the result to the auxiliary intersection from Theorem \ref{estimate of integral points in conics of cubic}.

Let $\overline{\sE}$ be the Hermitian vector bundle of rank $n+1$ over $\spec\O_K$ defined in \S \ref{O{n+1} with l2-norm}. We choose the coordinate sections $T_0,\ldots,T_n$ corresponding to the standard basis of $\O_K^{\oplus(n+1)}$, as in \S \ref{affine open subset of P(E)}, and put $H_\infty=V(T_0)\subset\mathbb P(\sE_K)$. We use the affine open subset $\mathbb A_0(\sE_K)$ and the sets $s(Y;B)$ defined there, with the same Hermitian norms at the archimedean places.

\subsection{Conics tangent to the hyperplane at infinity}
We begin with a geometric bound for the number of smooth conics on a cubic surface tangent to $H_\infty$.
\begin{prop}\label{finitely many tangent conics}
Assume $n=3$. Let $X$ be a geometrically integral non-ruled cubic surface in $\mathbb P(\sE_K)$, and suppose that $X_\infty=X\cap H_\infty$ is geometrically integral. Then $X_{\overline K}$ contains at most $162$ smooth conics tangent to $H_\infty$, which means that their intersection with $H_\infty$ is supported at one point. Each residual pencil contains at most $6$ such conics.
\end{prop}
\begin{proof}
We work over the algebraic closure $\overline K$ of $K$ and fix a residual pencil. The six coefficients $a_1,\ldots,a_6$ of $\psi_{X,2}^{t_1,t_2}$ in \eqref{def of a_1,...,a_6} are binary cubics with no common zero by Proposition \ref{homogeneous deg 2 part non zero}. Hence, the morphism $\iota'$ in \eqref{embedding P1 into P5} satisfies $\iota'^*\O_{\mathbb P^5}(1)\simeq\O_{\mathbb P^1}(3)$.

For a smooth member, the tangency of a conic to $H_\infty$ means that its intersection with $H_\infty$ is supported at one point. The form $\psi_{X,2}^{t_1,t_2}$ is the Chow form of this length $2$ intersection: it is the product of the linear forms corresponding to the intersection points, counted with multiplicities. Thus, the tangency is equivalent to $\psi_{X,2}^{t_1,t_2}$ being proportional to the square of a linear form. Such squares form the quadratic Veronese surface in $\mathbb P^5$. The image of $\iota'$ cannot be contained in this surface. Otherwise, $\iota'$ would factor through the quadratic Veronese embedding $\mathbb P^2\hookrightarrow\mathbb P^5$, whose pullback of $\O_{\mathbb P^5}(1)$ is $\O_{\mathbb P^2}(2)$. The degree of $\iota'^*\O_{\mathbb P^5}(1)$ would then be even, whereas it equals $3$.

The Veronese surface is cut out by quadrics, namely the $2\times2$ minors of the symmetric coefficient matrix. At least one of these quadrics therefore pulls back under $\iota'$ to a nonzero homogeneous polynomial of degree $6$ in $t_1,t_2$. Every smooth member tangent to $H_\infty$ gives a zero of this polynomial, so there are at most $6$ such members in the pencil. Finally, the smooth conics on $X_{\overline K}$ are covered by at most $27$ residual pencils, by \S \ref{summary: geometry of non-ruled cubic surface}. The total is therefore at most $27\cdot6=162$.
\end{proof}

\subsection{Integral points on conics with two points at infinity}
For conics not tangent to $H_\infty$, we establish the following number field form of \cite[Lemma 13]{Bro_HeathB_Salb}.

For a geometrically integral affine conic $C'$ whose projective closure $C$ meets $H_\infty$ at two distinct geometric points, our aim is to bound $\#s(C';B)$ by an explicit constant depending only on $[K:\mathbb Q]$ and $\epsilon$, multiplied by $B^\epsilon$, for every $\epsilon>0$ and $B\geqslant2$. The estimate will be uniform in the conic and its containing plane. To prove it, we first choose a plane equation with bounded coefficients and associate an algebraic integer to each point. We then count the possible principal ideals and their generators. The following lemmas provide these estimates.

\subsubsection{}
For this target, we will introduce several lemmas. 

First we construct the plane equation from five integral points.

\begin{lemm}\label{bounded equation of a non-tangent conic}
Let $\overline{\sE}$ be the Hermitian vector bundle of rank $n+1$ over $\spec\O_K$ defined in \S \ref{O{n+1} with l2-norm}. Let $C'$ be a geometrically integral curve in $\mathbb A_0(\sE_K)$ whose projective closure $C$ in $\mathbb P(\sE_K)$ is a conic, and suppose that $C\cap H_\infty$ consists of two distinct geometric points. If $B\geqslant2$ and $\#s(C';B)\geqslant5$, then we can choose affine coordinates $x,y$ on the plane containing $C'$ such that the coordinate map sends $s(C';B)$ injectively into $\O_K^2$ and $C'$ has an equation
\[
 ax^2+bxy+cy^2+ex+fy+g=0,
 \qquad a,b,c,e,f,g\in\O_K,
\]
with $a\ne0$ and $b^2-4ac\ne0$. For every point in $s(C';B)$ and every embedding $\sigma:K\hookrightarrow\mathbb C$, these coordinates and coefficients satisfy
\[
\begin{gathered}
 |\sigma(x)|\leqslant B^{1/[K:\mathbb Q]},
 \qquad |\sigma(y)|\leqslant2B^{1/[K:\mathbb Q]},\\
 \max\{ |\sigma(a)|,|\sigma(b)|,|\sigma(c)|,
 |\sigma(e)|,|\sigma(f)|,|\sigma(g)|\}
 \leqslant240B^{10/[K:\mathbb Q]}.
\end{gathered}
\]
\end{lemm}
\begin{proof}
We choose two of the standard affine coordinate functions $T_1/T_0,\ldots,T_n/T_0$ whose restrictions give an affine isomorphism from the plane containing $C'$ to $\mathbb A^2_K$. Such a choice exists because the linear parts of these coordinate functions span the dual of the direction space of the plane. This projection is injective, preserves integrality of the projected points, and does not increase their archimedean coordinate bounds. Its extension to the projective planes identifies their lines at infinity, so the two geometric points at infinity remain distinct. We may therefore work with a plane conic, with coordinates $x,y$ satisfying
\[
 |\sigma(x)|,\ |\sigma(y)|
 \leqslant B^{1/[K:\mathbb Q]}
\]
for every embedding $\sigma:K\hookrightarrow\mathbb C$.

By hypothesis, we can choose five distinct points $(x_i,y_i)\in\O_K^2$, $1\leqslant i\leqslant5$, in the prescribed region. Evaluating the six monomials $x^2,xy,y^2,x,y,1$ at these points gives the following $5\times6$ matrix:
\[
 \begin{pmatrix}
 x_1^2 & x_1y_1 & y_1^2 & x_1 & y_1 & 1\\
 x_2^2 & x_2y_2 & y_2^2 & x_2 & y_2 & 1\\
 x_3^2 & x_3y_3 & y_3^2 & x_3 & y_3 & 1\\
 x_4^2 & x_4y_4 & y_4^2 & x_4 & y_4 & 1\\
 x_5^2 & x_5y_5 & y_5^2 & x_5 & y_5 & 1
 \end{pmatrix}.
\]
A polynomial of total degree at most $2$ vanishes at all five points precisely when its coefficient vector lies in the kernel of this matrix. If the matrix had rank less than $5$ over $K$, its kernel would have dimension at least $2$. There would then be a nonzero polynomial of total degree at most $2$, not proportional to the defining equation of the conic, which also vanishes at all five points. Since the conic is geometrically integral, the two corresponding projective plane curves would have no common component. By B\'ezout's theorem, they would have at most $4$ intersection points counted with multiplicity, contradicting the existence of the five distinct common points. The matrix therefore has rank $5$.

Its six $5\times5$ minors, obtained by deleting one column at a time and taken with alternating signs, form a nonzero vector in its kernel. They therefore give a defining equation of the conic
\[
 ax^2+bxy+cy^2+ex+fy+g=0\text{, where }a,b,c,e,f,g\in\O_K.
\]

For every embedding $\sigma:K\hookrightarrow\mathbb C$, the chosen
points satisfy
\[
 |\sigma(x_i)|\leqslant B^{1/[K:\mathbb Q]},\;|\sigma(y_i)|\leqslant B^{1/[K:\mathbb Q]}\text{ for }
 1\leqslant i\leqslant5.
\]
Since $B\geqslant2$, each entry of the displayed matrix has absolute value
at most $B^{2/[K:\mathbb Q]}$ after applying the above $\sigma$.
Expanding the $5\times5$ determinants therefore gives
\[
 \begin{aligned}
 &\max\bigl\{
 |\sigma(a)|,|\sigma(b)|,|\sigma(c)|,
 |\sigma(e)|,|\sigma(f)|,|\sigma(g)|
 \bigr\}\\
 &\qquad\leqslant
 5!\left(B^{2/[K:\mathbb Q]}\right)^5
 =120B^{10/[K:\mathbb Q]}.
 \end{aligned}
\]

If $a=0$ and $c\ne0$, interchange $x,y$. If $a=c=0$, then $b\ne0$, and the new coordinates $(x,y-x)$ give a nonzero coefficient of the square of the first coordinate. In the latter case the equation becomes $bx^2+bxy+(e+f)x+fy+g=0$ in the new coordinates, before its coefficients are relabelled. Keeping the same letters after this change, we have $a\ne0$ and, for every $\sigma:K\hookrightarrow\mathbb C$, we obtain
\[
\begin{gathered}
 |\sigma(x)|\leqslant B^{1/[K:\mathbb Q]},
\;
 |\sigma(y)|\leqslant2B^{1/[K:\mathbb Q]},\\
 \max\{
 |\sigma(a)|,|\sigma(b)|,|\sigma(c)|,
 |\sigma(e)|,|\sigma(f)|,|\sigma(g)|
 \}
 \leqslant240B^{10/[K:\mathbb Q]}.
\end{gathered}
\]

The coordinate changes are affine isomorphisms defined over $K$, and their extensions identify the respective lines at infinity. Thus, the resulting conic remains geometrically integral and has two distinct geometric points at infinity. Its binary quadratic part therefore has nonzero discriminant $b^2-4ac$.
\end{proof}

The next lemma converts the point-counting problem into a problem about algebraic integers in an extension of $K$ of degree at most $2$.

\begin{lemm}\label{algebraic integers associated with a conic}
We keep the hypotheses of Lemma \ref{bounded equation of a non-tangent conic} and fix the coordinates and plane equation furnished by that lemma. There exist an extension $K'/K$ of degree at most $2$, a nonzero $m\in\O_K$ and an injective map from $s(C';B)$ to nonzero algebraic integers $\alpha\in\O_{K'}$ such that $\alpha\O_{K'}$ divides $m\O_{K'}$. These objects satisfy, for every embedding $\tau:K'\hookrightarrow\mathbb C$,
\[
 |\tau(\alpha)|\leqslant2^{21}B^{21/[K:\mathbb Q]},
 \qquad |\tau(m)|\leqslant2^{40}B^{40/[K:\mathbb Q]},
\]
and
\[
 \left|N_{K'/\mathbb Q}(m)\right|
 \leqslant2^{80[K:\mathbb Q]}B^{80}.
\]
In particular, $K'$ and $m$ are fixed for all points in $s(C';B)$.
\end{lemm}
\begin{proof}
Put $\Delta=b^2-4ac$. The hypothesis on the conic at infinity gives $\Delta\ne0$. Define
\[
\begin{aligned}
 U&=2ax+by+e,\\
 V&=\Delta y+be-2af,\\
 m&=(be-2af)^2-\Delta(e^2-4ag).
\end{aligned}
\]
The conic equation gives
\[
 U^2=\Delta y^2+2(be-2af)y+e^2-4ag,
\]
and hence $V^2-\Delta U^2=m$. Moreover,
\[
 m=-16a\det
 \begin{pmatrix}
 a&b/2&e/2\\
 b/2&c&f/2\\
 e/2&f/2&g
 \end{pmatrix}
 \ne0,
\]
since $a\ne0$ and the projective conic is nonsingular.

Put $K'=K(\sqrt\Delta)$ and $\alpha=V+\sqrt\Delta\,U$. This setting also includes the case where $\Delta$ is a square in $K$, in which case $K'=K$. Since $\Delta\in\O_K$, the element $\sqrt\Delta$ is integral over $\O_K$, being a root of the monic polynomial $T^2-\Delta$. Hence both $\alpha$ and $V-\sqrt\Delta\,U$ belong to $\O_{K'}$, and
\[
 \alpha(V-\sqrt\Delta\,U)=m.
\]
Thus, the principal ideal $\alpha\O_{K'}$ divides $m\O_{K'}$. Furthermore, we have
\[
 V=\frac{\alpha+m/\alpha}{2},
 \qquad
 U=\frac{\alpha-m/\alpha}{2\sqrt\Delta}.
\]
Since $a\Delta\ne0$, we can then recover
\[
 y=\frac{V-be+2af}{\Delta},
 \qquad
 x=\frac{U-by-e}{2a}.
\]
Thus, $\alpha$ determines $(x,y)$, and the map from the original points to $\alpha$ is injective. The coefficients of the plane equation have been fixed once and for all, so $K'$ and $m$ are the same for every point under consideration.

The coefficient and coordinate bounds in Lemma \ref{bounded equation of a non-tangent conic} give, for every embedding $\sigma:K\hookrightarrow\mathbb C$,
\[
\begin{aligned}
 |\sigma(\Delta)|
 &\leqslant5\cdot240^2B^{20/[K:\mathbb Q]},\\
 |\sigma(U)|
 &\leqslant5\cdot240B^{11/[K:\mathbb Q]},\\
 |\sigma(V)|
 &\leqslant13\cdot240^2B^{21/[K:\mathbb Q]}.
\end{aligned}
\]
The terms involving only coefficients also satisfy
\[
\begin{aligned}
 |\sigma(be-2af)|&\leqslant3\cdot240^2B^{20/[K:\mathbb Q]},\\
 |\sigma(e^2-4ag)|&\leqslant5\cdot240^2B^{20/[K:\mathbb Q]}.
\end{aligned}
\]
The definition of $m$ therefore gives the coefficient $3^2+5^2=34$ in its bound below. For every embedding $\tau:K'\hookrightarrow\mathbb C$, with restriction $\sigma$ to $K$, we have $|\tau(\sqrt\Delta)|=|\sigma(\Delta)|^{1/2}$. Consequently,
\begin{equation}\label{upper bound of tau(alpha) and tau(m)}
\begin{aligned}
 |\tau(\alpha)|
 &\leqslant(13+5\sqrt5)\,240^2B^{21/[K:\mathbb Q]}
 \leqslant2^{21}B^{21/[K:\mathbb Q]},\\
 |\tau(m)|
 &\leqslant34\cdot240^4B^{40/[K:\mathbb Q]}
 \leqslant2^{40}B^{40/[K:\mathbb Q]}.
\end{aligned}
\end{equation}
Since $[K':\mathbb Q]\leqslant2[K:\mathbb Q]$, we obtain
\[
 \left|N_{K'/\mathbb Q}(m)\right|
 \leqslant2^{80[K:\mathbb Q]}B^{80}.
\]
\end{proof}

We now bound the number of ideal divisors which can occur in the preceding construction.

\begin{lemm}\label{explicit count of ideal divisors}
Let $K'/K$ be an extension of degree at most $2$, let $B\geqslant2$, and let $m\in\O_{K'}\smallsetminus\{0\}$ satisfy
\[
 \left|N_{K'/\mathbb Q}(m)\right|
 \leqslant2^{80[K:\mathbb Q]}B^{80}.
\]
Then, for every $\eta>0$, we have
\[
\begin{aligned}
 &\#\{\mathfrak a\subseteq\O_{K'}\mid
       \mathfrak a\text{ is an integral ideal and }
       \mathfrak a\mid m\O_{K'}\}\\
 &\quad\leqslant
 \left(1-2^{-\eta}\right)^{-2[K:\mathbb Q]\,2^{1/\eta}}
 \left|N_{K'/\mathbb Q}(m)\right|^\eta\\
 &\quad\leqslant
 \left(1-2^{-\eta}\right)^{-2[K:\mathbb Q]\,2^{1/\eta}}
 2^{80[K:\mathbb Q]\eta}B^{80\eta}.
\end{aligned}
\]
\end{lemm}
\begin{proof}
We count the integral ideal divisors of $m\O_{K'}$. Throughout this proof, $\mathfrak p$ denotes a nonzero prime ideal of $\O_{K'}$, and $N(\mathfrak p)=\#(\O_{K'}/\mathfrak p)$ is its absolute norm. Fix an $\eta>0$. For a given $\mathfrak p$, let $r\geqslant0$ be its exponent in the prime ideal factorization of $m\O_{K'}$. Its contribution to the number of ideal divisors is $r+1$.

If $N(\mathfrak p)\geqslant2^{1/\eta}$, then
\[
 r+1\leqslant2^r\leqslant N(\mathfrak p)^{\eta r}.
\]
If $N(\mathfrak p)<2^{1/\eta}$, then
\[
 (r+1)N(\mathfrak p)^{-\eta r}
 \leqslant
 \sum_{j=0}^{r}N(\mathfrak p)^{-\eta j}\leqslant
 \frac{1}{1-N(\mathfrak p)^{-\eta}}
 \leqslant
 \frac{1}{1-2^{-\eta}}.
\]
Every prime ideal with $N(\mathfrak p)<2^{1/\eta}$ lies above a rational prime smaller than $2^{1/\eta}$. Since at most $[K':\mathbb Q]$ prime ideals of $\O_{K'}$ lie above any rational prime, the number of such prime ideals is at most
\[
 [K':\mathbb Q]\,2^{1/\eta} \leqslant2[K:\mathbb Q]\,2^{1/\eta}.
\]

The product of $N(\mathfrak p)^r$ over these prime ideal factors is $|N_{K'/\mathbb Q}(m)|$. Multiplying the preceding estimates over the prime ideal factors of $m\O_{K'}$, and using the assumed norm bound, we obtain
\[
\begin{aligned}
 &\#\{\mathfrak a\subseteq\O_{K'}\mid 
       \mathfrak a\text{ is an integral ideal and }
       \mathfrak a\mid m\O_{K'}\}\\
 &\quad\leqslant
 \left(1-2^{-\eta}\right)^{-2[K:\mathbb Q]\,2^{1/\eta}}
 \left|N_{K'/\mathbb Q}(m)\right|^\eta\\
 &\quad\leqslant
 \left(1-2^{-\eta}\right)^{-2[K:\mathbb Q]\,2^{1/\eta}}
 2^{80[K:\mathbb Q]\eta}B^{80\eta}.
\end{aligned}
\]
\end{proof}

To count the generators of a fixed principal ideal, we need a uniform separation bound for the logarithms of units.

\begin{lemm}\label{separation of logarithms of units}
Let $K'/K$ be an extension of degree at most $2$. Every unit $u\in\O_{K'}^\times$ which is not a root of unity satisfies
\[ \max_{\tau:K'\hookrightarrow\mathbb C}
 \bigl|\log|\tau(u)|\bigr|
 \geqslant2^{-8[K:\mathbb Q]^2}.
\]
\end{lemm}
\begin{proof}
Since $|N_{K'/\mathbb Q}(u)|=1$, if all its conjugates had absolute value at most $1$, they would all have absolute value $1$. By \cite[Corollary B.2.3.1]{Hindry}, this would imply that $u$ is a root of unity. We may therefore choose an embedding $K'\hookrightarrow\mathbb C$ for which $|u|>1$. We regard $u$ as a complex number through this embedding, and write $\overline u$ for its complex conjugate.

Suppose, contrary to the claimed bound, that
\[
 \bigl|\log|\tau(u)|\bigr|
 <2^{-8[K:\mathbb Q]^2},\quad \forall\;
 \tau:K'\hookrightarrow\mathbb C.
\]
The algebraic integer $u\overline u-1$ is nonzero, since $u\overline u=|u|^2>1$. Moreover, we have
\[
 [\mathbb Q(u\overline u):\mathbb Q]
 \leqslant[\mathbb Q(u,\overline u):\mathbb Q]
 \leqslant[\mathbb Q(u):\mathbb Q]^2
 \leqslant[K':\mathbb Q]^2
 \leqslant4[K:\mathbb Q]^2.
\]

At the chosen embedding, we have
\[ 0<u\overline u-1=|u|^2-1<\exp\left(2\cdot2^{-8[K:\mathbb Q]^2}\right)-1<4\cdot2^{-8[K:\mathbb Q]^2},\]
where the last inequality follows from $2^{-8[K:\mathbb Q]^2}\leqslant1/256$. Every conjugate of $u\overline u$ is a product of two conjugates of $u$: indeed, an embedding of $\mathbb Q(u\overline u)$ into $\mathbb C$ extends to $\mathbb Q(u,\overline u)$, and both images of $u$ and $\overline u$ are conjugates of $u$. Consequently, every conjugate of $u\overline u-1$ has absolute value strictly smaller than $\exp\left(2\cdot2^{-8[K:\mathbb Q]^2}\right)+1<3.$

The norm of $u\overline u-1$ from $\mathbb Q(u\overline u)$ to $\mathbb Q$ is a nonzero integer. Using the smaller bound for the chosen conjugate and the bound $3$ for all the other conjugates, we obtain
\[1\leqslant\left|N_{\mathbb Q(u\overline u)/\mathbb Q}(u\overline u-1) \right|< 4\cdot2^{-8[K:\mathbb Q]^2} 3^{4[K:\mathbb Q]^2-1} = \frac43\left(\frac{81}{256}\right)^{[K:\mathbb Q]^2} \leqslant\frac{27}{64}<1.\]
This contradiction proves the asserted lower bound.
\end{proof}

We use this separation bound to count generators with bounded archimedean absolute values.

\begin{lemm}\label{explicit count of generators}
Let $K'/K$ be an extension of degree at most $2$, and let $B\geqslant2$. The number of generators $\alpha$ of a fixed nonzero principal integral ideal of $\O_{K'}$ satisfying
\[
 |\tau(\alpha)|\leqslant2^{21}B^{21/[K:\mathbb Q]}
 \qquad\text{for every embedding }\tau:K'\hookrightarrow\mathbb C
\]
is at most
\[
 8[K:\mathbb Q]^2
 \left(1+84[K:\mathbb Q]\,2^{8[K:\mathbb Q]^2}\log B\right)^{2[K:\mathbb Q]}.
\]
For every $\epsilon>0$, it is consequently at most
\[
 8[K:\mathbb Q]^2
 \left(\frac{2^{8[K:\mathbb Q]^2+9}[K:\mathbb Q]^2}{\epsilon}\right)^{2[K:\mathbb Q]}
 B^{\epsilon/2}.
\]
\end{lemm}
\begin{proof}
We fix a nonzero principal integral ideal of $\O_{K'}$ and count its generators $\alpha$ satisfying the stated bound. For every embedding $\tau:K'\hookrightarrow\mathbb C$, this bound gives
\[
 \log|\tau(\alpha)|
 \leqslant21\log2+\frac{21}{[K:\mathbb Q]}\log B.
\]
Since $\alpha$ is a nonzero algebraic integer,
\[
 \sum_{\tau:K'\hookrightarrow\mathbb C}\log|\tau(\alpha)|
 =
 \log\left|N_{K'/\mathbb Q}(\alpha)\right|
 \geqslant0.
\]
Applying the upper bound to all embeddings other than a fixed $\tau$
therefore gives
\[
 \log|\tau(\alpha)|
 \geqslant
 -\bigl([K':\mathbb Q]-1\bigr)
 \left(21\log2+\frac{21}{[K:\mathbb Q]}\log B\right).
\]

Thus, the logarithmic vectors $\displaystyle \left(\log|\tau(\alpha)|\right)_{\tau:K'\hookrightarrow\mathbb C}$
lie in a box in $\mathbb R^{[K':\mathbb Q]}$ whose side length is at most
\[
\begin{aligned}
 &[K':\mathbb Q]
 \left(21\log2+\frac{21}{[K:\mathbb Q]}\log B\right)\leqslant
 42[K:\mathbb Q]\log2+42\log B\\
 &\qquad\leqslant
 42\bigl([K:\mathbb Q]+1\bigr)\log B
 \leqslant84[K:\mathbb Q]\log B.
\end{aligned}
\]
Two generators of the same principal ideal differ by a unit. Hence, the difference of their logarithmic vectors is the logarithmic vector of a unit. Lemma \ref{separation of logarithms of units} shows that two distinct logarithmic vectors are separated by at least $2^{-8[K:\mathbb Q]^2}$ in the maximum norm.

If two generators have the same logarithmic vector, their quotient is a unit all of whose conjugates have absolute value $1$. By \cite[Corollary B.2.3.1]{Hindry}, this quotient is a root of unity. The group of roots of unity in the fixed field $K'$ is finite and cyclic. If its order is $j$, then $K'$ contains a primitive $j$-th root of unity, so $\varphi(j)\leqslant[K':\mathbb Q]$, where 
\[
\varphi(j)
=
\#\{a\in\mathbb Z\mid 1\leqslant a\leqslant j,\ \gcd(a,j)=1\}.
\] The elementary inequality
\[
 \frac{\varphi(j)^2}{j}
 =
 \prod_{p^r\parallel j}p^{r-2}(p-1)^2
 \geqslant\frac12
\]
follows because each factor is at least $1$, except for the factor
corresponding to $p=2$, $r=1$, which is $1/2$.
It follows that
\[
 j\leqslant2[K':\mathbb Q]^2\leqslant8[K:\mathbb Q]^2.
\]
Therefore, each logarithmic vector corresponds to at most $8[K:\mathbb Q]^2$ admissible generators.

Starting at the lower corner of the box, cover it by half-open cubes of side $2^{-8[K:\mathbb Q]^2}$. In each coordinate, at most
\[
 1+84[K:\mathbb Q]\,2^{8[K:\mathbb Q]^2}\log B
\]
intervals are needed, including the upper endpoint when it lies on the grid. Each small cube contains at most one distinct logarithmic vector, since two vectors in the same cube would differ by strictly less than $2^{-8[K:\mathbb Q]^2}$ in every coordinate.

Multiplying by the number of possible roots of unity, and using $[K':\mathbb Q]\leqslant2[K:\mathbb Q]$, we conclude that the number of admissible generators of each fixed principal ideal is at most
\[
 8[K:\mathbb Q]^2
 \left(
 1+84[K:\mathbb Q]\,
 2^{8[K:\mathbb Q]^2}\log B
 \right)^{2[K:\mathbb Q]}.
\]

For $B\geqslant2$, we have
\[
 1+84[K:\mathbb Q]\,
 2^{8[K:\mathbb Q]^2}\log B
 \leqslant
 2^{8[K:\mathbb Q]^2+7}[K:\mathbb Q]\log B.
\]
Indeed, $1\leqslant44[K:\mathbb Q]2^{8[K:\mathbb Q]^2}\log B$ under the same hypothesis. Moreover, for $\epsilon>0$, the function $t^{2[K:\mathbb Q]}\exp(-\epsilon t/2)$ on $t\geqslant0$ attains its maximum at $t=4[K:\mathbb Q]/\epsilon$, and
\[
 \max_{t\geqslant0}\{t^{2[K:\mathbb Q]}\exp(-\epsilon t/2)\}
 =\left(\frac{4[K:\mathbb Q]}{\epsilon}\right)^{2[K:\mathbb Q]}
  \exp(-2[K:\mathbb Q]).
\]
Taking $t=\log B$ and bounding the last exponential factor by $1$ gives
\[
 (\log B)^{2[K:\mathbb Q]}
 \leqslant
 \left(\frac{4[K:\mathbb Q]}{\epsilon}\right)^{2[K:\mathbb Q]}
 B^{\epsilon/2}.
\]
Consequently, the number of admissible generators of each principal ideal is at most
\[
 8[K:\mathbb Q]^2
 \left(
 \frac{2^{8[K:\mathbb Q]^2+9}[K:\mathbb Q]^2}{\epsilon}
 \right)^{2[K:\mathbb Q]}
 B^{\epsilon/2}.
\]
\end{proof}

\subsubsection{}
We can now state and prove the uniform estimate for integral points on conics with two distinct geometric points at infinity.

\begin{prop}\label{integral points on non-tangent conics}
Let $\overline{\sE}$ be the Hermitian vector bundle of rank $n+1$ over $\spec\O_K$ defined in \S \ref{O{n+1} with l2-norm}. Let $C'$ be a geometrically integral curve in $\mathbb A_0(\sE_K)$ whose projective closure $C\subset\mathbb P(\sE_K)$ is a conic. Suppose that $C\cap H_\infty$ consists of two distinct geometric points. Then, for every $\epsilon>0$ and $B\geqslant2$, we have
\[
\begin{aligned}
 \#s(C';B)\leqslant \Bigg[4
 &+8[K:\mathbb Q]^2\,2^{\epsilon[K:\mathbb Q]/2}\cdot
 \left(
 \frac{2^{8[K:\mathbb Q]^2+9}[K:\mathbb Q]^2}{\epsilon}
 \right)^{2[K:\mathbb Q]}\\
 &\quad\cdot
 \left(1-2^{-\epsilon/160}\right)^{
 -2[K:\mathbb Q]\,2^{160/\epsilon}}
 \Bigg]B^\epsilon.
\end{aligned}
\]
In particular, the coefficient is explicit and independent of $C$ and its containing plane.
\end{prop}

\begin{proof}
If $\#s(C';B)\leqslant4$, then
\[
 \#s(C';B)\leqslant4\leqslant4B^\epsilon,
\]
so the first summand in the coefficient gives the asserted bound.

Suppose now that $\#s(C';B)\geqslant5$. Lemma \ref{bounded equation of a non-tangent conic} provides the plane equation, and Lemma \ref{algebraic integers associated with a conic} injects the points under consideration into nonzero algebraic integers $\alpha$ in one fixed extension $K'/K$ of degree at most $2$. Their principal ideals divide the same ideal $m\O_{K'}$, where
\[
 \left|N_{K'/\mathbb Q}(m)\right|
 \leqslant2^{80[K:\mathbb Q]}B^{80}.
\]

Taking $\eta=\epsilon/160$ in Lemma \ref{explicit count of ideal divisors}, the number of possible principal ideals is at most
\[
 2^{\epsilon[K:\mathbb Q]/2}
 \left(1-2^{-\epsilon/160}\right)^{
 -2[K:\mathbb Q]\,2^{160/\epsilon}}
 B^{\epsilon/2}.
\]
Here we bound the principal ideal divisors by all integral ideal divisors. By \eqref{upper bound of tau(alpha) and tau(m)}, every admissible $\alpha$ satisfies the archimedean bound required in Lemma \ref{explicit count of generators}. The number of such generators for each principal ideal is therefore at most
\[
 8[K:\mathbb Q]^2
 \left(
 \frac{2^{8[K:\mathbb Q]^2+9}[K:\mathbb Q]^2}{\epsilon}
 \right)^{2[K:\mathbb Q]}
 B^{\epsilon/2}.
\]

Multiplying these two estimates gives the second summand in the stated coefficient, multiplied by $B^\epsilon$. Together with the case of at most four points, this proves the proposition.
\end{proof}

\subsection{The affine conic contribution}
We combine the preceding estimates to control the contribution of a finite collection of conics.
\begin{theo}\label{integral points in finite conic families}
Assume $n=3$. Let $X$ be a geometrically integral non-ruled cubic surface in $\mathbb P(\sE_K)$ such that $X_\infty=X\cap H_\infty$ is geometrically integral. For every finite collection $\mathcal C$ of distinct smooth $K$-conics on $X$, every $B\geqslant2$ and every $\epsilon>0$, we have
\[
\begin{aligned}
 &\sum_{C\in\mathcal C}\#s(C\setminus H_\infty;B)\\
 &\quad\leqslant
 \Bigg[4+8[K:\mathbb Q]^2\,2^{[K:\mathbb Q]\epsilon/2}
 \left(
 \frac{2^{8[K:\mathbb Q]^2+9}[K:\mathbb Q]^2}{\epsilon}
 \right)^{2[K:\mathbb Q]}\\
 &\qquad\qquad\cdot
 \left(1-2^{-\epsilon/160}\right)^{
 -2[K:\mathbb Q]\,2^{160/\epsilon}}
 \Bigg]\#\mathcal C\,B^\epsilon\\
 &\qquad+
 2592\,C_4\left(3,1,2,K,\min\{\epsilon,1/2\}\right)
 B^{1/2+\epsilon/4},
\end{aligned}
\]
where $C_4$ is given by the formula in Corollary \ref{improved version of the control of integral points}. 
\end{theo}

\begin{proof}
By Proposition \ref{finitely many tangent conics}, at most $162$ members of $\mathcal C$ are tangent to $H_\infty$. For each of these conics, its leading Cayley form is nonzero by Proposition \ref{homogeneous deg 2 part non zero}. Applying the height-free bound in Corollary \ref{integral points in affine curve} with $n=3$, $\delta=2$ and $\min\{\epsilon,1/2\}$ in place of $\epsilon$, we obtain
\[
\begin{aligned}
 \#s(C\setminus H_\infty;B)
 &\leqslant
 16\,C_4\left(3,1,2,K,\min\{\epsilon,1/2\}\right)
 B^{1/2+\min\{\epsilon,1/2\}/4}\\
 &\leqslant
 16\,C_4\left(3,1,2,K,\min\{\epsilon,1/2\}\right)
 B^{1/2+\epsilon/4}.
\end{aligned}
\]
Their total contribution is therefore bounded by the second term in the statement, since $162\cdot16=2592$.

Every other member of $\mathcal C$ has two distinct geometric points at infinity. Proposition \ref{integral points on non-tangent conics} bounds its contribution by the explicit coefficient of $\#\mathcal C\,B^\epsilon$ in the statement, multiplied by $B^\epsilon$. There are at most $\#\mathcal C$ such members. Adding these estimates proves the assertion.
\end{proof}

\begin{rema}\label{necessity of square-root term}
Take $K=\mathbb Q$. We consider the cubic surface $X$ in $\mathbb P^3_{\mathbb Q}$ defined by 
\[
T_0T_2^2-T_1^2T_2+T_3(T_0^2+T_2^2+T_3^2)=0.
\]
This surface is smooth over $\overline{\mathbb Q}$ by the Jacobian criterion. In addition, its section at infinity is given by
\[
 -T_1^2T_2+T_2^2T_3+T_3^3=0, 
\]
and this section is also smooth. In particular, $X$ is geometrically integral and non-ruled, and $X_\infty$ is geometrically integral.
\endgraf
The surface contains the smooth conic $C$ defined by
\[
 T_3=0,\qquad T_0T_2=T_1^2,
\]
which is tangent to $H_\infty$ at $[0:0:1:0]$. On the affine chart $T_0=1$, its integral points are exactly $(m,m^2,0)$ with $m\in\mathbb Z$. For $B\geqslant2$, we therefore have
\[
 \#s(C\smallsetminus H_\infty;B)
 =\#\{m\in\mathbb Z\mid m^2+m^4\leqslant B^2\}
 \asymp B^{1/2}.
\]
Indeed, every $|m|\leqslant\sqrt{B/2}$ satisfies the norm bound, whereas the norm bound implies $|m|\leqslant\sqrt B$. Taking the collection $\{C\}$ shows that the exponent $1/2$ in the second term of Theorem \ref{integral points in finite conic families} cannot be replaced by a smaller fixed exponent while retaining the assertion for every $\epsilon>0$.
\end{rema}

We apply the preceding theorem to the conic components of the auxiliary intersection supplied by Theorem \ref{estimate of integral points in conics of cubic}.
\begin{coro}\label{sharper affine conic contribution}
Assume $n=3$ and the hypotheses of Theorem \ref{estimate of integral points in conics of cubic}, with $X'=X\smallsetminus H_\infty\subset\mathbb A_0(\sE_K)$. For $B\geqslant2$ and $0<\epsilon<1$, we take the auxiliary hypersurface in that theorem. Then the number of points of $s(X';B)$ lying on the geometrically integral conic components of this auxiliary intersection is at most
\[
\begin{aligned}
 \Bigg[&
 \frac{3^{9/2}}2\,C_4(3,2,3,K,\epsilon)
 \Bigg(
 4+8[K:\mathbb Q]^2\,2^{[K:\mathbb Q]\epsilon/4}\\
 &\qquad\cdot
 \left(
 \frac{2^{8[K:\mathbb Q]^2+10}[K:\mathbb Q]^2}{\epsilon}
 \right)^{2[K:\mathbb Q]}
 \left(1-2^{-\epsilon/320}\right)^{
 -2[K:\mathbb Q]\,2^{320/\epsilon}}
 \Bigg)\\
 &\quad+
 2592\,C_4(3,1,2,K,\epsilon/2)
 \Bigg]B^{1/\sqrt3+\epsilon}.
\end{aligned}
\]
The coefficient is explicit and independent of $X'$ and $B$.
\end{coro}

\begin{proof}
The auxiliary hypersurface in Theorem \ref{estimate of integral points in conics of cubic} has degree at most
\[
 3^{7/2}C_4(3,2,3,K,\epsilon)
 B^{1/\sqrt3+\epsilon/(6\sqrt3)}.
\]
Let $\mathcal C$ be the collection of projective closures of the geometrically integral conic components of its intersection with $X'$. These are distinct smooth $K$-conics on $X$. By B\'ezout's theorem in the intersection theory, we have
\[
 \#\mathcal C
 \leqslant
 \frac{3^{9/2}}2\,C_4(3,2,3,K,\epsilon)
 B^{1/\sqrt3+\epsilon/(6\sqrt3)}.
\]
The number of points under consideration is at most
$\sum\limits_{C\in\mathcal C}\#s(C\setminus H_\infty;B)$.

Apply Theorem \ref{integral points in finite conic families} with $\epsilon/2$ in place of $\epsilon$. Since $0<\epsilon<1$, the parameter in its second occurrence of $C_4$ is $\epsilon/2$. After substitution of the preceding bound for $\#\mathcal C$, the exponents of $B$ in the two terms are, respectively,
\[
 \frac1{\sqrt3}+\frac{\epsilon}{6\sqrt3}+\frac\epsilon2
 \quad\hbox{and}\quad
 \frac12+\frac\epsilon8.
\]
Both are at most $1/\sqrt3+\epsilon$, since
$1/(6\sqrt3)+1/2<1$ and $1/2<1/\sqrt3$.
As $B\geqslant2$, replacing both powers by
$B^{1/\sqrt3+\epsilon}$ and adding their explicit coefficients gives the stated bound.
\end{proof}

\section*{Declaration}
\subsection*{Conflict of interest statement}
The author declares that there is no potential conflict of interest.
\subsection*{Ethical statement}
The author guarantees that the manuscript has neither been published elsewhere nor submitted for publication in another journal, and will not be submitted to other journals while being reviewed. The author certifies that they have written the entire manuscript and that its main results are original, and cites all necessary sources. The author will inform the editors immediately if any error is found in the manuscript after the submission.

The author used ChatGPT 6 to correct some mistakes in \S \ref{chap.5} in the previous version and improve the written expression. 
\subsection*{The data availability statement}
The data that support the findings of this study are openly available. 
\backmatter

\bibliography{liu}
 \bibliographystyle{smfplain}

\end{document}